\documentclass[11pt]{article}
\usepackage{a4}
\usepackage{amsfonts}
\usepackage{amssymb}
\usepackage{amsmath}
\usepackage{amsthm}
\usepackage{latexsym}
\usepackage{fancybox}
\usepackage{layout}
\usepackage[english]{babel}
\usepackage{epsfig}

\usepackage[usenames]{color}
\def\blue{\color{blue}}
\numberwithin{equation}{section} 
\newtheorem{theorem}{Theorem}[section]
\newtheorem{definition}[theorem]{Definition}
\newtheorem{proposition}[theorem]{Proposition}

\newtheorem{lemma}[theorem]{Lemma}
\newtheorem{remark}[theorem]{Remark}

\def\1{\raisebox{2pt}{\rm{$\chi$}}}

\renewcommand{\H}       {\mathcal{H}}

\newcommand{\nada}[1]   {}

\newcommand{\R}         {\ensuremath{\mathbb R}}

\def\R{I\!\!R}

\def\1{\raisebox{2pt}{\rm{$\chi$}}}

\def\a{{\bf a}}
\def\b{{\bf b}}
\def\z{{\bf z}}

\def\res{\mathbin{\vrule height9pt width.1pt\vrule height.1pt width9pt}}

\renewcommand{\L}{\mathcal{L}}

\begin{document}
\title{On the relativistic heat equation in one space dimension}

\author{{\bf J.A. Carrillo}
        \thanks{Department of Mathematics, Imperial College London, London SW7 2AZ, UK. Email:
                \texttt{carrillo@imperial.ac.uk}}\,,
        {\bf V. Caselles}
        \thanks{Departament de Tecnologia, Universitat Pompeu-Fabra, Barcelona,
        Spain. E-mail: \texttt{vicent.caselles@upf.edu}}\,,
        {\bf S. Moll}
        \thanks{Departament d'An\`alisi Matme\`atica, Universitat de Val\`encia, Valencia, Spain.
        E-mail: \texttt{j.salvador.moll@uv.es}}}

\date{}
\maketitle

\begin{abstract}
We study the relativistic heat equation in one space dimension. We
prove a local regularity result when the initial datum is locally
Lipschitz in its support. We propose a numerical scheme that
captures the known features of the solutions and allows for
analysing further properties of their qualitative behavior.
\end{abstract}




\section{Introduction}\label{sect:intro}

In this work, we explore both analytically and numerically
the implications of a new strategy to study flux-dominated
nonlinear diffusions in one dimension. To be more precise, we consider
the so-called relativistic heat equation (RHE)
\begin{equation}\label{RHEg}
u_t = \nu \, \left( \frac{u u_x}{\sqrt{u^2 +
\frac{\nu^2}{c^2}(u_x)^2}} \right)_x, \qquad x\in\R, \,t > 0.
\end{equation}
introduced by Rosenau in \cite{Rosenau2} and, later on,  by
Brenier in \cite{Brennier1} based on optimal transportation ideas.
The name of RHE comes from the fact that \eqref{RHEg} converges as
$c\to \infty$ to the heat equation both formally and rigorously
\cite{Case}, while the flux in \eqref{RHEg}, understood as a
conservation law, is bounded by the speed of light $c$ whenever
the solution is positive.

Many other models of nonlinear degenerate parabolic equations with
flux saturation as the gradient becomes unbounded have been
proposed by Rosenau and his coworkers \cite{CKR,Rosenau2}, and
Bertsch and Dal Passo \cite{BDP1,DP}. Notice also \cite{MM} for
the presence of flux limited diffusion equations in the context of
radiation hydrodynamics.

The general class of flux limited diffusion equations and the
properties of the relativistic heat equation have been studied in
a series of papers \cite{ACMRelat,ACMRelatE,ACMRelatQ,CER2}. An
existence and uniqueness theory of entropy solutions for the
Cauchy problem associated to the quasi-linear parabolic equation
\begin{equation}
\label{NCproblem} \displaystyle \frac{\partial u}{\partial t} =
{\rm div} \ \b (u, Du),
\end{equation}
was developed in \cite{ACMRelat,ACMRelatE}. Here, the flux
function is given by $\b(z, \xi) = \nabla_{\xi} f(z, \xi)$ and
$f:\R\times\R^N\to\R^+$ is a convex function with linear growth as
$\Vert \xi \Vert \to \infty$, such that $\nabla_{\xi} f(z, \xi)
\in C(\R\times\R^N)$ satisfying other additional technical
assumptions. In particular, the relativistic heat equation
(\ref{RHEg}) satisfies these assumptions, and other models
considered in \cite{Rosenau2}. To avoid the difficulty of the lack
of a-priori estimates that ensure the compactness in time of
solutions of (\ref{NCproblem}), the existence problem was
approached using Crandall-Liggett's theorem
\cite{CrandallLiggett}. For that, we first considered the
associated elliptic problem and we defined a notion of entropy
solution for which we developed a well-posedness theory. The
notion of entropy solution permits to prove a uniqueness result
using Kruzhkov's doubling variables technique
\cite{Kruzhkov,Carrillo1}. This technique was suitably adapted to
work with functions whose truncatures are of bounded variation
\cite{ACMRelat,ACMRelatE}, which is the natural functional setting
for (\ref{NCproblem}) and its associated elliptic equation.

The evolution of the support of solutions of the relativistic heat
equation \eqref{RHEg} was studied in \cite{ACMRelatQ}. By
constructing sub- and super-solutions which are fronts evolving at
speed $c$ and using a comparison principle between entropy
solutions and sub- and super-solutions, it was proved in
\cite{ACMRelatQ} that the support of solutions evolves at speed
$c$. Moreover, the existence of solutions which have discontinuity
fronts moving at the speed $c$ was again shown using the
comparison principle with sub-solutions. This implies, in
particular, that the maximal regularity in time that one can
expect for general solutions of (\ref{RHEg}) is that $u \in
BV([\tau,T]\times \R^N)$ for any $0 < \tau < T$. That this happens
for a general class of initial conditions was proved in
\cite{ACMregularity} and later extended in \cite{CER2}. This lack
of regularity is at the origin of the notion of entropy solutions
for this type of equations. The only regularity result for smooth
initial conditions was proved in \cite{Case} and it guarantees
that $\nabla \ln u$ is bounded whenever initially is. But the
study of the local regularity of solutions of (\ref{RHEg}) is
still an open question. One of the purposes of this paper is to
address this problem for the Cauchy problem associated to
(\ref{RHEg}) in one space dimension with {\it compactly supported
bounded probability densities} as initial data.

Assuming that the initial data in non-negative, we can easily
change variables to observe that $\tilde u(t,x)$ is a solution of
(\ref{RHEg}) if and only if $u(t,x) = \tilde u(\frac{\nu}{c^2}
t,\frac{\nu}{c} x)$ is a solution of
\begin{equation}\label{RHE}
u_t = \, \left( \frac{ u u_x}{\sqrt{u^2 + (u_x)^2}} \right)_x.
\end{equation}
Thus, without loss of generality we may assume that $\nu = c=1$,
and for simplicity we shall assume it in the rest unless
explicitly stated. Notice also that if $u(t)$ is a solution
corresponding to $u_0$, then $\lambda u(t)$ is a solution
corresponding to $\lambda u_0$, $\lambda > 0$. Thus, without loss
of generality we assume that $\Vert u_0\Vert_1 = \Vert u(t)\Vert_1
= 1$ for any $t>0$, and reduce our evolution to probability
densities. In this paragraph, the term solution refers to entropy
solution for which the well-posedness theory was developed and for
which a summary of its concept is reminded to the reader in the
Appendix.

The local regularity of entropy solutions to \eqref{RHE} will be
done by a change variables, writing (\ref{RHE}) in terms of its
inverse distribution function. This change of variables has its
origin in using mass transport techniques to study diffusion
equations \cite{CarrilloToscani04,BCC}. It is known
\cite{Brennier1} that equation \eqref{RHEg} has the structure of a
gradient flow of a certain functional (the physical entropy) with
respect to some transport distance. This structure was already
used to give well-posedness results to \eqref{RHEg} in \cite{MP}.
Nonlinear diffusions have received lots of attention from optimal
transport theory viewpoint starting from the seminal works
\cite{JKO,Otto01}.

Transport distances between probability measures in one dimension
are much easier to compute since they can be written in terms of
distribution functions and their generalized inverses
(pseudo-inverse), the so-called Hoeffding-Fr\'echet Lemma
\cite[Section 2.2]{Villani}. This result led to the following
change of variables based on the distribution function $F$
associated to the probability measure $u$, defined as
$$
F(t,x)=\int_{-\infty}^x u(t,y)\,dy\,.
$$
We formally consider its inverse $\varphi$ defined on the mass
variable $\eta\in (0,1)$ that verifies
$$
F(t,\varphi(t,\eta)) =\eta, \qquad   \eta\in (0,1).
$$
After straightforward computations assuming that all involved
functions are well-defined and smooth, one obtains the equation
\begin{equation}\label{RHEd}
\varphi_t =   \frac{\varphi_{\eta\eta}}{\sqrt{(\varphi_\eta)^4 +
(\varphi_{\eta\eta})^2}}
\end{equation}
for the inverse distribution function $\varphi$. This change of
variables has first been used for nonlinear diffusions in
\cite{CarrilloToscani04} to show contractivity properties of
transport distances for porous-medium like equations. It is worthy
to remark that an implicit Euler discretization of \eqref{RHEd} is
equivalent to the variational JKO scheme whose convergence is
proved in \cite{MP} for \eqref{RHEg} under certain assumptions.
Numerical schemes to solve the equation for the pseudo-inverse
function in the case of the porous medium equation were analysed
in \cite{GT1}. This Lagrangian approach was generalized to several
dimensions in \cite{CarrilloMoll} in order to propose numerical
schemes for equations with gradient flow structure in optimal
transport theory and general quasilinear problems in divergence
form.

In Section \ref{sectheo}, we will first take advantage of this change of
variables to prove the following regularity result:

\begin{theorem}\label{thm:reg1D}
Let $u_0\in  L^\infty(\R)$ with $u_0(x) \geq \kappa > 0$ for $x\in
[a,b]$, and $u_0(x)=0$ for $x\not\in [a,b]$. Assume that $u_0\in
W^{1,\infty}([a,b])$. Let $u(t,x)$ be the entropy solution of
\eqref{RHEg} with $u(0)=u_0$, $\Vert u_0\Vert_1 = 1$ . Then $u\in
C([0,T],L^1(\R^N))$ satisfies:
\begin{itemize}
\item[(i)] $u(t,x)\geq \kappa(t) > 0$ for any $x\in (a-ct,b+ct)$
and any $t > 0$, $u(t,x) = 0$, $x\not \in [a-ct,b+ct]$, $t\in
(0,T)$,

\item[(ii)] $u(t)\in BV(\R)$, $u(t)\in W^{1,1}(a-ct,b+ct)$ for
almost any $t\in (0,T)$, and $u(t)$ is smooth inside its support,

\item[(iii)] if  $u_0\in W^{2,1}(a,b)$, then $u_t$ is a Radon
measure in $(0,T)\times \R$.
\end{itemize}
\end{theorem}

We emphasize that the new parts of this result with respect to the
literature discussed above refer to the regularity stated on
points (ii) and (iii). This result implies that sharp corners on
the support of the initial data are immediately smoothed out by
the evolution of the RHE. This result will be extended in Section
\ref{sect:Extensions}, in particular, we cover the case where the
initial condition $u_0$ vanishes at the boundary of its support.

In Section \ref{sect:numerics}, we will propose an adaption of the numerical scheme
in \cite{CarrilloMoll} based on equation \eqref{RHEd} with
suitable boundary conditions that fully captures the demonstrated
behavior of the solutions of the RHE. Moreover, we will show
different numerical tests in situations where the theory has not
been developed yet. For instance, we numerically study the
conditions for the formation or not of discontinuities on the bulk
of the solutions for RHE and its porous medium counterparts
\begin{equation*}
u_ t = \, \left( \frac{ u^m u_x}{\sqrt{u^2 + (u_x)^2}} \right)_x
\end{equation*}
with $m>1$ and their long-time asymptotic behaviour. Finally, we
include in Appendix A some basic material to describe the notion
of entropy solutions for (\ref{RHE}) for the sake of completeness.


\section{Regularity of Solutions}\label{sectheo}

As proved in \cite{ACMRelat}, there exists a unique entropy
solution of the Cauchy problem for \eqref{RHE} for any $u_0\in
L^1(\R)\cap L^\infty(\R)$, $u_0\geq 0$, see Appendix for the full
notion of solution. Moreover if $u_0$ has compact support in $\R$
and is locally bounded away from zero in any interior point of its
support, then $\mathrm{supp}(u(t)) = \mathrm{supp}(u_0) \oplus
B(0,t)$ \cite{ACMRelatQ}. The rest of this Section is devoted to
the proof of the regularity statements (ii) and (iii).

Let us recall that the entropy condition on the jump set of $u$
can be expressed by saying that the profile of $u$ is vertical at
those points. Since the support of $u(t)$ is $(a-t,b+t)$, and
$u(t)\geq \kappa(t) > 0$ in $(a-t,b+t)$ for any $t>0$ \cite{ACMRelatQ}, there is
a jump at the points $x=a-t,b+t$ and we have \cite{CER2}
\begin{equation}\label{BCmm1}
\frac{u_x}{\sqrt{u^2+(u_x)^2}} (t,a-t)=1, \qquad
\frac{u_x}{\sqrt{u^2+(u_x)^2}} (t,b+t)=-1.
\end{equation}
Let us consider the change of variables discussed in the
introduction and define the function $\varphi(t,\eta)$ by the
relation
\begin{equation}\label{CVinici1}
\int_{a-t}^{\varphi(t,\eta)} u(t,x) dx = \eta, \qquad   \eta\in (0,1).
\end{equation}
Proceeding formally, assuming that the function is smooth inside
its support and differentiating with respect to $\eta$ we obtain
\begin{equation*}
u(t,x) \varphi_\eta = 1, \qquad \hbox{\rm for $x=\varphi(t,\eta)$.}
\end{equation*}
Differentiating with respect to $t$ we have
$$
u(t,x)\varphi_t + u(t,a-t) + \int_{a-t}^{\varphi(t,\eta)} u_t(t,x) dx = 0.
$$
Taking into account the boundary conditions (\ref{BCmm1}) \cite{CER2}, one has
\begin{eqnarray*}
\int_{a-t}^{\varphi(t,\eta)} u_t(t,r) dr & = & \int_{a-t}^{\varphi(t,\eta)} \, \left( \frac{ u u_x}{\sqrt{u^2 +
(u_x)^2}} \right)_x dx =  \frac{ u u_x}{\sqrt{u^2 + (u_x)^2}}(t,x) - u(t,a-t),
\end{eqnarray*}
hence
$$
u(t,x)\varphi_t = - \frac{ u u_x}{\sqrt{u^2 + (u_x)^2}}(t,x)
\qquad \hbox{\rm for $x=\varphi(t,\eta)$\,.}
$$
Then the equation satisfied by $\varphi$ is
\begin{equation*}
\varphi_t =   \frac{\varphi_{\eta\eta}}{\sqrt{(\varphi_\eta)^4 +
(\varphi_{\eta\eta})^2}}\,.
\end{equation*}


\subsection{Regularity result in mass variables}

Now, let us consider the change of variables $v=\varphi_\eta$. The
equation satisfied by $v$ is
\begin{equation}\label{RHEdv}
v_t =  \, \left( \frac{ v_x}{\sqrt{v^4 +
(v_x)^2}} \right)_x  \qquad  t > 0, \, \, x\in (0,1).
\end{equation}
where we have written $x$ instead of $\eta$. This will done
through this subsection for convenience.

The initial condition $v_0$ is determined from the initial
condition $u_0$. We assume that $u_0 \in L^\infty(\R)$, $u_0 \geq
\kappa$, and $u_0\in W^{1,\infty}([a,b])$. Since  the relation
between $u_0$ and $v$ is determined by
$v_0(\eta)=\frac{1}{u_0(x)}$, then $\alpha_1:=\frac{1}{\Vert
u_0\Vert_\infty} \leq v_0\leq\frac{1}{\kappa}:=\alpha_2$. We also
have $v_0\in W^{1,\infty}(0,1)$. Note that
\begin{equation*}
\int_0^1 v_0(\eta)\, d\eta = \int_a^b \, dx = b-a.
\end{equation*}
If we denote by $\nu$ the outer unit normal to $(0,1)$, that is $\nu(0) = -1$ and $\nu(1)=1$,
the natural boundary conditions for (\ref{RHEdv}) are
\begin{equation}\label{verticalBC}
\frac{ v_x}{\sqrt{v^4 + (v_x)^2}} \nu = 1 \qquad \mbox{at }
x\in\partial(0,1)\,,
\end{equation}
with $\partial(0,1)=\{0,1\}$. The first step toward Theorem
\ref{thm:reg1D} is to show a regularity result for the Cauchy
problem \eqref{RHEdv}-\eqref{verticalBC}.

\begin{theorem}\label{thm:regv1D}
Assume that $v_0\in W^{1,\infty}(0,1)$, $v_0\geq \alpha_1 > 0$.
Then there exists a smooth solution of \eqref{RHEdv} in
$(0,T)\times (0,1)$ with $v(0,x)= v_0(x)$ and satisfying the
boundary conditions \eqref{verticalBC} (in a weak sense).
\end{theorem}

\begin{proof}
To prove this claim, we consider the following approximated Cauchy
problem
\begin{align}
&v_t =  \, \left( \frac{ v_x}{ \sqrt{ v^4 + (v_x)^2}} \right)_x
+\epsilon v_{xx} \qquad  t\in (0,T),\, \, x\in (0,1)\,,\label{RHEdvA} \\
&\left(\frac{v_x}{\sqrt{v^4 + (v_x)^2}} +\epsilon v_x\right) \nu =
1 -\epsilon^{1/3}, \qquad  t\in (0,T),\, \, x\in
\partial (0,1),\label{RHEdvAB}
\end{align}
where $\epsilon >0$. The proof is divided in several Steps. In
Steps 1 to 3 we prove some formal estimates that are also useful
to state the existence of solutions of
(\ref{RHEdvA})-\eqref{RHEdvAB} in Step 4. For simplicity we write
$$
\a(z,\xi) = \frac{\xi}{\sqrt{  z^4 +  (\xi)^2}},\qquad z \geq 0,
\,\xi \in\R.
$$
Let us observe that
\begin{equation}\label{fi1}
\a(z,\xi)  \xi \geq  |\xi|- z^2.
\end{equation}

\medskip\noindent {\em Step 1. $L^p$ bounds on $v$ for $p\in [1,\infty)$.}
Let us first consider the evolution of the $L^1$ norm. For that we
integrate (\ref{RHEdvA}) on $(0,1)$. We have
$$
\frac{d}{dt} \int_0^1 v(t,x)\, dx = (\a(v,v_x) + \epsilon v_x)(1)-
(\a(v,v_x) + \epsilon v_x)(0) = 2(1 -\epsilon^{1/3})\,,
$$
and thus,
\begin{equation}\label{mass1}
 \int_0^1 v(t,x)\, dx = \int_0^1 v_0(x)\, dx + 2 (1 -\epsilon^{1/3})
 t\,.
\end{equation}
Given $1\leq p <\infty$, we have
\begin{align*}
\frac{1}{p+1}\frac{d}{dt} \int_0^1 v^{p+1}(t,x)\, dx + \int_0^1
\a(v,v_x) (v^p)_x \,dx+ \epsilon p\int_0^1  v^{p-1}(v_x)^2\,dx &=
{\blue(}1
-\epsilon^{1/3}{\blue )} \int_{\partial (0,1)} v^p\\
&\leq \int_0^1 v^p \, dx +  \int_0^1 |(v^p)_x| \, dx\,,
\end{align*}
where the inequality
$$
v^p(0)+v^p(1) = \int_{\partial (0,1)} v^p \leq \int_0^1 v^p \, dx
+ \int_0^1 |(v^p)_x| \, dx
$$
holds in one dimension. Using (\ref{fi1}) we have
$$
\int_0^1 \a(v,v_x) (v^p)_x \geq  \int_0^1 |(v^p)_x| - p \int_0^1
v^{p+1}\,,
$$
hence
\begin{equation*}
\frac{1}{p+1}\frac{d}{dt} \int_0^1 v^{p+1}(t,x)\, dx + \epsilon p
\int_0^1  v^{p-1}(v_x)^2\,dx \leq  \int_0^1 v^p \, dx +  p
\int_0^1 v^{p+1} \, dx.
\end{equation*}
Using this recurrence relation, by Gronwall's inequality, we obtain that
\begin{equation*}
 \int_0^1 v(t,x)^{p} \, dx \leq C(T,p) \qquad \forall t\in (0,T),\forall p\in [1,\infty),
\end{equation*}
and that
\begin{equation}\label{de1}
\epsilon \int_0^T \int_0^1 v^{p-1} (v_x)^2 \, dx dt \leq C(T,p),
\qquad\forall p\in [1,\infty)\,,
\end{equation}
where the constant $C(T,p)$ does not depend on $\epsilon$.

\medskip\noindent {\em Step 2. $L^\infty$ bounds above and below on $v$ independent of $\epsilon$.}
Let us construct a supersolution to the Cauchy problem
\eqref{RHEdvA}-\eqref{RHEdvAB}. Let $V(t,x)= B(t) -
\sqrt{\epsilon^{2/3} + x(1-x)}$ with $B$ smooth and increasing.
Take $B(0)$ such that
$$
V(0,x)= B(0) - \sqrt{\epsilon^{2/3} + x(1-x)} \geq v_0(x).
$$
We compute
$$
V_t = B'(t),
$$
$$
V_x = \frac{(x-1/2)}{\sqrt{\epsilon^{2/3} + x(1-x)}},\qquad V_{xx} = \frac{\epsilon^{2/3}+1/4}{(\epsilon^{2/3} + x(1-x))^{3/2}},
$$
$$
\a(V,V_x)= \frac{V_x}{(V^4 + V_x^2)^{1/2}} = \frac{(x-1/2)}{D(t,x)},
$$
where $D(t,x)=\left( V(t,x)^4 (\epsilon^{2/3} + x(1-x)) +
(x-1/2)^2\right)^{1/2}$. Note that $D(t,x)$ is a smooth and
strictly positive function in $[0,1]$. Moreover, since $B$ is
increasing, $D\geq
(V(0,x)^4(\varepsilon^\frac{2}{3}+x(1-x))+(x-\frac{1}{2})^2)^\frac{1}{2}$.
Thus $|\a(V,V_x)_x| \leq C$ for a constant $C$ that can be taken
independent of $\epsilon$ and $t\in [0,T]$. Thus, a direct
computation shows that
$$
\a(V,V_x)_x + \epsilon V_{xx} \leq C + \epsilon \frac{\epsilon^{2/3}+1/4}{\epsilon} \leq C + \epsilon^{2/3}+\frac{1}{4} \leq C + \frac{5}{4}= \tilde C,
$$
where $\tilde C$ does not depend on $\epsilon \in (0,1]$. Take
$B'(t)\geq \tilde C$, for instance $B(t) = B(0) + \tilde C t$. Let
us prove that given $T > 0$, for $\epsilon > 0$ small enough
$V(t,x)$ satisfies
\begin{equation*}
(\a(V,V_x) + \epsilon V_x)\nu \geq 1 -\epsilon^{1/3},
\end{equation*}
for $t\in [0,T]$, hence $V(t,x)$ is a supersolution of the Cauchy
problem (\ref{RHEdvA})-\eqref{RHEdvAB} in $[0,T]$. Indeed, since
$$
D(t,0)=\left( (B(t)-\epsilon^{1/3} )^4 \epsilon^{2/3}  + 1/4
\right)^{1/2}\,,
$$
we have at $x=0$
\begin{align*}
(\a(V,V_x) + \epsilon V_x)\nu \vert_{x=0} & = \frac{1/2}{D(t,0)} +
\epsilon \frac{1/2}{\epsilon^{1/3}} = \frac{1}{\left( 1+ 4
(B(t)-\epsilon^{1/3} )^4 \epsilon^{2/3}  \right)^{1/2}} +
\frac{1}{2} \epsilon^{2/3} \geq  1 -\epsilon^{1/3}
\end{align*}
for $\epsilon > 0$ small enough, and analogously at $x=1$. Since
$V(t,x)$ is a supersolution for the Cauchy problem
(\ref{RHEdvA})-\eqref{RHEdvAB}, by the classical comparison
principle we get $v\leq V$ in $[0,T]\times [0,1]$, and thus there
exists $M>0$ depending only on $u_0$ and $T$ such that $v(t,x)\leq
M$ in $(t,x)\in [0,T]\times [0,1]$.

Let us finally observe that $v\geq \alpha_1$. Indeed, $\overline v
=\alpha_1$ is a subsolution for the Cauchy problem
(\ref{RHEdvA})-\eqref{RHEdvAB} and by the comparison principle in
its weak version, we deduce that $v\geq \alpha_1$.

\medskip\noindent {\em Step 3. $L^p$ bounds on $v_x$ independent of $\epsilon$.}
Putting together the estimates in Step 2 and (\ref{de1}), we
deduce that
\begin{equation*}
\int_0^T \int_0^1 |(v^p)_x| \, dx dt \leq C(T,p),
\end{equation*}
for any $p\in [1,\infty)$.

\medskip\noindent {\em Step 4. Existence of smooth solutions for the Cauchy problem \eqref{RHEdvA}-\eqref{RHEdvAB}.}
The existence of solutions of (\ref{RHEdvA})-\eqref{RHEdvAB}
follows from classical results in \cite{LSU} and \cite[Theorem
13.24]{lieberman2005second}. We note that thanks to the a priori
bounds stated above we could  use the flux
$$
\a_M(v,v_x) = \frac{v_x}{\sqrt{ \inf(|v|,M)^4 + v_x^2}}\,,
$$
so that the assumptions of the existence theorems in \cite{LSU}
and \cite[Theorem 13.24]{lieberman2005second} hold. Finally,
observe that we need to assume a compatibility condition on $v_0$
so that $v_0$ satisfies (\ref{RHEdvAB}). If $v_0$ does not satisfy
(\ref{RHEdvAB}), we modify it to define a function $v_{0,\epsilon}
\in W^{1,\infty}(0,1)$ satisfying (\ref{RHEdvAB}). This
modification is only done in a neighborhood of  $x\in \partial
(0,1)$ which vanishes as $\epsilon \to 0+$, so that
$v_{0,\epsilon}$ is locally Lipschitz inside $(0,1)$ with bounds
independent of $\epsilon$. Finally, we observe that this
modification can be done in such a way that
\begin{equation}\label{boundv0eps1}
\sup_{\epsilon\in (0,1]} \epsilon \Vert v_{0\epsilon
x}\Vert_\infty < \infty.
\end{equation}
Although we omit the details of the construction, let us check
that (\ref{boundv0eps1}) is compatible with (\ref{RHEdvAB}). For
that, notice that we can take $v_{0\epsilon x} =
A(\epsilon)\epsilon^{-a}$ with $a=\frac{1}{6}$ and $A(\epsilon) =
\frac{1}{\sqrt{2}} v_{0\epsilon}(0)^2 + O(\epsilon^{1/3})$. Indeed
substituting this expression in (\ref{RHEdvAB}), we have
$$
\frac{A(\epsilon)/\epsilon^a}{\sqrt{v_{0\epsilon}(0)^4 + A(\epsilon)^2/\epsilon^{2a}}} + \epsilon \frac{A(\epsilon)}{\epsilon^a} = 1-\epsilon^{1/3}.
$$
An asymptotic expansion shows $A(\epsilon) = \frac{1}{\sqrt{2}}
v_{0\epsilon}(0)^2 + O(\epsilon^{1/3})$, and thus
(\ref{boundv0eps1}) is compatible with (\ref{RHEdvAB}).

Let $v_\epsilon$ be the solution of the Cauchy problem
(\ref{RHEdvA})-\eqref{RHEdvAB}. Then $v_\epsilon$ has first
derivatives Holder continuous up to the boundary and for
$g=v_{\epsilon xx},v_{\epsilon t}$, we have
$$
\sup_{x\neq y}
\left\{\min(d((x,t),\mathcal{P}),d((y,s),\mathcal{P}))^{1-\delta}
\frac{|g(x)-g(y)|}{(|x-y|^2+|s-t|)^{\alpha/2}}\right\}
$$
for some $\alpha,\delta > 0$, where $\mathcal{P}$ is the parabolic
boundary of $(0,1)\times (0,T)$, that is $[0,1] \times \{0\} \cup
\{0,1\}\times (0,T)$, and $d(\cdot,\mathcal{P})$ denotes the
distance to $\mathcal{P}$. On the other hand, by the interior
regularity theorem \cite[Chapter V, Theorem 3.1]{LSU}, the
solution is infinitely smooth in the interior of the domain. At
this point the smoothness bounds depend on $\epsilon$.

\medskip\noindent {\em Step 5. A local Lipschitz bound on $v_\epsilon$ uniform on $\epsilon$.}
For simplicity of notation, let us write $v$ instead of $v_\epsilon$.
Let $w= |v_x|^2\phi^2$ where $\phi\geq 0$ is smooth with compact support  $[a,b]\subset (0,1)$.
This Step is a consequence of the following inequality
\begin{equation}\label{parabolicG1}
w_t \leq A(t,x) w_{xx} + B(t,x) w_x + C w + f(t,x),
\end{equation}
where $A,B$ are smooth functions, $C= (12 + \frac{\epsilon}{2})$,
and $0\leq f = P(v,\phi,\phi_x)|\phi_x| + \frac{7}{2}\epsilon
\phi_x^2  v_x^2$, where $P$ is a polynomial in $v$ of degree $3$.
Assume for the moment that the last term $ \epsilon \Vert
v_x^2(t)\Vert_\infty \in L^\infty(0,T)$. Using Step 2, this
implies that $f\in L^{\infty}([0,T]\times [0,1])$. Thus we may
replace $f$ by $\Vert f(t)\Vert_{\infty}$. The change of variables
$$
\bar w(t,x) = e^{-Ct} w(t,x) - \int_0^t f(s)\, ds
$$
permits to write (\ref{parabolicG1}) as $\bar w_t \leq A(t,x) \bar
w_{xx} + B(t,x) \bar w_x$. Then, using the maximum principle, this
implies
$$
\sup_{t\in [0,T]}\Vert \bar w(t) \Vert_\infty \leq \Vert \bar w(0)
\Vert_\infty\,
$$
hence we get
$$
\sup_{t\in [0,T]}\Vert w(t)\Vert_\infty\leq C(T,\phi, \Vert w(0)\Vert_\infty).
$$
Let us now prove the claim (\ref{parabolicG1}). We first compute
$$
\a_z(z,\xi) = \frac{-2z^3 \xi}{(z^4+\xi^2)^{3/2}}, \qquad
\a_{zz}(z,\xi) = \frac{-6z^2 \xi}{(z^4+\xi^2)^{3/2}} +  \frac{12z^6 \xi}{(z^4+\xi^2)^{5/2}},
$$
$$
\a_\xi(z,\xi) = \frac{z^4}{(z^4+\xi^2)^{3/2}}, \qquad \a_{\xi
z}(z,\xi) = \frac{-2z^7 + 4z^3 \xi^2 }{(z^4+\xi^2)^{5/2}}, \qquad
\mbox{and}\qquad\a_{\xi \xi}(z,\xi) = \frac{-3z^4 \xi
}{(z^4+\xi^2)^{5/2}}.
$$
We also compute $w_x = 2\phi\phi_x v_x^2 + 2 \phi^2  v_x v_{xx}$
and $ w_{xx} = (2\phi_x^2 + 2 \phi\phi_{xx}) v_x^2 + 8\phi\phi_x
v_x v_{xx} + 2\phi^2 v_{xx}^2 + 2\phi^2 v_x v_{xxx}$.
Differentiating (\ref{RHEdvA}) with respect to $x$ and multiplying
by $\phi^2$ we obtain
$$
\frac{1}{2} w_t =  \a_{zz} v_x^3 \phi^2 + 2 \a_{\xi z} v_x^2 v_{xx} \phi^2
+ \a_{\xi \xi} v_x v_{xx}^2 \phi^2 + \a_{z} v_x v_{xx} \phi^2 + \a_{\xi} v_x v_{xxx} \phi^2 + \epsilon v_x v_{xxx} \phi^2.
$$
Now, we get
$$
\a_{zz} v_x^3 \phi^2 = - \frac{6v^2
v_x^4\phi^2}{(v^4+v_x^2)^{3/2}} +  \frac{12v^6 v_x^4
\phi^2}{(v^4+v_x^2)^{5/2}} \leq 12 w\,,
$$
$$
2 \a_{\xi z} v_x^2 v_{xx} \phi^2  = \a_{\xi z} v_x w_x - 2 \a_{\xi
z} v_x^3 \phi\phi_x \leq \a_{\xi z} v_x w_x + 12 v^3 \phi
|\phi_x|\,,
$$
and
$$
\a_{\xi \xi} v_x v_{xx}^2 \phi^2 = \frac{1}{2} \a_{\xi \xi} v_{xx} w_x -  \a_{\xi \xi} v_{xx} v_x^2 \phi\phi_x =
\frac{1}{2} \a_{\xi \xi} v_{xx} w_x -  X,
$$
where $X= \a_{\xi \xi} v_{xx} v_x^2 \phi\phi_x$. Similarly, we
obtain
$$
\a_{z} v_x v_{xx} \phi^2   = \frac{1}{2} \a_z w_x - \a_z v_x^2 \phi\phi_x \leq \frac{1}{2} \a_z w_x + 2v^3 \phi|\phi_x|,
$$
and
\begin{align*}
\a_{\xi} v_x v_{xxx} \phi^2 &= \frac{1}{2} \a_\xi w_{xx} -\a_\xi
(\phi_x^2 + \phi\phi_{xx}) v_x^2 - 4 \a_\xi v_{xx} v_{x}\phi\phi_x
- \a_\xi v_{xx}^2 \phi^2
\\
&\leq  \frac{1}{2} \a_\xi w_{xx} + v^2(\phi_x^2 + \phi|\phi_{xx}|)
- Y - \a_\xi v_{xx}^2 \phi^2,
\end{align*}
where $Y= 4 \a_\xi v_{xx} v_{x}\phi\phi_x$. Direct estimates show
that
$$
|Y| \leq \frac{1}{2} \a_\xi v_{xx}^2 \phi^2 + 8 \a_\xi
v_x^2\phi_x^2 \leq \frac{1}{2} \a_\xi v_{xx}^2 \phi^2  + 8 v^2
\phi_x^2
$$
and
$$
|X| \leq \frac{1}{2} \a_\xi v_{xx}^2 \phi^2 +
\frac{\a_{\xi\xi}^2}{2\a_\xi } v_x^4 \phi_x^2 \leq \frac{1}{2}
\a_\xi v_{xx}^2 \phi^2 +  \frac{9}{2} v^2 \phi_x^2.
$$
Finally, let us compute the term
\begin{align*}
v_x v_{xxx} \phi^2 & = \frac{1}{2} w_{xx} - (\phi_x^2 +
\phi\phi_x) v_x^2 - 4\phi\phi_x v_x v_{xx} - \phi^2 v_{xx}^2
\\
& \leq \frac{1}{2} w_{xx} - \phi_x^2 v_x^2 + \frac{1}{2} \phi^2
v_x^2  + \frac{1}{2} \phi_x^2  v_x^2
+ 4 \phi_x^2  v_x^2  + \phi^2 v_{xx}^2 - \phi^2 v_{xx}^2 \\
& = \frac{1}{2} w_{xx} + \frac{1}{2} w  + \frac{7}{2} \phi_x^2
v_x^2.
\end{align*}
Putting all together, we get the desired claim \eqref{parabolicG1}
\begin{equation}\label{parabolicG3}
\frac{1}{2} w_t \leq \frac{1}{2} (\a_\xi +\epsilon) w_{xx} +
\left( \a_{\xi z} v_x  +  \frac{1}{2} \a_{\xi \xi} v_{xx}  +
\frac{1}{2} \a_z \right) w_x + \left(12 +
\frac{\epsilon}{2}\right) w + P(v,\phi,\phi_x)|\phi_x| +
\frac{7}{2}\epsilon \phi_x^2 v_x^2,
\end{equation}
where $P$ is a polynomial of degree $3$ in $v$.

Now, we have to show that $\epsilon \|v_x^2(t)\|_\infty \in
L^\infty(0,T)$. Let us first exploit the boundary condition in
(\ref{RHEdvAB}). Multiplying it by $v_x$ and using \eqref{fi1}, we
get
$$
|v_x|-v^2\leq \a(v,v_x)v_x=\frac{|v_x|^2}{(v^4+v_x^2)^{1/2}} +
\epsilon v_{x}^2 = (1 -\epsilon^{1/3}) v_x\,,
$$
and thus we get that $\epsilon v_{x}^2 \leq v^2$ on $\partial
(0,1)$. Moreover, using Step 2 we finally deduce that
\begin{equation}\label{linftyboundary}
\epsilon v_{x}^2(t) \leq \sup_{t\in [0,T]}(|v((t,0)|,|v(t,1)|)\leq
M\,, \qquad \mbox{on } \partial (0,1)\,.
\end{equation}
Taking $\phi=1$ in (\ref{parabolicG3}), we obtain
\begin{equation*}\label{parabolicG4}
\frac{1}{2} w_t \leq \frac{1}{2} (\a_\xi +\epsilon) w_{xx} + \left( \a_{\xi z} v_x  +  \frac{1}{2} \a_{\xi \xi} v_{xx}  +
\frac{1}{2} \a_z \right) w_x + (12 + \frac{\epsilon}{2}) w,
\end{equation*}
that together with \eqref{linftyboundary} and the maximum
principle, imply that
\begin{equation}\label{eBf1}
\epsilon \Vert v_x(t)^2\Vert_\infty \leq C,
\end{equation}
for some constant $C$ that depends on the bound
(\ref{boundv0eps1}), and thus independent of $\epsilon$.

Summarizing, now the term $\frac{7}{2}\epsilon \Vert \phi_x^2
v_x^2(t) \Vert_\infty \in L^\infty (0,T)$ with bounds independent
of $\epsilon$.  Again, Step 2 implies that $\Vert f(t)
\Vert_\infty \leq \Vert P(v(t),\phi,\phi_x)|\phi_x| \Vert_\infty +
\Vert \frac{7}{2}\epsilon \phi_x^2  v_x^2(t) \Vert_\infty \in
L^\infty(0,T)$ with bounds independent of $\epsilon$. Then the
argument given above shows that there are local Lipschitz bounds
on $v_\epsilon$ uniform in $\epsilon$.

\medskip\noindent {\em Step 6. Interior regularity of higher order derivatives uniform in $\epsilon$.}
Thanks to the smoothness results stated in Step 4 and the local
uniform bounds on the gradient in Step 5, the classical interior
regularity results in \cite[Chapter V, Theorem 3.1]{LSU} shows
uniform (in $\epsilon$) interior bounds for any space and time
derivative of $v_\epsilon$.

\medskip\noindent {\em Step 7. Passing to the limit as $\epsilon \to 0^+$.}
Letting $\epsilon \to 0^+$ is not completely obvious due to the
boundary condition (\ref{verticalBC}). Another difficulty stems
from the fact that we do not know if $v_{\epsilon t}$ are Radon
measures with uniform bounds in $\epsilon$. This means that the
notion of normal boundary trace has to be considered in a weak
sense as considered in \cite{ACM4:01} (see also \cite[Section
5.6]{ACMBook} or \cite{ACMMdirichlet}). Thus, we only sketch the
proof of this result. Let us first prove that the interior
regularity bounds on $v_\epsilon$ permit to pass to the limit and
obtain a solution $v$ of
\begin{equation*}
v_t =  \, \left( \frac{ v_x}{\sqrt{v^4 + (v_x)^2}} \right)_x
\qquad  \hbox{\rm in $\mathcal{D}^\prime((0,T)\times (0,1))$}\,.
\end{equation*}
Let
$$
\xi^\epsilon := v_{\epsilon t} =  \left( \frac{ v_{\epsilon
x}}{\sqrt{v_{\epsilon}^4 + (v_{\epsilon x})^2}} + \epsilon
v_{\epsilon x}\right)_x \qquad \mbox{and} \qquad \a_\epsilon =
\frac{ v_{\epsilon x}}{\sqrt{v_{\epsilon}^4 + (v_{\epsilon x})^2}}
+ \epsilon v_{\epsilon x}\,.
$$
Estimate (\ref{eBf1}) implies that $\a_\epsilon$ are uniformly
bounded independently of $\epsilon$. Then by extracting a
subsequence, we may assume that $\a_\epsilon \rightharpoonup \a
\in L^{\infty}((0,T)\times (0,1))$ weakly$^*$. On the other hand,
the interior regularity bounds on $v_\epsilon$ ensure that $\a=
\frac{ v_\eta}{\sqrt{v^4 + (v_x)^2}}$. By passing to the limit as
$\epsilon\to 0$, we have $v_t = \a_x$ in ${\cal
D}^{\prime}((0,T)\times (0,1))$. Finally, if we take $\varphi \in
C^1([0,T]\times [0,1])$ with $\varphi(0)=\varphi(T)=0$, multiply
(\ref{RHEdvA}) by $\varphi$ and integrate by parts, we obtain
$$
\int_0^T\int_0^1 v_\epsilon \varphi_t \, dx dt = \int_0^T\int_0^1 \a_\epsilon \varphi_x\, dx dt -
2(1-\epsilon^{1/3}) T.
$$
Letting $\epsilon\to 0^+$, we obtain
$$
\int_0^T\int_0^1 v  \varphi_t \, dx dt = \int_0^T\int_0^1 \a  \varphi_x\, dx dt -
2 T.
$$
This is a weak form of the boundary condition (\ref{verticalBC}).
The correct notion of weak trace is much more technical and is
described in \cite{ACMBook}. Using Lemma 5.7 in
\cite{ACMMdirichlet} one can directly obtain that $v$ satisfies
(\ref{verticalBC}) in this generalized sense. Since we do not need
this result here, we skip the details that would need several
technical definitions to be fully explained.
\end{proof}

\begin{remark}\label{rk:LB1}{\rm
Note that we can apply Step 5 to the smooth solution obtained in
Theorem \ref{thm:regv1D} to the Cauchy problem
\eqref{RHEdv}-\eqref{verticalBC}. In this case $\Vert
f\Vert_\infty \leq \Vert P(v,\phi,\phi_x)|\phi_x|\Vert_\infty$ and
we obtain a local Lipschitz bound for $v(t,x)$ which only depends
on local uniform bounds of $v(t,x)$ and on the local Lipschitz
bound of $v_0(x)$. }
\end{remark}

\begin{remark}\label{rk:LB2}{\rm
In Section \ref{sect:ESRHE} we will give sufficient conditions on
$u_0$ that imply that $v_t$ is a {R}adon measure. In that case,
the notion of weak trace $\a\cdot \nu$ can be found in
\cite{CER2,ChenFrid1}. }
\end{remark}

\begin{remark}{\rm
We could define the notion on entropy solutions of equation (\ref{RHEdv}) with boundary condition (\ref{verticalBC})
and prove that the solution constructed is indeed an entropy solution of it. We will not pursue this here.
}
\end{remark}

\subsection{Getting an entropy solution of \eqref{RHE} from \eqref{RHEdv}}\label{sect:ESRHE}

Here, we use several notations and definitions that are introduced
in the Appendix to which we refer for details. In this Section, we
come back to the notation $v(t,\eta)$ instead of $v(t,x)$,
$\eta\in (0,1)$. Recall that by passing to the limit as
$\epsilon\to 0^+$ we have found a solution $v$ of
\begin{equation}\label{RHEdvB}
v_t =  \, \left( \frac{ v_\eta}{\sqrt{v^4 +
(v_\eta)^2}} \right)_\eta  \qquad  \hbox{\rm in $\mathcal{D}^\prime((0,T)\times (0,1))$},
\end{equation}
for any $T > 0$. Thus, let $v(t,\eta)$ be the solution of
(\ref{RHEdvB}) constructed in Theorem \ref{thm:regv1D} which
satisfies $[\a(t,\eta)\cdot\nu]=1$ for $\eta=0,1$ and a.e. for
$t\in (0,T)$ in a weak sense. As we shall see, we do not need this
here, we only need a weaker form of the boundary condition as
expressed in (\ref{massBC}) below.

In the next Lemma we construct an entropy solution of (\ref{RHEg}) from a solution
$v(t,\eta)$ of (\ref{RHEdvB}). To prepare its statement, let
$u_0\in  L^\infty(\R)$ with $u_0(x) \geq \kappa > 0$ for $x\in [a,b]$, and $u_0(x)=0$
for $x\not\in [a,b]$. Assume that $u_0\in W^{1,\infty}([a,b])$.
Let $v_0(\eta) = \frac{1}{u_0(x)}$, $\eta\in (0,1)$, where $x = \varphi(0,\eta)$ is such that
$$
\int_a^{\varphi(0,\eta)} u_0(x) \, dx = \eta\,.
$$
Let $u(t,x)$ be defined in $[a-t,b+t]$ by
\begin{equation}\label{e:cv1}
u(t,x) = \frac{1}{v(t,\eta)}\,,\,\, \mbox{ where } x =
\varphi(t,\eta) = a - t + \int_0^\eta v(t,\bar \eta)d\bar \eta\,.
\end{equation}
By (\ref{mass1}), we have
\begin{equation}\label{massBC}
\int_0^1 v(t,\eta)\, d\eta = b-a + 2t\,,
\end{equation}
and $x=\varphi(t,\eta)\in [a-t,b+t]$ when $\eta$ varies in $[0,1]$.
Note that
$$
\int_{a-t}^{\varphi(t,\eta)} u(t,x) dx = \eta, \qquad \eta\in (0,1).
$$
We define $u(t,x) = 0$, $x\not \in [a-t,b+t]$, $t\in (0,T)$. Notice that $u(t,x)\geq \kappa(t) > 0$ for any
$x\in (a-t,b+t)$ and any $t > 0$.

The statement $(ii)$ in Theorem \ref{thm:reg1D} follows from next Proposition.

\begin{proposition}\label{prop:suficient}
Given $u$ defined by \eqref{e:cv1} where $v$ is a solution given
by Theorem {\rm\ref{thm:regv1D}}. Then $u\in C([0,T],L^1(\R))$,
$u(0)=u_0$, and satisfies
\begin{itemize}
\item[(i)] $u(t)\in BV(\R)$, $u(t)\in W^{1,1}(a-t,b+t)$ for almost
any $t\in (0,T)$, and $u(t)$ is smooth inside its support,

\item[(ii)] $u_t = \z_x$ in $\mathcal{D}^\prime((0,T)\times \R)$,
where $\z(t) = \frac{u(t) u_x(t)}{\sqrt{u(t)^2+ u_x(t)^2}}$\,,

\item[(iii)] $u(t,x)$ is the entropy solution of \eqref{RHE} with
initial data $u_0$ in $(0,T)$.
\end{itemize}
\end{proposition}

\begin{proof}
$(i)$ Since $v$ is bounded and bounded away from zero from Step 2
in Theorem \ref{thm:regv1D}, then $u$ is bounded and bounded away
from zero in its support. The smoothness properties of $v$ prove
that $u\in C([0,T],L^1(\R))$, $u(0)=u_0$, and $u(t)$ is smooth
inside its support. By Step 3 from Theorem \ref{thm:regv1D}, we
have that $u(t)\in W^{1,1}(a-t,b+t)$ for almost any $t\in (0,T)$.
This implies that $u(t)\in BV(\R)$ for almost any $t\in (0,T)$.
From the change of variables (\ref{e:cv1}) we have that
\begin{equation}\label{CV1}
\frac{u_x}{\sqrt{u^2+u_x^2}} = - \frac{v_\eta}{\sqrt{v^4+v_\eta^2}}.
\end{equation}

\smallskip\noindent $(ii)$ For simplicity, let us write
$Q_T = (0,T)\times \R$, and $\Omega(t)= (a-t,b+t)$.
Since
$$
Du(t) =  u_x\1_{\Omega(t)} - u^i(t)
\nu^t\mathcal{H}^{0}\res{\partial \Omega(t)},
$$
we have that $u\in L^1_{loc, w}(0, T; BV(\R))$. We have denoted by
$u^i(t)$ the trace of $u\vert_{\Omega(t)}$ on $\partial
\Omega(t)$. Note that it coincides with $u^+(t)$. Let us prove
that
\begin{equation}\label{divzESAA} u_t =  \z_x  \quad \hbox{in} \ \
{\mathcal D}^{\prime}((0,T)\times \R).
\end{equation}
Let $\phi\in {\mathcal D}(Q_T)$. Let $\overline{\phi}(t,\eta) =
\phi(t,\varphi(t,\eta))$, $\eta \in [0,1]$. Then
$\overline{\phi}_t = \phi_t(t,\varphi(t,\eta)) +
\phi_x(t,\varphi(t,\eta)) \varphi_t$ and
\begin{align*}
- \int_0^T \int_{\R} u \phi_t \, dx dt &= - \int_0^T
\int_{\Omega(t)} u \phi_t \, dx dt = - \int_0^T \int_0^1
\frac{1}{v} (\overline{\phi}_t -  \phi_x(t,\varphi(t,\eta))
\varphi_t) v\, d\eta dt
\\
&=- \int_0^T \int_0^1  (\overline{\phi}_t -
\phi_x(t,\varphi(t,\eta)) \varphi_t) \, d\eta dt =
 \int_0^T \int_0^1   \phi_x(t,\varphi(t,\eta)) \varphi_t  \, d\eta dt
\\
&= \int_0^T \int_0^1   \phi_x(t,\varphi(t,\eta))
\frac{v_\eta}{\sqrt{v^4+v_\eta^2}}  \, d\eta dt = - \int_0^T
\int_{\Omega(t)} \frac{uu_x}{\sqrt{u^2+u_x^2}} \phi_x(t,x) \, dx
dt
\\
&= - \int_0^T \int_{\R} \z \phi_x  \, dx dt\,,
\end{align*}
where \eqref{CV1} was used. Thus (\ref{divzESAA}) holds.

\medskip\noindent $(iii)$
To prove that $u$ is an entropy solution of (\ref{RHE}),
we have to prove that
\begin{align}
\int_{Q_T}h_{S}(u,DT(u))\phi \,dx dt & +\int_{Q_T}  h_T(u,
DS(u))\phi \,dx dt \nonumber \\ & \leq \int_{Q_T}J_{TS}(u)\phi_t
\,dx dt-\int_0^T\int_{\R} \z(t,x)\cdot\nabla\phi(t) T(u(t))S(u(t))
\,dx dt,\label{subrsolucioI}
\end{align}
holds for any any $T, S \in {\mathcal T}^+$ and any $\phi\in
\mathcal D((0, T) \times\R)$, $\phi(t,x) = \eta(t)\rho(x)$. As in
\cite[Proposition 1]{ACMRelatQ}, we have
\begin{equation}\label{0jumpRH}
(h_S(u(t),DT(u(t))))^s =  \left\vert D^jJ_{SR T'}(u(t))
 \right\vert = J_{S R T'}(u^i(t)) {\mathcal H}^{0}\res{\partial \Omega(t)}
\end{equation}
and \begin{equation}\label{0jumpRHI}(h_T(u(t),DS(u(t))))^s =
\left\vert D^j J_{TR S'}(u(t)) \right\vert = J_{T R
S'}(u^i(t)) {\mathcal H}^{0}\res{\partial \Omega(t)},
\end{equation}
where $R(r)=r$, $r\in\R$. Thus, by (\ref{0jumpRH}) and
(\ref{0jumpRHI}), we get
\begin{align}
(h_S(u(t),DT(u(t))))^s + (h_T(u(t),DS(u(t))))^s & = \left(J_{S R
T'}(u^i(t))+ J_{T R S'}(u^i(t)) \right) {\mathcal
H}^{0}\res{\partial \Omega(t)} \nonumber \\ & = \left(TS R(u^i(t))
- J_{TS}(u^i(t))\right){\mathcal H}^{0}\res{\partial
\Omega(t)}.\label{A3E1ESol}
\end{align}
On the other hand, it is easy to prove that
\begin{align}
\int_{Q_T}(h_{S}(u,DT(u)))^{ac}\phi \, dx dt \,+ &
\int_{Q_T}(h_T(u, DS(u)))^{ac}\phi \, dx dt \nonumber
\\ &= \displaystyle\int_0^T\int_{\Omega(t)} \z(t,x)\cdot [T(u(t,x
))S(u(t,x))]_x \phi(t) \, dx dt.\label{acE1ESol}
\end{align}
Adding (\ref{A3E1ESol}) and (\ref{acE1ESol}), we obtain
\begin{align}
\int_{Q_T} \phi h_S(u(t),DT(u(t))) \,dx dt  \, + &\,\int_{Q_T}
\phi h_T(u(t),DS(u(t))) \,dx dt \nonumber\\ = &\, \int_0^T
\int_{\partial \Omega(t)} \left(TSR (u^i(t)) -
J_{TS}(u^i(t))\right) \phi(t) \, d{\mathcal H}^{0}\, dt
\nonumber\\ &\,+ \displaystyle\int_0^T\int_{\Omega(t)}
\z(t,x)\cdot  [T(u(t,x ))S(u(t,x))]_x \phi(t) \, dx dt.
\label{TE1ESol}
\end{align}

To simplify the subsequent notation let us denote $p(u) =
T(u)S(u)=J'(u)$ and $J(u)= J_{TS}(u)$. Let us now prove that
\begin{align}
\int_0^T \int_{\partial \Omega(t)} \left(p(u^i(t))u^i(t) -
J(u^i(t))\right) \phi(t) \, d{\mathcal H}^{0}\, dt \, +&\,
\int_0^T\int_{\Omega(t)} \z \cdot  [p(u)]_x \phi \, dx dt \nonumber \\
&\leq \displaystyle \int_{Q_T}J(u)\phi_t \, dxdt -
\int_0^T\int_{\Omega(t)}\phi_x \z  p(u) \, dx dt.
\label{e:ineqju1}
\end{align}

The main technical difficulty comes from the fact that we do not know that $u_t=\z_x$ is a Radon measure.
We circumvent this difficulty by using instead discrete derivatives.
Let us denote
$$
\Delta_\tau^+ w(t) = \frac{1}{\tau} (w(t+\tau) - w(t)), \qquad \Delta_\tau^- w(t) = \frac{1}{\tau} (w(t) - w(t-\tau)).
$$
Then, we can obtain
\begin{align*}
-\int_0^T \int_{\R}  up(u)\phi \,& \Delta_\tau^- \1_{\Omega(t)}\,
dx dt =  \int_0^T \int_{\Omega(t)} \Delta_\tau^+(up(u)\phi) \, dx
dt
\\ = & \int_0^T \int_{\Omega(t)} \Delta_\tau^+ u (t)
p(u(t+\tau))\phi(t+\tau) \, dx dt  +
\int_0^T \int_{\Omega(t)} u(t)   \Delta_\tau^+ (p(u)\phi) (t) \, dx dt \\
 \geq & \int_0^T \int_{\Omega(t)} \Delta_\tau^+ J(u) (t) \phi(t+\tau) \, dx dt  +
\int_0^T \int_{\Omega(t)} u(t)   \Delta_\tau^+ (p(u)\phi) (t) \, dx dt \\
 = & - \int_0^T \int_{\Omega(t)} J(u) (t) \Delta_\tau^- [\phi(t+\tau)] \, dx dt
-\int_0^T \int_{\R} J(u) (t) \phi(t) \Delta_\tau^- \1_{\Omega(t)}\, dx dt  \\
& + \int_0^T \int_{\Omega(t)} u(t)   \Delta_\tau^+ (p(u)\phi) (t)
\, dx dt
\end{align*}
which is a discrete version of (\ref{e:ineqju1}). Note that we
have used the inequality $\Delta_\tau^+ u (t)   p(u(t+\tau)) \geq
\Delta_\tau^+ J(u)$ which is a consequence of the convexity of
$J$. By letting $\tau \to 0^+$, we need to show that
\begin{equation}\label{e:tau1to0}
\int_{Q_T}  (u(t)p(u(t)) - J(u(t))) \phi(t) \Delta_\tau^-
\1_{\Omega(t)} \, dx dt \to \int_0^T \int_{\partial \Omega(t)}
\left(p(u^i(t))u^i(t) - J(u^i(t))\right) \phi(t) \, d{\mathcal
H}^{0}\, dt,
\end{equation}
\begin{equation}\label{e:tau2to0}
\int_0^T \int_{\Omega(t)} J(u) (t) \Delta_\tau^- [\phi(t+\tau)] \, dx dt
\to  \int_0^T \int_{\Omega(t)} J(u) (t) \phi_t(t) \, dx dt,
\end{equation}
and
\begin{equation}\label{e:tau3to0}
\int_0^T \int_{\Omega(t)} u(t)   \Delta_\tau^+ (p(u)\phi) (t) \,
dx dt \to  \int_0^T \int_{\Omega(t)}     (p(u)\phi)_x   \z \, dx
dt.
\end{equation}
This will result in (\ref{e:ineqju1}). The limit (\ref{e:tau1to0})
follows since $u(t)\in BV(\R)$ a.e. in $t$, $u\in
L^1_{w}(0,T;BV(\R))$ (hence $\Vert u_x(t)\Vert \in L^1(0,T)$) and
the trace functions $u(t,a-t)$, $u(t,b+t)$ are integrable in
$[0,T]$. The second limit (\ref{e:tau2to0}) follows easily. To
prove (\ref{e:tau3to0}), for any $\tau > 0$  let
$$
\psi^{\tau}(t,x):= \frac{1}{\tau} \int_{t}^{t+\tau} \phi (s,x) p(u(s,x)) \, ds,
$$
and observe that
$$
\Delta_\tau^+ (p(u)\phi) (t,x) = \frac{\partial}{\partial t} \psi^{\tau}(t,x).
$$
Observe also that
$$
\frac{d}{dt} [\psi^{\tau}(t,\varphi(t,\eta))] = \frac{\partial}{\partial t} \psi^{\tau}(t,\varphi(t,\eta))  +
\psi^{\tau}_x(t,\varphi(t,\eta))\varphi_t(t,\eta).
$$
Then, as $\tau\to 0^+$
\begin{align*}
\int_0^T \int_{\Omega(t)} u(t)   \Delta_\tau^+ (p(u)\phi) (t) \,
dx dt  = &   \int_0^T \int_{\Omega(t)} u(t)
\frac{\partial}{\partial t}  \psi^{\tau}(t,x) \, dx dt
\\ = & \int_0^T \int_0^1 \left( \frac{d}{dt} [\psi^{\tau}(t,\varphi(t,\eta))] -  \psi^{\tau}_x(t,\varphi(t,\eta))\varphi_t(t,\eta) \right) \, d\eta dt \\
= & - \int_0^T \int_0^1   \psi^{\tau}_x(t,\varphi(t,\eta)) \frac{v_\eta}{\sqrt{v^4 + v_\eta^2}}   \, d\eta dt
\\ = & \int_0^T \int_0^1   \psi^{\tau}_x(t,x) \frac{uu_x}{\sqrt{u^2 + u_x^2}}   \, dx dt \to \int_0^T \int_{\Omega(t)}     (p(u)\phi)_x   \z \, dx dt.
\end{align*}
We have proved (\ref{e:tau3to0}). Finally we observe that from
(\ref{TE1ESol}) and (\ref{e:ineqju1}) we obtain
(\ref{subrsolucioI}).
\end{proof}

\begin{remark}{\rm
In a similar way, using this time
$$
\Delta_\tau^+(up(u)\phi)(t)= \Delta_\tau^+(p(u)\phi)(t)u(t+\tau) + p(u(t))\phi(t)
\Delta_\tau^+(u)(t)
$$
and $\Delta_\tau^+(J(u))(t)\geq p(u(t))\Delta_\tau^+(u)(t)$ one can prove that the opposite inequality in
(\ref{e:ineqju1}) holds, and we have equality. Note also that equality holds also in the entropy conditions
(\ref{subrsolucioI}).
}
\end{remark}

With some additional regularity on the initial condition, one has
that $u_t$ is a Radon measure in $(0,T)\times \R$. Indeed, the
following proposition follows immediately from the results in
\cite{ACMSV,CER2}.

\begin{proposition}\label{Rm1}
Let $u_0\in  L^\infty(\R)$, $u_0(x) \geq \kappa > 0$ for $x\in
[a,b]$ and $u_0=0$ outside $[a,b]$. Assume that $u_0\in
W^{2,1}(a,b)$.  If $u$ is the entropy solution of \eqref{RHE} with
initial data $u_0$, then $u_t$ is a Radon measure in $(0,T)\times
\R$.
\end{proposition}

From Proposition \ref{Rm1} and the results in \cite{CER2}, it
follows that $[\z\cdot\nu^{\Omega(t)}]=-u^i(t)$ on
$\partial\Omega(t)$ for almost any $t\in (0,T)$. This permits also
to define the notion of normal trace of $\a(v,v_\eta)$ in the
sense of \cite{CER2,ChenFrid1}.


\section{Regularity for touching-down initial data}\label{sect:Extensions}

Let us start by getting local estimates.

\begin{proposition}\label{prop:locL1}
Let $u_0\in  L^\infty(\R)$ with $u_0(x) \geq \kappa > 0$ for $x\in
[a,b]$, and $u_0(x)=0$ for $x\not\in [a,b]$. Assume that $u_0\in
W^{1,\infty}_{\rm loc}(a,b)$. The entropy solution $u(t,x)$ of
\eqref{RHE} with $u(0)=u_0$ satisfies $(i)$ and $(ii)$ in Theorem
{\rm\ref{thm:reg1D}}.
\end{proposition}

\begin{proof}
Let $u_{0\delta} \in L^\infty(\R)$ with $u_{0\delta}(x) \geq
\kappa > 0$ for $x\in [a,b]$, $u_{0\delta}(x)=0$ for $x\not\in
[a,b]$, $u_{0\delta} \to u_0$ locally uniformly in $(a,b)$ as
$\delta\to 0^+$, and $u_{0\delta}\in W^{1,\infty}([a,b])$ with
uniform local Lipschitz bounds in $(a,b)$. Let $v_{0\delta}(\eta)$
be the functions obtained by the change of variables
(\ref{CVinici1}) (with $t=0$). Let $u_\delta(t,x)$ be the entropy
solution of (\ref{RHEg}) with $u_\delta(0)=u_{0\delta}$. By
Theorem \ref{thm:reg1D} we know that each $u_\delta(t,x)$  is
smooth inside $(a,b)$. Let us note that the local bounds on
$u_\delta$ and its derivatives do not depend on $\delta$. It
suffices to observe that this is true for the associated functions
$v_\delta(t,\eta)$ which are solutions of (\ref{RHEdv}),
(\ref{verticalBC}), with initial data
$v_\delta(0,\eta)=v_{0\delta}(\eta)$. Note that the bounds in
Steps 1, 2, 3 in the proof of Theorem \ref{thm:regv1D} are
independent of $\delta$. By Remark \ref{rk:LB1}, the Lipschitz
bound in Step 5 depends only on the local Lipschitz bounds of
$v_{0\delta}(\eta)$ and are, thus, uniform in $\delta$. Step 6
proves uniform (in $\delta$) interior bounds for any space and
time derivative of $v_{\delta}(t,\eta)$. By passing to the limit
as $\delta \to 0+$ we conclude that $u(t,x)$ is smooth inside its
support and $(i)$ and $(ii)$ in Theorem \ref{thm:reg1D} hold.
\end{proof}

We now generalize our main results to initial data vanishing at
the boundary of the support.

\begin{proposition}\label{prop:locL2}
Let $u_0\in  L^\infty(\R)$ with $u_0(x) > 0$ for $x\in (a,b)$, and
$u_0(x)=0$ for $x\not\in (a,b)$. Assume that $u_0\in
W^{1,\infty}_{\rm loc}(a,b)$ and $u_0(x) \to 0$ as $x\to a,b$. The
entropy solution $u(t,x)$ of \eqref{RHE} with $u(0)=u_0$ satisfies
$(i)$ and $(ii)$ in Theorem {\rm \ref{thm:reg1D}}. Moreover, if
$u_0(x) \leq A (b-x)^\alpha(x-a)^\alpha$ for some $A, \alpha > 0$,
then $u(t,x) \leq A(t) (b+t-x)^\alpha(x-a+t)^\alpha$ for any $x\in
(a-t,b+t)$, $t > 0$ and some $A(t)$. In that case, $u(t,x)$ is a
continuous function that tends to $0$ as $x \to a-t,b+t$.
\end{proposition}

\begin{proof}
Let $u_{0\delta} \in  L^\infty(\R)$ with $u_{0\delta}(x) = u_0(x)
+ \delta$ for $x\in [a,b]$, $u_{0\delta}(x)=0$ for $x\not\in
[a,b]$, and $u_{0\delta}\in W^{1,\infty}([a,b])$ with uniform
local Lipschitz bounds in $(a,b)$. Let $v_{0\delta}(\eta)$ be the
functions obtained by the change of variables (\ref{CVinici1})
(with $t=0$). Let $u_\delta(t,x)$ be the entropy solution of
(\ref{RHEg}) with $u_\delta(0)=u_{0\delta}$. By Theorem
\ref{thm:reg1D} we know that each $u_\delta(t,x)$  is smooth
inside $(a,b)$. Let us note that the local bounds on $u_\delta$
and its derivatives do not depend on $\delta$. Again, it suffices
to observe that this is true for the associated functions
$v_\delta(t,\eta)$ which are solutions of (\ref{RHEdv}),
(\ref{verticalBC}), with initial data
$v_\delta(0,\eta)=v_{0\delta}(\eta)$.

The $L^p$ bounds follow from Step 1 in the proof of Theorem
\ref{thm:regv1D} for $p\in [1,\infty)$ and they only depend on the
$L^p$ bound of $v_{0\delta}$. Actually, we have
$$
\int_0^1 v_{0\delta}(\eta)^p\, d\eta = \int_a^b
\frac{1}{u_{0\delta}(x)^{p-1}}\, dx\,,
$$
that depends on the integrability of $\frac{1}{u_{0\delta}(x)}$ at
the boundary points. But multiplying (\ref{RHEdv}) by
$v_\delta^p\phi$ where $p\in [1,\infty)$ and $\phi$ is a positive
smooth test function with compact support in $(0,1)$ we obtain
\begin{equation*}
\frac{1}{p+1}\frac{d}{dt} \int_0^1 v_\delta^{p+1}(t,\eta)\phi\,
d\eta + \int_0^1  |(v_{\delta}^p)_\eta|\phi \leq   p \int_0^1
v_\delta^{p+1} \phi\, d\eta + \int_0^1 v_\delta^p |\phi_\eta| \,
d\eta.
\end{equation*}
Thus we derive local $L^p$ bounds for $v_\delta$ which are
independent of $\delta$. We also obtain local bounds on the total
variation of  $v_\delta^p$ which are independent of $\delta$. To
obtain a local $L^\infty$ bound independent of $\delta$ we observe
that this follows from the identity $v_\delta(t,\eta)=
\frac{1}{u_\delta(t,x)}$, where $x = \varphi_\delta(t,\eta)$ is
given by (\ref{CVinici1}), since we know that $u_\delta(t,x)$ is
locally bounded away from zero in its support \cite{ACMRelatQ}.
Thus Steps 1, 2, 3 hold in their local versions. By Remark
\ref{rk:LB1}, the Lipschitz bound in Step 5 depends only on the
uniform local bounds on $v_\delta(t,\eta)$ and on the local
Lipschitz bounds of $v_{0\delta}(\eta)$ and are, thus, uniform in
$\delta$. Step 6 proves uniform (in $\delta$) interior bounds for
any space and time derivative of $v_{\delta}(t,\eta)$. By passing
to the limit as $\delta \to 0+$ we conclude that $u(t,x)$ is
smooth inside its support and $(i)$ and $(ii)$ in Theorem
\ref{thm:reg1D} hold. The last assertion is a consequence of the
comparison principle using Lemma \ref{ss} below.
\end{proof}

\begin{remark}{\rm
Note that the last assertion implies that if the initial profile
is not vertical at the boundary at $t=0$ it remains non-vertical
for any $t > 0$. Moreover, during the proof we have observed that
if $u_0$ has a vertical profile with $\frac{1}{u_0} \in L^p(a,b)$,
then $\frac{1}{u(t,x)} \in L^p(a-t,b+t)$ for any $t > 0$. Thus in
that case $u(t,x)$ has a vertical profile at the boundary of its
support.}
\end{remark}

Due to translational invariance of \eqref{RHE}, we state our next
Lemma in an interval symmetric around zero.

\begin{lemma}\label{ss}
Let $U(t,x)= A(t) (R(t)^2-x^2)^\alpha$ where $R(t)=R_0 + t$,
$\alpha > 0$. If $A'(t)\geq 0$, then $U(t,x)$ is a supersolution
of \eqref{RHE}.
\end{lemma}

\begin{proof}
Computing the derivatives, we get
$$
U_t = A' (R(t)^2-x^2)^\alpha + 2A \alpha R(R^2 -x^2)^{\alpha-1} ,
$$
$$
U_x = -2A \alpha x(R^2 -x^2)^{\alpha-1}\,,
$$
and
$$
(U^2+U_x^2)^{1/2} = A(R^2 -x^2)^{\alpha-1}  Q(x)\,,
$$
where $Q(x) = \left( (R^2 -x^2)^2 + 4 \alpha^2 x^2\right)^{1/2}$,
and then
$$
\frac{UU_x}{(U^2+U_x^2)^{1/2}} = -\frac{2A\alpha x (R^2-x^2)^\alpha}{Q}.
$$
Thus, the claim
\begin{equation*}
U_t \geq \, \left( \frac{ U U_x}{\sqrt{U^2 + (U_x)^2}} \right)_x
\end{equation*}
holds if and only if
\begin{align*}
A' (R^2-x^2)^\alpha + 2A \alpha R(R^2 -x^2)^{\alpha-1} \geq & \,-
\frac{2A\alpha (R^2-x^2)^\alpha}{Q}  + \frac{2A\alpha x
(R^2-x^2)^\alpha Q_x}{Q^2} \\ & \,+ \frac{4A\alpha^2 x^2
(R^2-x^2)^{\alpha-1 }}{Q}\,.
\end{align*}
Let us prove that
\begin{equation*}
2A \alpha R(R^2 -x^2)^{\alpha-1}  \geq \frac{4A\alpha^2 x^2 (R^2-x^2)^{\alpha-1 }}{Q}.
\end{equation*}
Indeed, the above inequality is implied by $2R\geq 4 \alpha x^2
/Q$ and $ 4 \alpha  x^2 \leq (2R) 2 \alpha  |x|  \leq 2R Q$. Now,
we choose $A$ such that
$$
A' (R^2-x^2)^\alpha \geq - \frac{2A\alpha (R^2-x^2)^\alpha}{Q}  +
\frac{2A\alpha x (R^2-x^2)^\alpha Q_x}{Q^2}\,,
$$
that is,
\begin{equation}\label{Agrows}
A'   \geq - \frac{2A\alpha  }{Q}  + \frac{2A\alpha x  Q_x}{Q^2} =
\frac{2A\alpha  }{Q} \left(-1 +  \frac{xQ_x}{Q}\right)\,.
\end{equation}
Noticing that
$$
 \frac{ x Q_x}{Q } = \frac{4\alpha^2 x^2 - 2x^2(R^2 -x^2)}{Q^2}   \leq \frac{4\alpha^2 x^2 }{Q^2}\leq
 1\,,
$$
hence (\ref{Agrows}) holds if $A'\geq 0$. We have proved that if
$A'\geq 0$, then $U(t,x)$ is a supersolution of (\ref{RHE}).
\end{proof}


\section{Numerical experiments and heuristics}
\label{sect:numerics}

In this section, we will propose a numerical scheme for more
general equations than the RHE \eqref{RHEg}. We deal with the
Cauchy problem for the generic porous media relativistic heat
equation (RHEm) \cite{casellespm} given by
\begin{equation}\label{pmrel}
u_ t = \, \left( \frac{ u^m u_x}{\sqrt{u^2 +
(u_x)^2}} \right)_x
\end{equation}
with initial data $u_0$ a probability density with compact
support. In order to propose the numerical scheme, we make use of
the change of variables to Lagrangian coordinates. As in the
introduction, let us denote by $F$ the distribution function
associated to the probability density $u$ and $\varphi(t,\eta)$
its inverse or generalized inverse, defined by
\begin{equation}\label{defphi}
  \varphi(t,\eta):=\left\{\begin{array}{lc}-\infty & \eta=-\frac{1}{2} \\[2mm] \inf\{x:\ F(t,x) >\eta+\tfrac12\}\,,& \qquad   \eta\in
  (-\tfrac{1}{2},\tfrac{1}{2})\\[2mm] +\infty & \eta=\frac{1}{2} .\end{array}\right.
\end{equation}
Here, we have preferred to shift the mass variable to the interval
$(-\frac{1}{2},\frac{1}{2})$ to simplify the notations about
boundary conditions. In this way, we simply have the relation
\begin{equation}\label{changeofvariables}
F(t,\varphi(t,\eta)) =\eta, \qquad   \eta\in
(-\tfrac{1}{2},\tfrac{1}{2}).
\end{equation}
For simplicity, most of the numerical tests have been chosen for
even initial data. Observe that this change of variables is a weak
diffeomorphism in case of connected compactly supported smooth
$u$, say on the interval $(-A(t),A(t))$ in which case
\begin{equation}\label{limboundary}
\lim_{\eta\to\pm\frac{1}{2}^\mp} \varphi(t,\eta)=\pm
A(t).
\end{equation}
Straightforward computations show that the equation satisfied by
$\varphi$ in $(-\frac{1}{2},\frac{1}{2})$ is
\begin{equation}\label{pmRHEd}
\varphi_t =   -\frac{\left(\frac{1}{\varphi_\eta}\right)^{m-1}\left(\frac{1}{\varphi_\eta}\right)_\eta}{\sqrt{1 +
\left(\frac{1}{\varphi_\eta}\right)_\eta^2}}\,,
\end{equation}
while at the boundary, formally, by \eqref{defphi} and
\eqref{limboundary}, we have to impose
\begin{equation}\label{pmeboundary3}
\varphi_\eta\left(t,\tfrac12\right)=+\infty.
\end{equation}
Moreover, thanks to the vertical contact angle property (see
\eqref{BCmm1} for the RHE and \cite{casellespm} for the RHEm), we
have that
\begin{equation}\label{pmeboundary}
\lim_{\eta\to\pm\frac{1}{2}^{\mp}}\left(\frac{1}{\varphi_\eta}\right)_\eta(t,\eta)=\mp\infty.
\end{equation}

The purpose of this section is two-fold. On one hand, we
heuristically observe some qualitative properties from the
Lagrangian viewpoint. On the other hand, these properties are
confirmed by numerical experiments with the use of an adaptation
of the algorithm proposed in \cite{CarrilloMoll} for general
equations in continuity form for the 2-dimensional case.

\subsection{Numerical Method}

Equations \eqref{RHE} and \eqref{pmrel} have been numerically
treated in \cite{Marq,SM} using the connection between nonlinear
diffusions and Hamilton-Jacobi equations and numerical methods for
conservation laws and in \cite{ACMSV} using an appropiate WENO
scheme. Here, we propose a completely different approach based on
the optimal transportation viewpoint. As we already mentioned in
the introduction an explicit Euler discretization of the equation
satisfied by the generalized inverse \eqref{pmRHEd} coincides with
the variational scheme introduced in \cite{JKO,Otto01}. Moreover,
the theoretical result proven in \cite{MP} shows that this scheme
applied to \eqref{RHE} is convergent for initial data compactly
supported smooth in their support and bounded below and above.
Therefore, we plan to use a similar algorithm for Eq.
\eqref{pmrel}. This Lagrangian formulation in 1D for nonlocal and
nonlinear diffusion problems was numerically analysed in
\cite{GT1,BCC}. These Lagrangian coordinates ideas were
generalized to several dimensions in \cite{CarrilloMoll}.

The advantages of this method are the adaptation of the mesh to
the mass distribution of the solution in an automatic way, the
immediate positivity of the solutions, and the decay of the
natural Liapunov functional of the equations. We refer to
\cite{CarrilloMoll} for more details and discussions on these
issues.

Here, we propose an adaptation of the algorithm in
\cite{CarrilloMoll}. First of all, the discretization in the mass
variable has been treated by finite difference approximations of
the derivatives of the unknown $\varphi$. We consider a partition
$\{\eta_i\}_{i=1:N}$ of the spatial interval
$[-\frac{1}{2},\frac{1}{2}]$ and we let
$\Delta_i:=\eta_{i+1}-\eta_i$. Note that, due to \eqref{defphi},
first derivatives at the points corresponding to the nodes
$\eta_2$ and $\eta_{N-1}$ have to be taken from the inside of the
domain. In order to avoid higher errors in the approximation of
the derivative at the boundaries, we decide to approximate
$\varphi_\eta$ as
$$
\varphi_\eta(\eta_i):=\left\{\begin{array}{cc}\displaystyle\frac{\varphi(\eta_{i+1})-\varphi(\eta_i)}{\Delta_i}
& {\rm if \ } \eta_i\leq \eta(t) \\[4mm]
\displaystyle\frac{\varphi(\eta_{i})-\varphi(\eta_{i-1})}{\Delta_{i-1}}
& {\rm if \ } \eta_i > \eta(t)\end{array}\right.
$$
with $\eta(t)$ to be specified. The derivative of the term
$\frac{1}{\varphi_\eta}$ is computed in the other direction for
better stability properties of the approximation of
$\left(({\varphi_\eta})^{-1}\right)_\eta$. At the boundary we just
impose  \eqref{pmeboundary3}.

As explained in \cite{CarrilloMoll}, the point $\eta(t)$ has to be
taken as the global maximum for $u$, which can be tracked at any
time step. In all examples computed, initial data are taken to be
radially symmetric and decreasing from the point $x=\eta=0$. In
all of them, the global maximum stays at $x=\eta=0$. Therefore, we
choose to take an even number of points $N$ in the discretization
and to take a symmetric partition $\{\eta_i\}_{i=1:N}$ of the
spatial interval $[-\frac{1}{2},\frac{1}{2}]$. Let us point out
that the spatial partition is never uniform since the change to
Lagrangian coordinates produces the accumulation of nodes near the
global maximum. We instead want to follow some particular features
of these type of equations such as propagation of fronts with a
vertical contact angle or formation of singularities. Therefore,
the partitions will be chosen accordingly in order to accumulate
more points around the points $\pm\frac{1}{2}$ and other points of
interest. The time derivative is evaluated through a simple
explicit Euler scheme with the CFL condition proposed in
\cite{CarrilloMoll}; i.e:
$$
\left\|\left(\frac{1}{\varphi_\eta}\right)^m\right\|_\infty\frac{\Delta
t}{(\Delta \eta)^2}\leq \frac{1}{\alpha_{CFL}}\,,
$$
with $\alpha_{CFL}>2$, for the porous-medium equation which is the
large-time limit behaviour of \eqref{pmrel}, see \cite{casellespm}
and subsection 5.3. All our simulations are done with
$\alpha_{CFL}=8$. Although the CFL analysis in \cite{CarrilloMoll}
applies only to equations written in variational form that
includes \eqref{pmrel} only for $m=1$, all numerical tests seem
not to be affected by the chosen CFL condition. Finally, we point
out that $u(t,\varphi(t,\eta_1))=u(t,\varphi(t,\eta_N))=0$.
Because of this fact, in all the plots which follow, the first and
last nodes are never plotted.

\subsection{Formation of discontinuities}

\subsubsection{Propagation of the support of solutions and waiting time phenomenon}

Observe that Eq. \eqref{pmRHEd} and \eqref{pmeboundary} imply that
the speed of propagation of the support is exactly
\begin{equation}\label{speedsupport}
\varphi_t(\pm\tfrac{1}{2}^\mp)=\pm\left(\frac{1}{\varphi_\eta(\pm\tfrac{1}{2}^\mp)}\right)^{m-1}=\pm
u^{m-1}(\pm A(t))\,.
\end{equation}
(here and from now on $f(a^\pm):=\lim_{x\to a^{\pm}}f(x)$ for a
generic function $f$ and point $a$). This coincides with well-known
results in \cite{CER2}. If we let
$0\leq\psi(\eta)=\frac{1}{\varphi_\eta}(\eta)=u(\varphi(\eta))$,
then \eqref{pmRHEd} transforms into:
\begin{equation}\label{eqlagrangian}
  \psi_t=\psi^2\left(\frac{\psi^{m-1}\psi_\eta}{\sqrt{1+\psi_\eta^2}}\right)_\eta.
\end{equation}
Note that
\begin{equation}\label{pmeboundary2}
\psi_\eta(t,\eta)=\left(\frac{1}{\varphi_\eta}\right)_\eta(t,\eta) =
\left(\frac{u_x}{u}\right)(t,\varphi(t,\eta))\,.
\end{equation}
In case $u(\pm A(t))\neq 0$ or $u_x(\pm A(t))\neq 0$ if $u(\pm
A(t))=0$, then the boundary condition for $\psi$ is just a
vertical contact angle using
\eqref{pmeboundary}-\eqref{pmeboundary2}:
\begin{equation}\label{vca}
\psi_\eta(t,\pm \tfrac{1}{2}^\mp)=\mp \infty\,.
\end{equation}

\begin{figure}[htb]
\includegraphics[width=8.5cm]{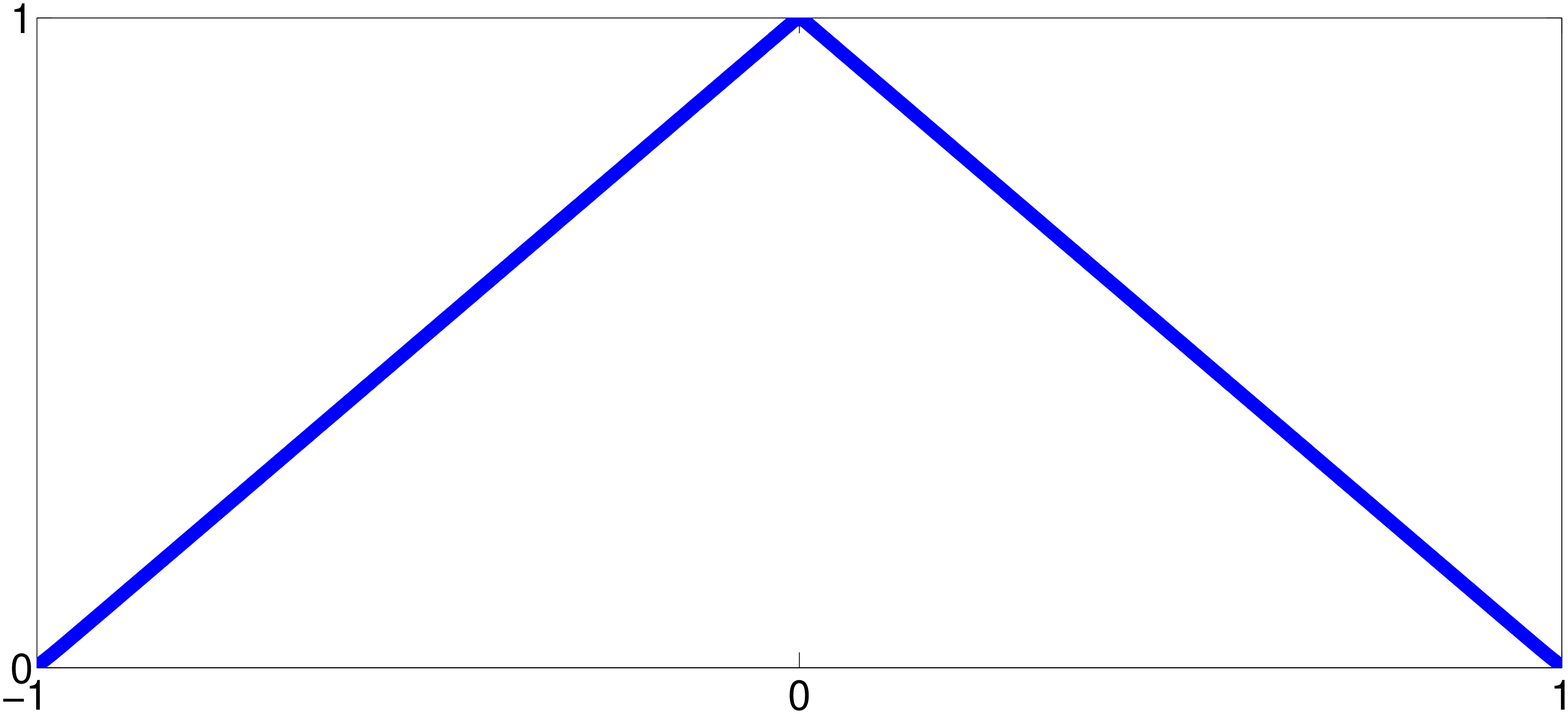}
\includegraphics[width=8.5cm]{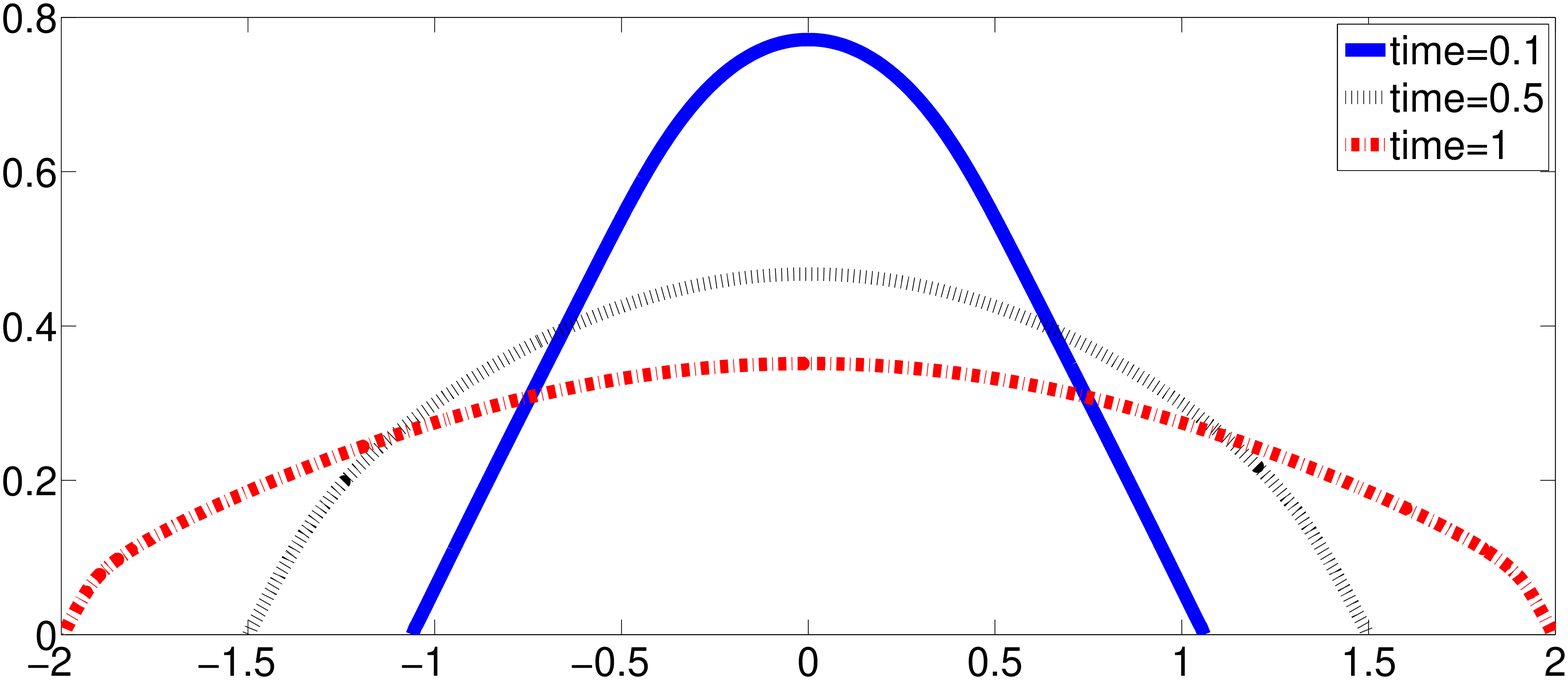}
\caption{Left: initial datum. Right: Evolution of $u_0$ in case
$m=1$ at different times.} \label{fig.triangle1}
\end{figure}

Consider now $m=1$. By \eqref{RHEd}, $|\varphi_t|\leq 1$ and
$\varphi_t(\pm\frac{1}{2}^\mp)=\pm 1$, it follows that
$(\varphi_\eta)_t(\pm\frac{1}{2})\geq 0$. This implies that
$\psi_t(\pm\frac{1}{2}^\mp)\leq 0$ by definition of
$\psi(t,\eta)$. In particular, this shows that in case
$\psi(t_0,\pm\frac{1}{2}^\mp)=0$, this condition remains true for
all time as shown in Proposition \ref{prop:locL2}.

We define next
$w(t,\eta):=\psi(t,\eta)\psi_\eta(t,\eta)=u_x(t,\varphi(t,\eta))$. The analysis above also shows that, in case $\psi(t_0,\pm\frac{1}{2}^\mp)=0$, then $|w(t,\pm\frac{1}{2})|\leq |w(0,\pm\frac{1}{2})|$.
On the other hand, in the bulk, $w$ verifies the following equation
$$
w_t=\frac{\psi^5
w_{\eta\eta}}{(\psi^2+w^2)^{\frac{3}{2}}}+\frac{3\psi
w}{(\psi^2+w^2)^{\frac{5}{2}}}(2w^2
w_\eta\psi^2-w_\eta^2\psi^4-w^4).
$$
Thus, if $w_0$ is initially bounded, $w$ remains bounded in
$[-\frac{1}{2},\frac{1}{2}]$ as proved in Section \ref{sect:Extensions}.
Observe that at a point $\eta_0$ of maximum of $w$, we have
$$
w_t(\eta_0)\leq -\frac{3\psi
w^5}{(\psi^2+w^2)^{\frac{5}{2}}}(\eta_0)\leq 0,
$$
implying the claim.

We show a numerical experiment with $u_0(x)=(1-|x|)_+$ as initial
datum which does not satisfy the conditions of Theorem
\ref{thm:reg1D}. We take $N=1000$ for the simulations. We point
out that since the initial datum is $0$ at the extremes of the
support, we need a lot of nodes in the discretization near them
since due to the change of variables \eqref{changeofvariables},
then $\varphi_\eta(\pm\frac{1}{2}^\mp)=\mp\infty$ and we want the
numerical scheme to be able to capture this feature. We report in
Fig. \ref{fig.triangle1} the precise evolution of the support
showing the smoothing effect at $x=0$, the boundedness of the
derivative all over the support including the boundaries, and the
expansion of the boundary at precise unit speed as expected by the
theory in Theorem \ref{thm:reg1D} and the heuristic arguments
above.

\begin{figure}[ht]
\includegraphics[width=8.5cm]{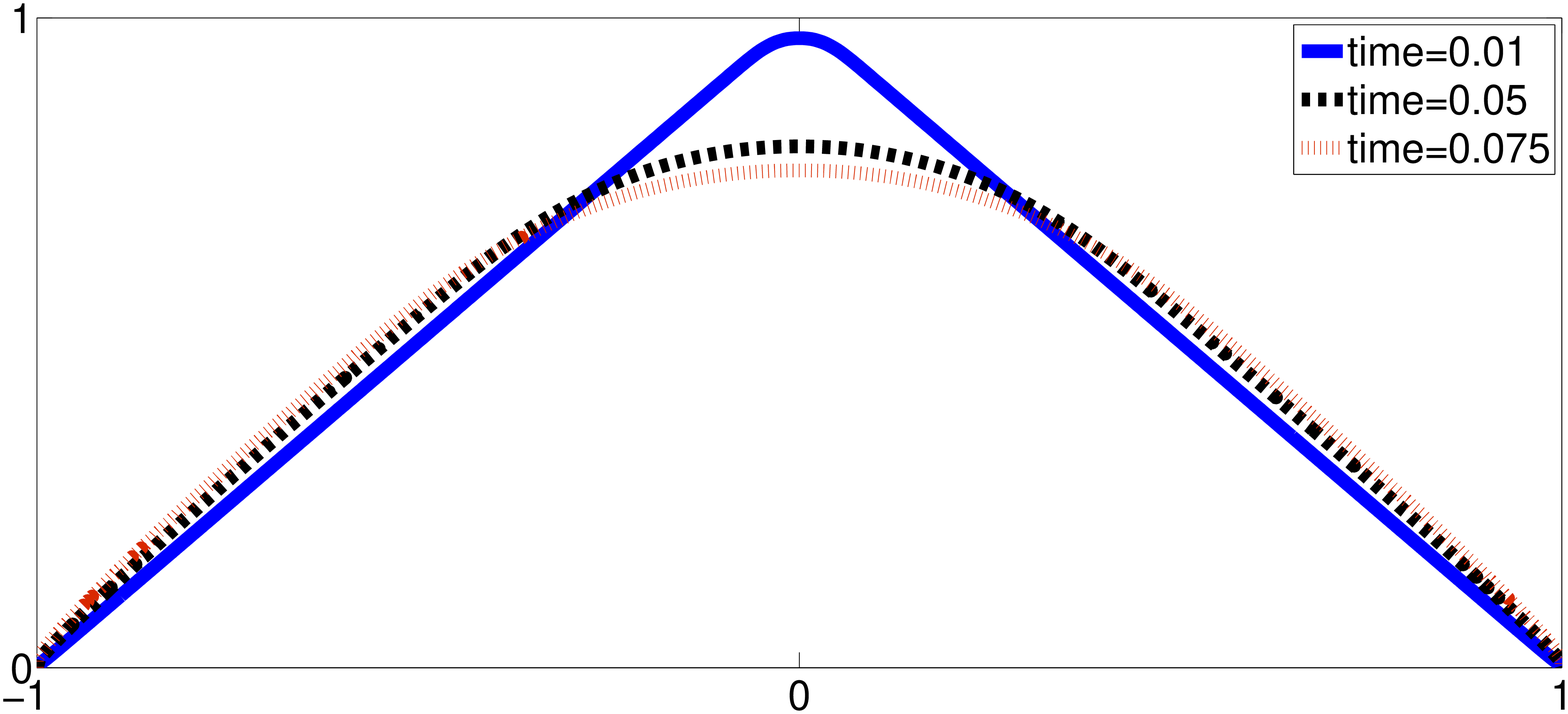}
\includegraphics[width=8.5cm]{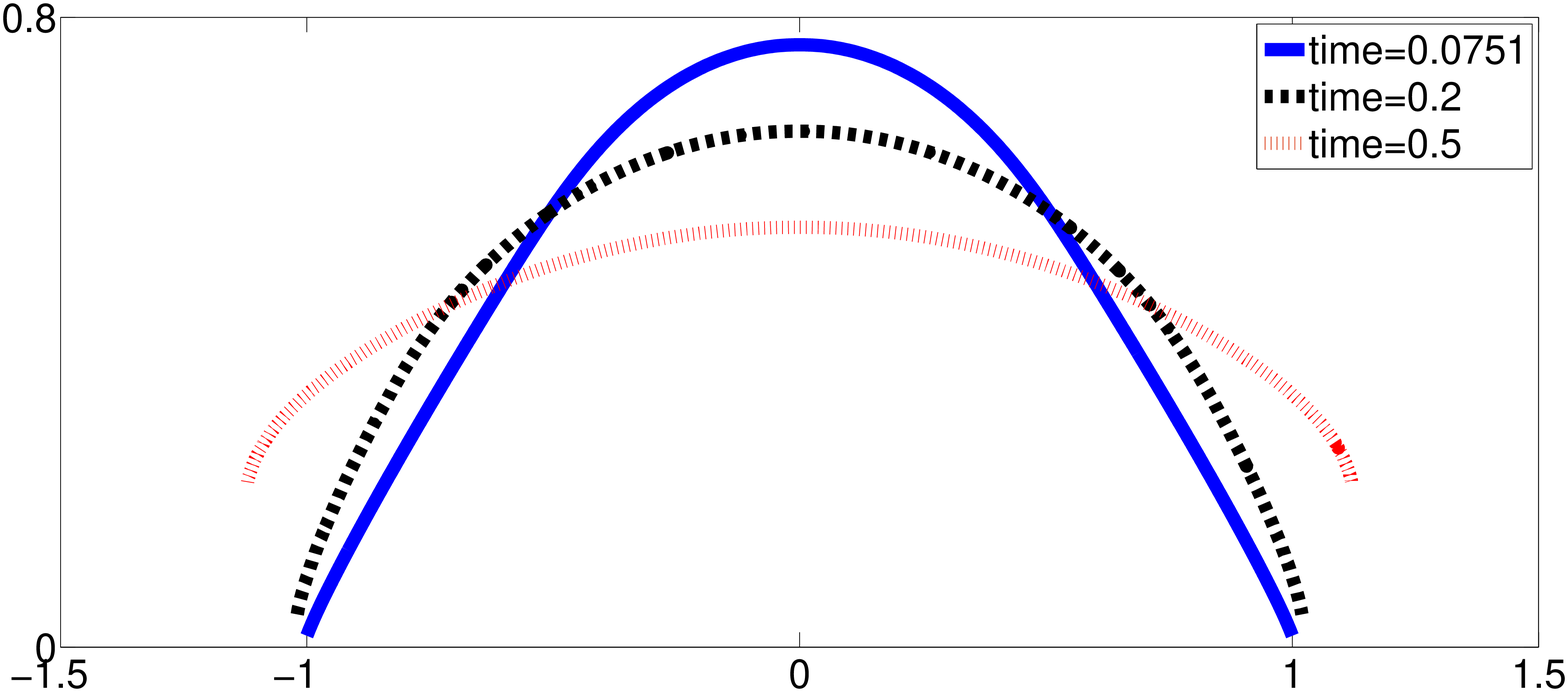}
\caption{Evolution of $u_0$ in case $m=1.5$ at different times.
Left: Before the discontinuity at the tip of the support appears.
Right: Evolution of the discontinuity front after.}
\label{fig.triangle1_5}
\end{figure}

Let us now take $m>1$. In case $u(\pm A(0)^\mp)=0$ (i.e.
$\psi(0,\pm\frac{1}{2}^\mp)=0$), then \eqref{speedsupport} implies
that the support of the solution does not move at all whenever
$u(\pm A(t))=0$. The solution will become positive at the tip of
the support $u(\pm A(t))>0$ with $t>t_0>0$ if and only if
$\psi_t(t_0,\pm\frac{1}{2}^\mp)\in (0,+\infty]$ with $u(\pm
A(t_0)^\mp)=\psi(t_0,\pm\frac{1}{2}^\mp)=0$. In case $u_x(\pm A(t_0))\neq
0$, we can use \eqref{vca} to approximate terms
$(1+\psi_\eta^2)^{1/2}\simeq \psi_\eta$ around $\pm \frac{1}{2}$
in the expression of \eqref{eqlagrangian} to get
$$
\psi_t(t_0,\pm \tfrac{1}{2})=\lim_{\eta\to \pm\frac{1}{2}}
\psi_t(t_0,\eta) = \lim_{\eta\to \pm\frac{1}{2}}
(m-1)(\psi^m\psi_\eta) (t_0,\eta).
$$
As a consequence, $\psi_t(t_0,\pm \tfrac{1}{2}^\mp)>0$ if and only if
\begin{equation}\label{conds}
\lim_{\eta\to \pm\frac{1}{2}} (\psi^{m+1})_\eta (t_0,\eta)>0\,.
\end{equation}
Observe that this condition in \eqref{conds} is implied by
$$
\lim_{x\to \pm A(t)} (u^{m+1}(x))_x (t_0,\pm A(t))>0\,.
$$
In such case, the solution becomes positive at $\pm A(t)$ and
then, according to \eqref{speedsupport}, its support starts to
increase. We note that this waiting time phenomenon is similar to
that of the classical porous medium equation but the condition for
the support to start moving is completely different to the one
obtained in \cite{acaka}. Supposing a potential growth of $\psi$,
i.e. $\psi(t_0,\eta)\simeq
C\left(t_0,\frac{1}{2}-|\eta|\right)^p$, $p>0$, for
$\eta\to\pm\frac{1}{2}^\mp$, then we obtain that
$\psi_t(t_0\pm\frac{1}{2}^\mp)=+\infty$ if and only if
$p<\frac{1}{m+1}$.

\begin{figure}[ht]
\includegraphics[width=8.5cm]{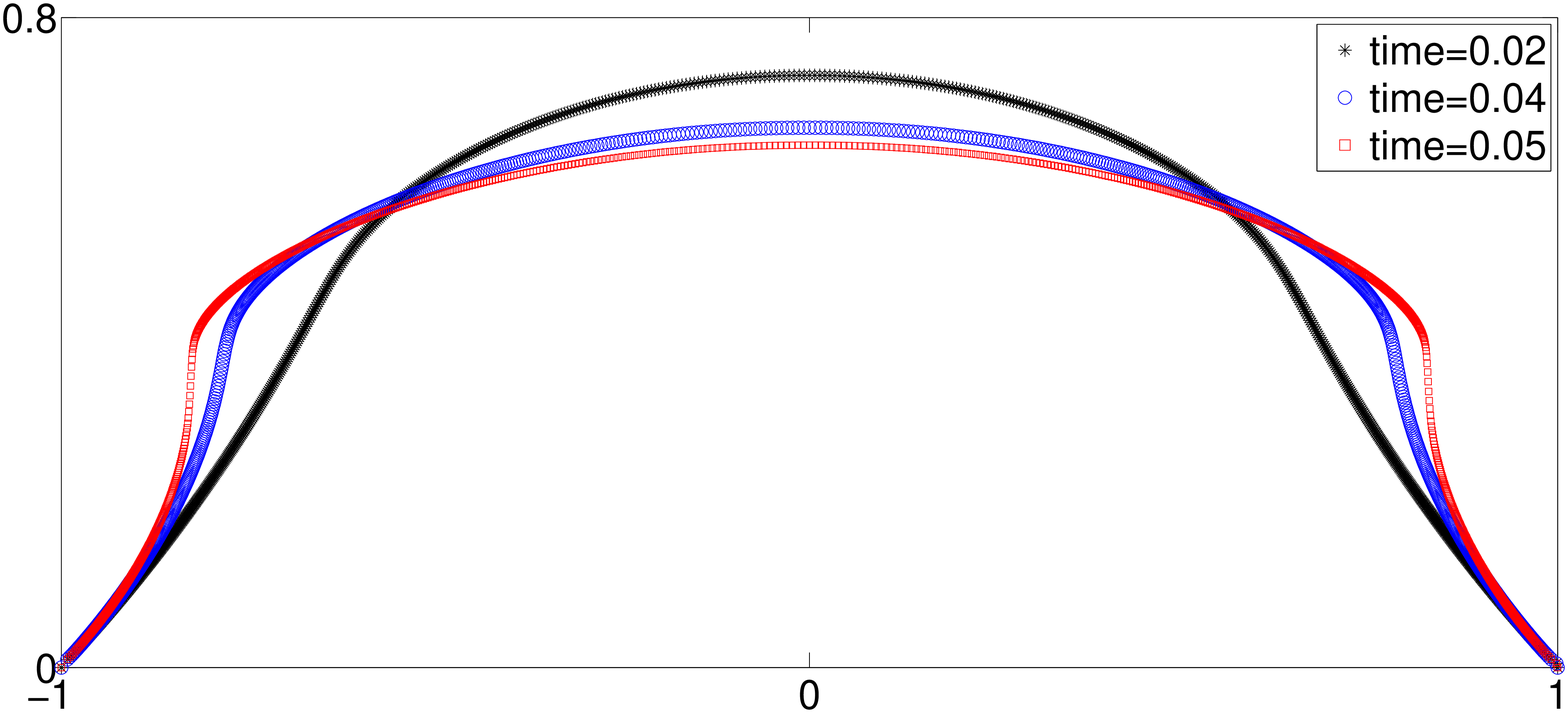}
\includegraphics[width=8.5cm]{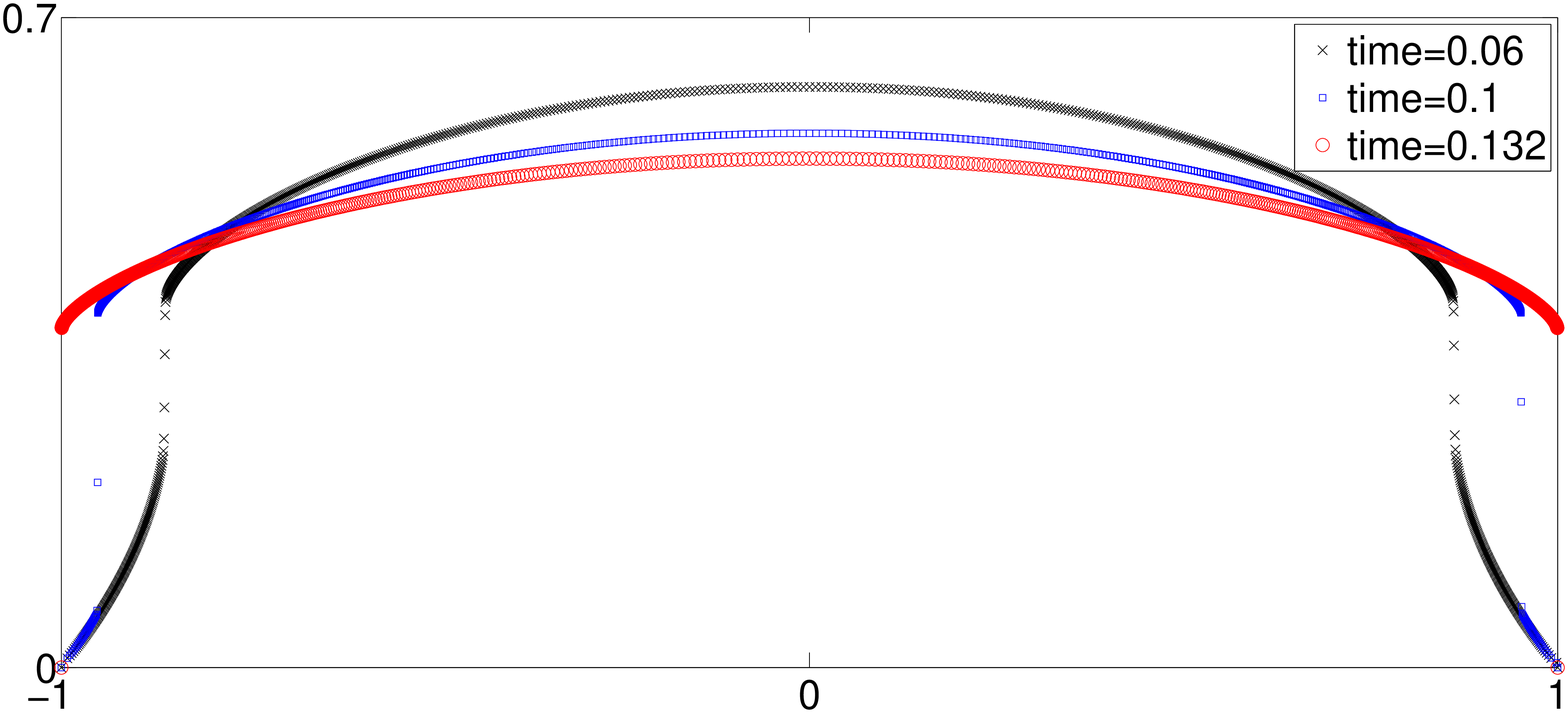}
\begin{center}{\includegraphics[width=8.5cm]{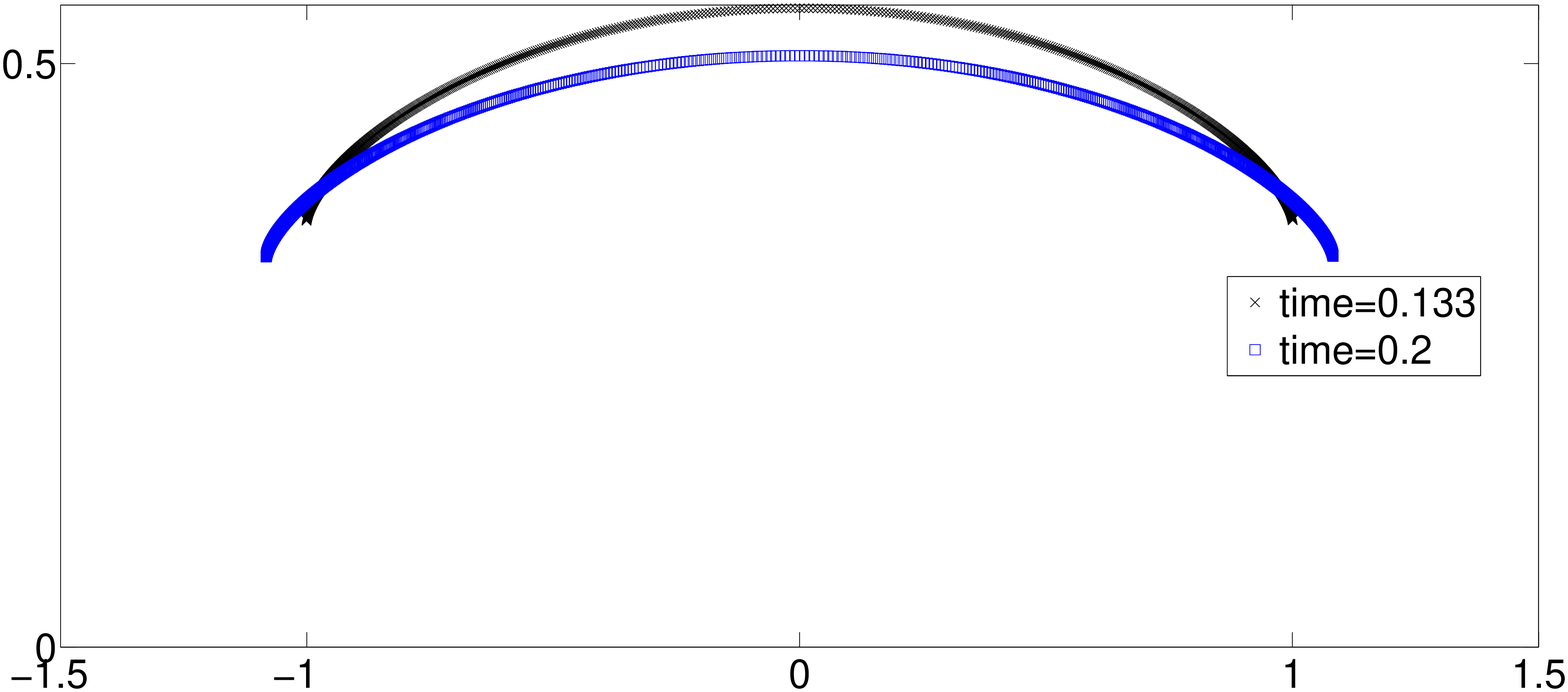}}\end{center}
\caption{Evolution of $u_0$ in case $m=3$ at different times. Top
left: Before a discontinuity on the bulk appears. Top right: After
the discontinuity front forms till it reaches the tip of the
support. Bottom: After the discontinuity front starts to move.}
\label{fig.triangle3}
\end{figure}

We point out that this behavior has already been numerically
obtained in \cite{ACMSV}. In Fig. \ref{fig.triangle1_5}, we show
this waiting time phenomenon for $m=1.5$. One can observe that
initially the support does not move since the behavior near the
boundary is $\psi(0,\eta)\simeq
C\left(0,\frac{1}{2}-|\eta|\right)^\frac12$, then the derivative
at the boundary builds up until the behavior at the boundary
reaches the critical value producing the lift-off of the boundary
point. More interesting is the case $m=3$ which we show in Fig.
\ref{fig.triangle3}. There, a discontinuity in the bulk appears
before the support starts to move.

\subsubsection{Formation of discontinuities in the bulk}

In view of the first example in the last section, one may think
that discontinuities may appear only as a consequence of the
waiting time phenomenon; i.e. particles tend to dissipate but
their support does not move, which may create the discontinuities.
In this section we heuristically study that it is possible to
create discontinuities inside the bulk even if the solution are
far away from zero as seen in Fig. \ref{fig.triangle3}.

First we treat the case $m=1$. In case of an upwards jump
discontinuity or a vertical angle at a point $\eta_0\in
]-\frac{1}{2},\frac{1}{2}[$ such that
$\psi_\eta(\eta_0)^\pm=+\infty$ , then we also have
$\varphi_t(\eta_0)=-1$. Since $|\varphi_t|\leq 1$, then
$\varphi_t(\eta_0)=-1$ implies that $\varphi_t$ is nonincreasing
to the left and nondecreasing to the right of $\eta_0$, i.e.,
$((\varphi_\eta)^-)_t\leq 0$ and $((\varphi_\eta)^+)_t\geq 0$.
This shows that $(\psi(\eta_0)^-)_t\geq 0$ while
$(\psi(\eta_0)^+)_t\leq 0$, which implies that the size of the
discontinuity reduces in for an upwards jump discontinuity or that
no discontinuity is created if initially there is a vertical angle.

This last phenomenon is not true if $m>1$ in the case of a
vertical angle at a point $\eta_0\in ]-\frac{1}{2},\frac{1}{2}[$
such that $\psi_\eta(\eta_0)^\pm=+\infty$. From the equation
\eqref{eqlagrangian} for $\psi$ as in previous subsection, we
deduce that $ \psi_t(\eta_0)=(m-1)\psi^m\psi_\eta(\eta_0)$, and
thus, a discontinuity is created. Once we have a discontinuity at
$\eta_0$ the evolution is theoretically unknown.

\begin{figure}[ht]
  \includegraphics[width=8.5cm]{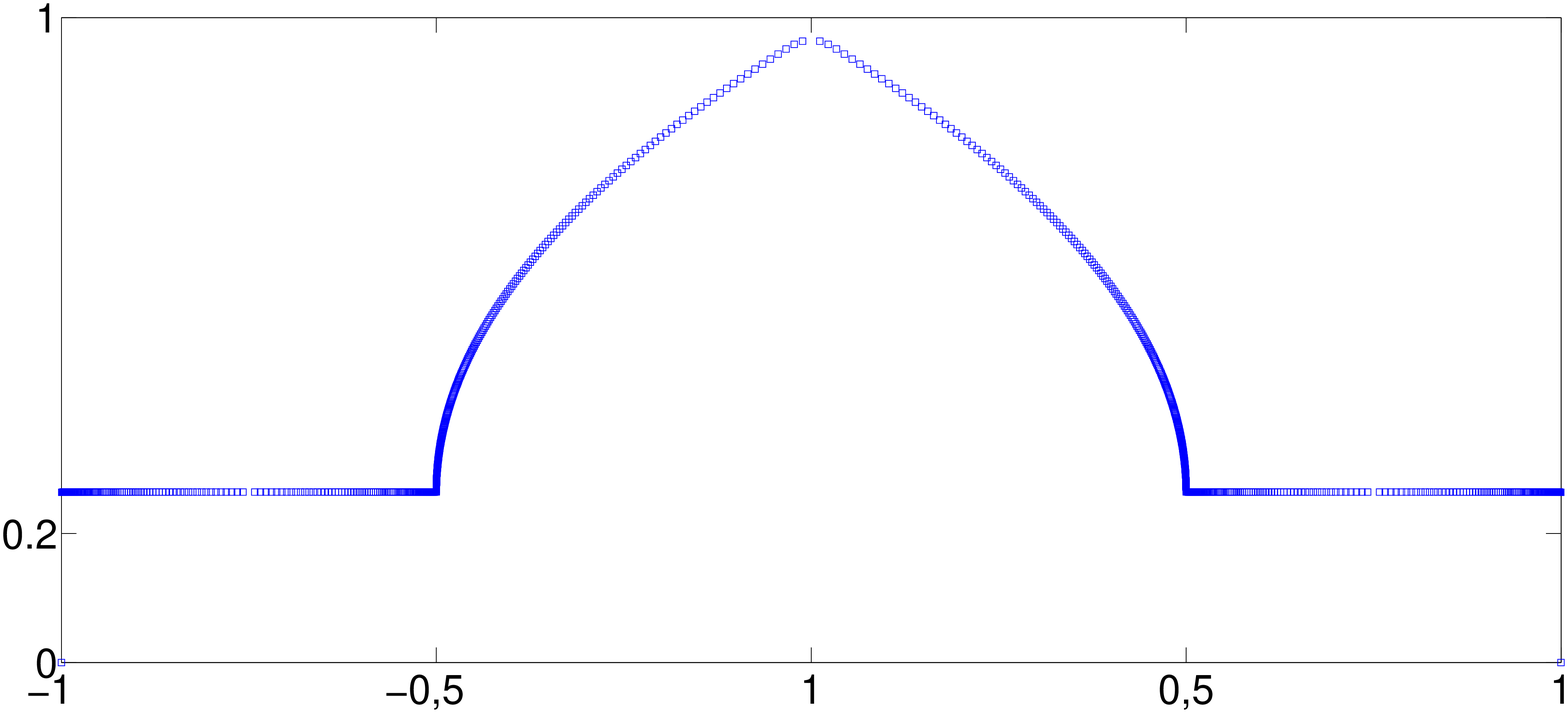}\includegraphics[width=8.5cm]{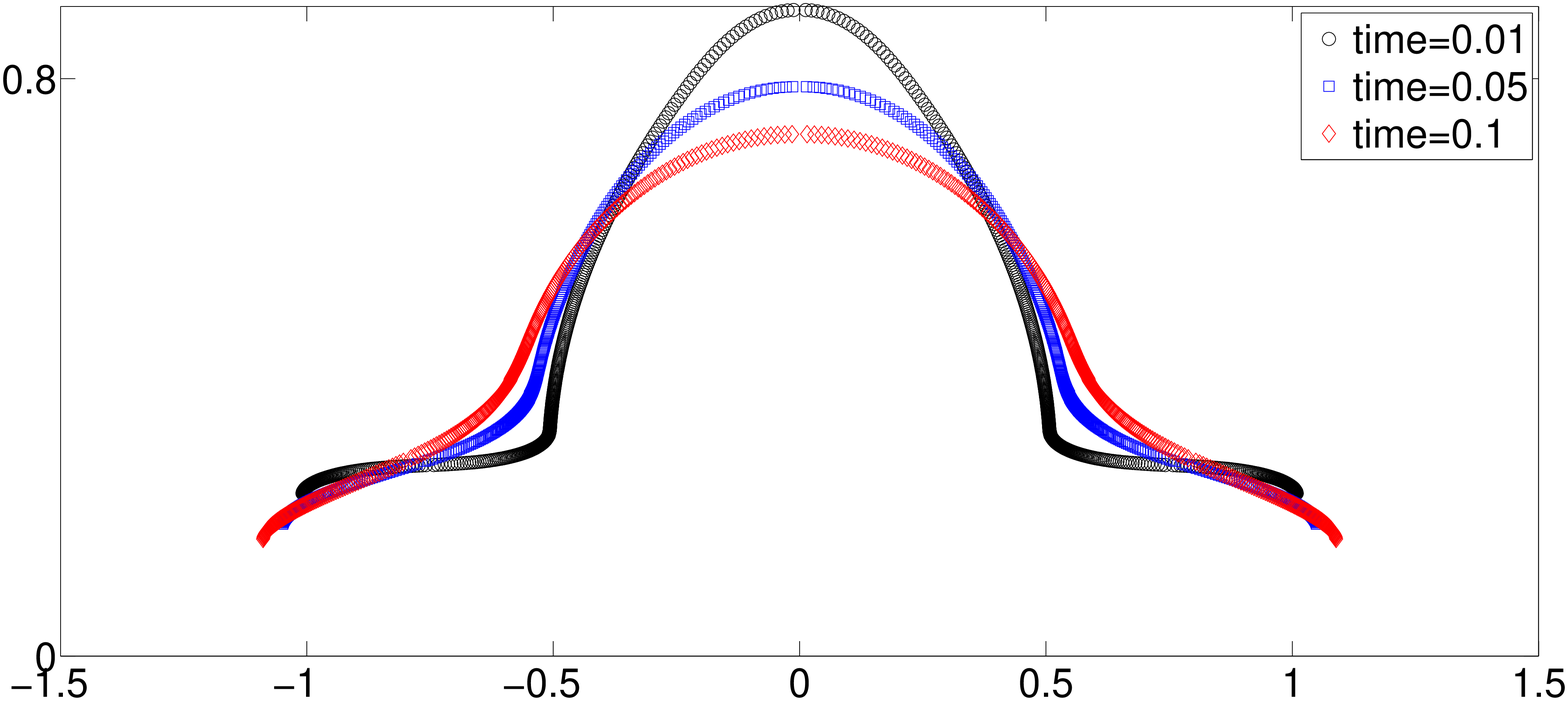}
\includegraphics[width=8.5cm]{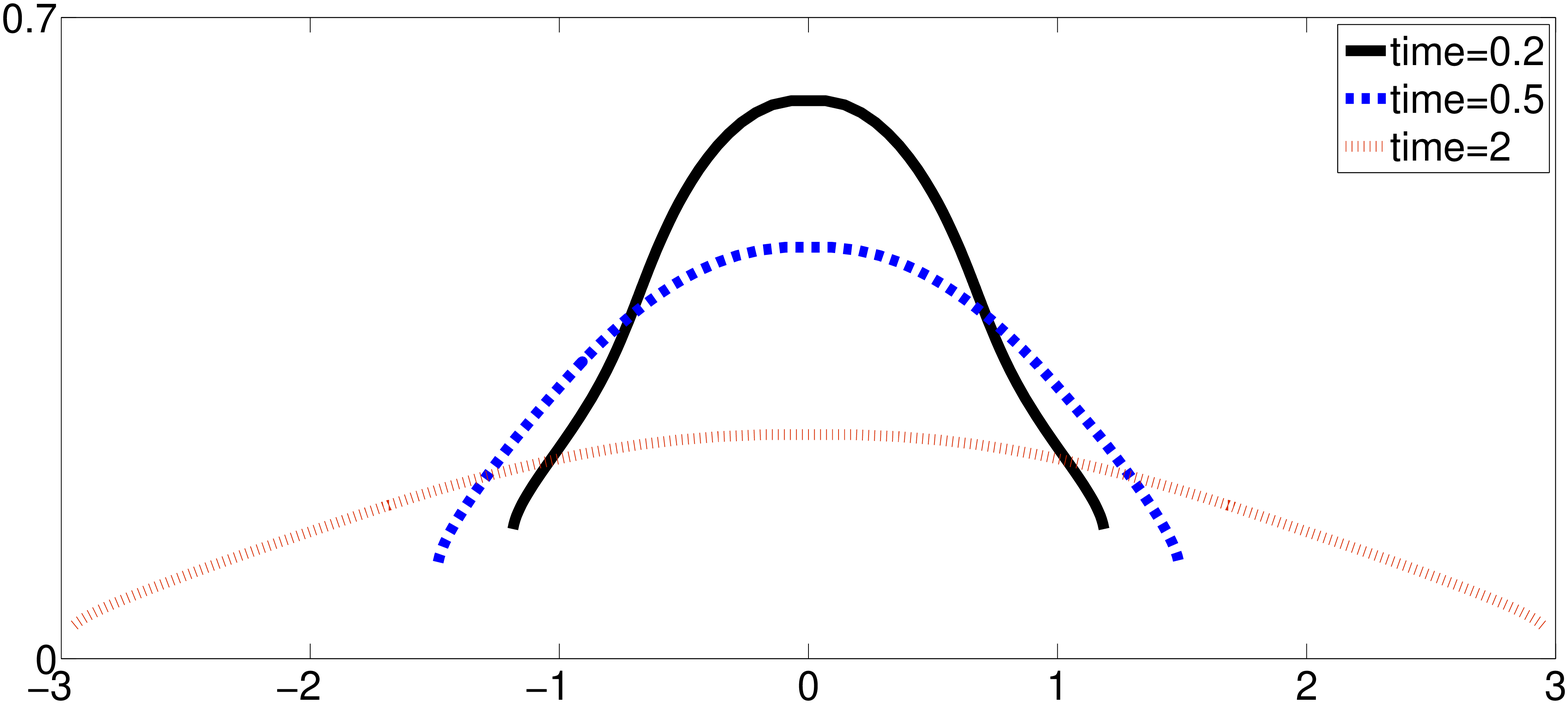}\includegraphics[width=8.5cm]{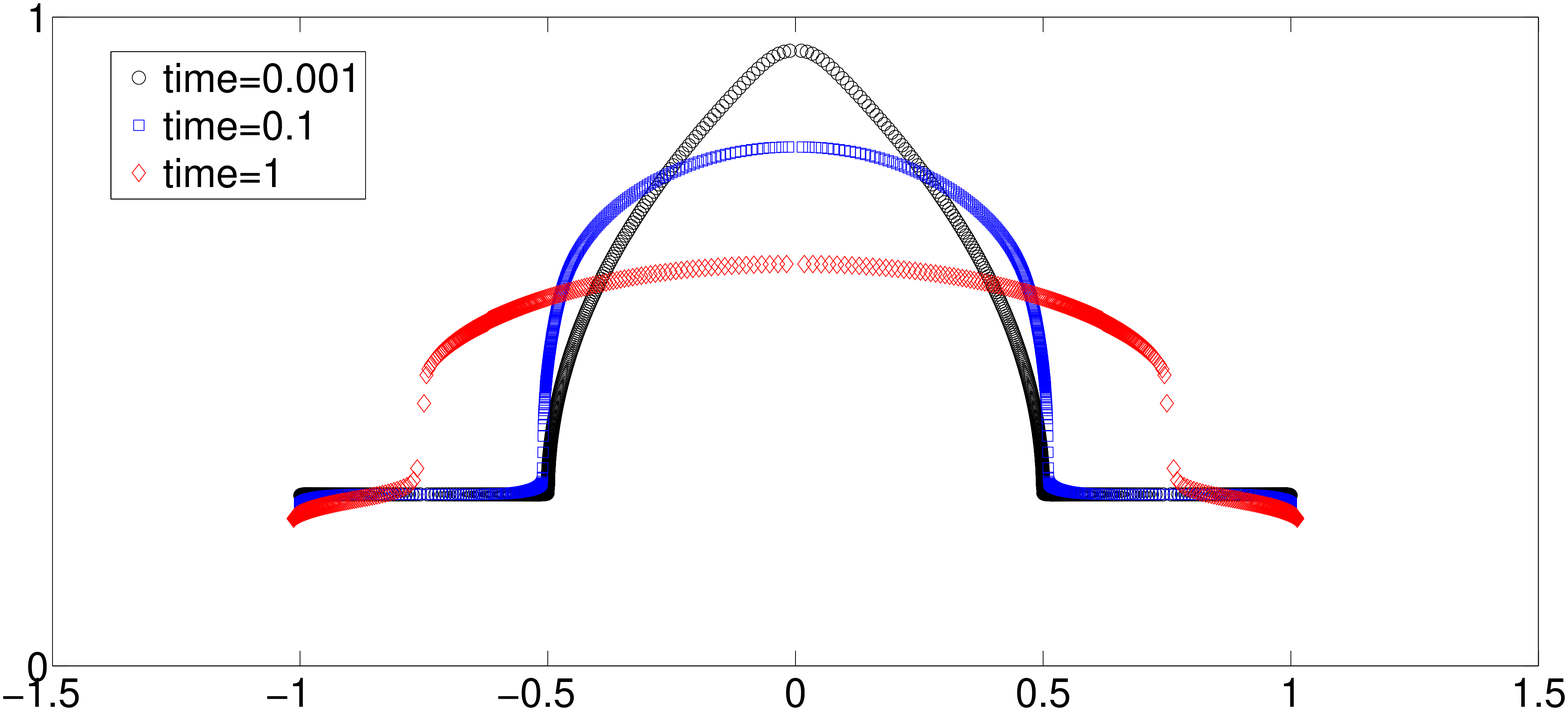}
\caption{Evolution of solutions corresponding to $u_0$. Top left:
Initial datum $u_0$. Top right: Evolution for $m=1$ at small
times. Bottom left: Evolution for $m=1$ for larger times. Bottom
right: Evolution for $m=4$.} \label{fig.disbulk}
\end{figure}

In order to show this behavior we have taken two types of initial
datum with $N=1000$:
$$
u_0(x):=\frac{1}{4}\chi_{[-1,1]}+\frac{3}{2\sqrt{2}}\sqrt{\frac{1}{2}-|x|}\chi_{[-\frac{1}{2},\frac{1}{2}]}
\qquad \mbox{ and } \qquad \tilde
u_0(x):=\frac{1}{4}\chi_{[-1,-\frac12]\cup[\frac12,1]}+\frac{3}{4}\chi_{]-\frac12,\frac12[}\,.
$$
We imposed a high concentration of nodes around the vertical
angles or discontinuities (i.e. $x=\pm \frac{1}{2}$).

\begin{figure}[ht]
  \includegraphics[width=8.5cm]{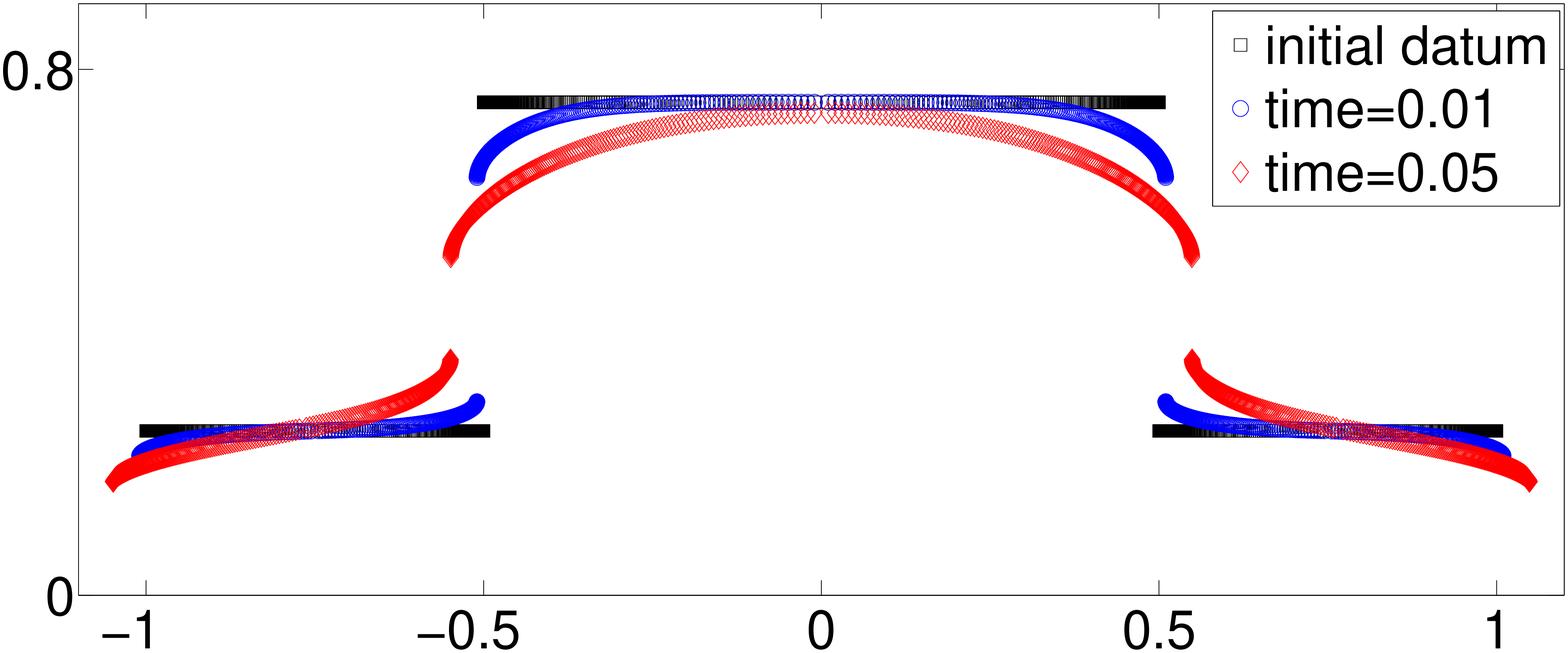}\includegraphics[width=8.5cm]{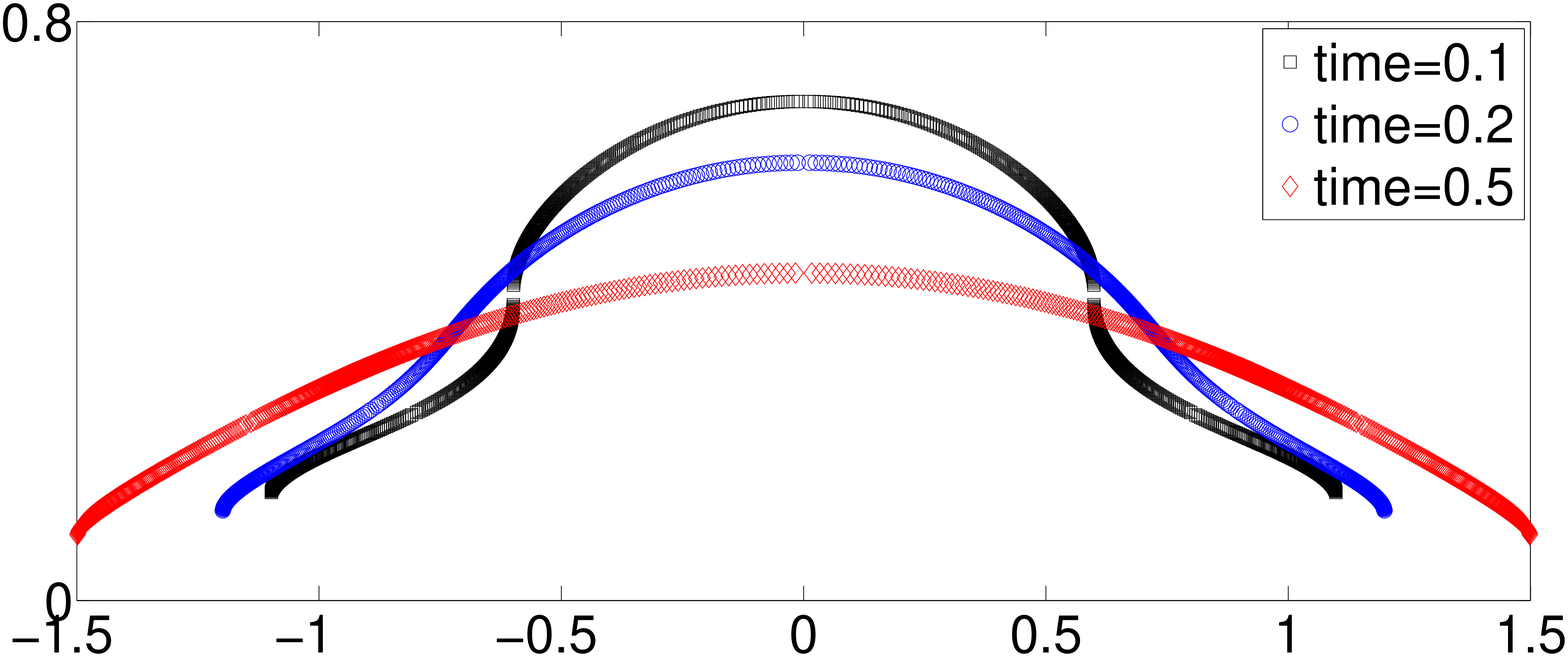}
\includegraphics[width=8.5cm]{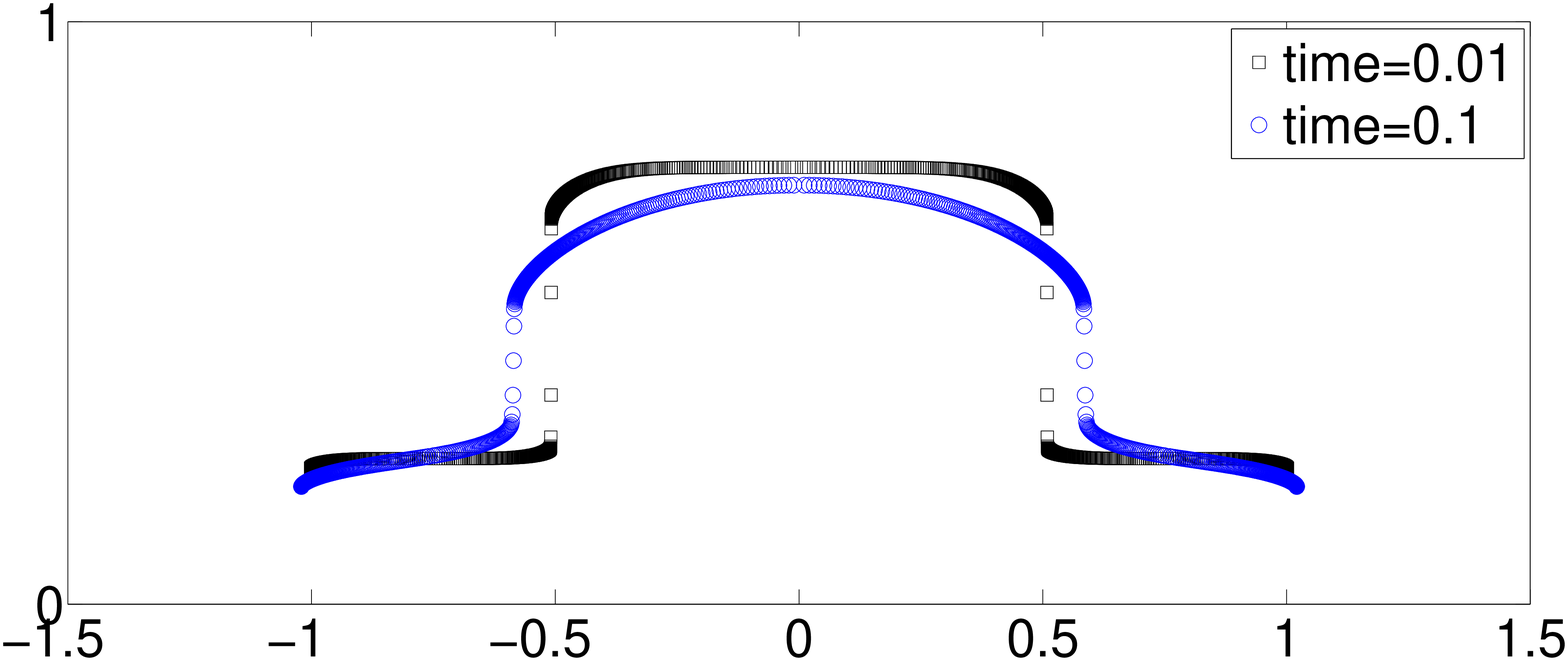}\includegraphics[width=8.5cm]{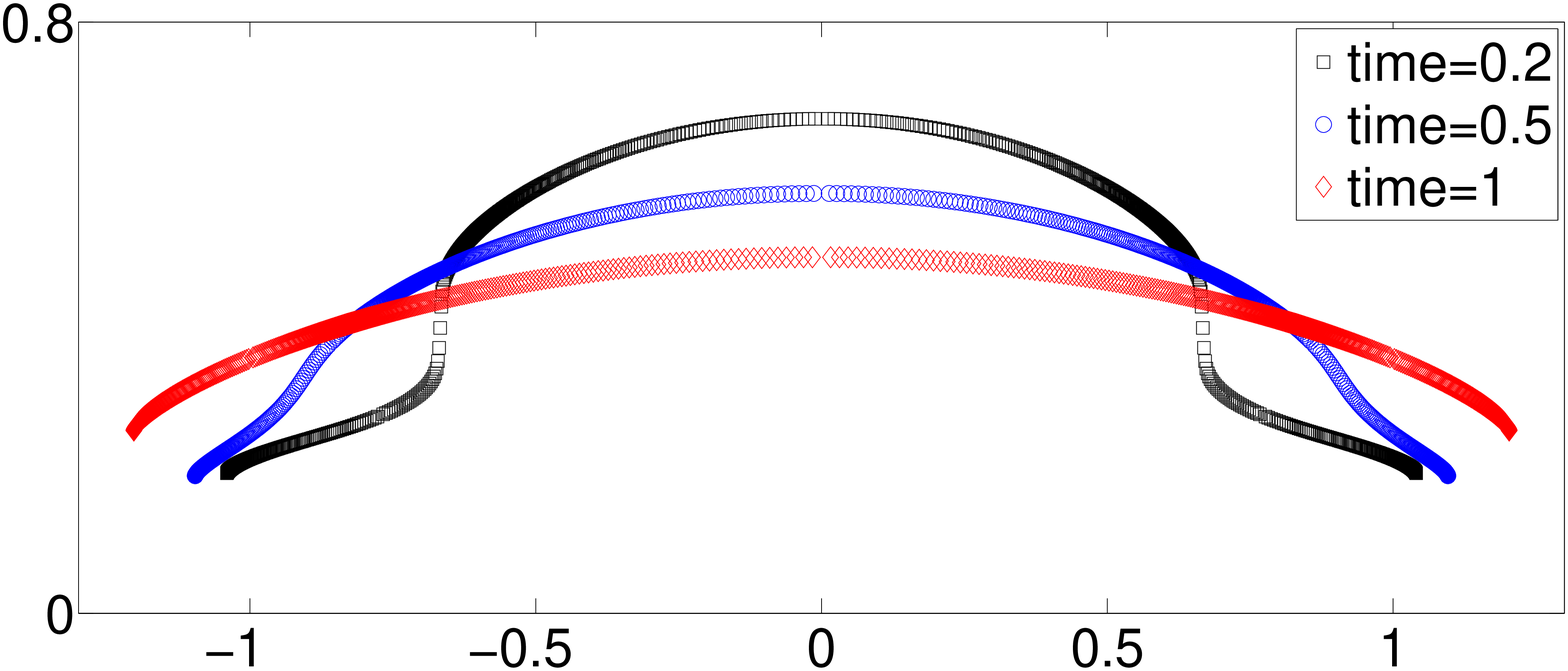}
\caption{Evolution of solutions for $\tilde u_0$. Top: Evolution
for $m=1$ at different times. Bottom: Evolution for $m=2$ at
different times.} \label{fig.disbulk2}
\end{figure}

In Fig. \ref{fig.disbulk} we observe the evolution of the
solutions corresponding to the initial datum $u_0$, demonstrating
the above heuristics. In Fig. \ref{fig.disbulk2}, we see how an
initially discontinuous initial datum $\tilde u_0$ is smoothed
during the evolution both for $m=1$ (as heuristically deduced
before) and for $m=2$. We observe that the smoothing of the
discontinuity is slower with $m>1$ than that of $m=1$.

\subsection{Asymptotic behavior}

In this Section, guided by heuristics, we numerically observe the
asymptotic behavior of solutions to \eqref{pmrel} and the rate of
convergence towards their asymptotic steady state, for which no
result is available in the literature. Performing the classical
self-similar change of variables \cite{carrtosc:pm:00} that
translates porous medium equation onto nonlinear Fokker-Planck
equations given by
\begin{equation}\label{change}
v(x,t)=e^{t}u(e^t x, k(e^{\frac{t}{k}}-1))\,,
\end{equation}
with $k=\frac{1}{m+1}$, then equation \eqref{pmrel} transforms
into
\begin{equation}
  \label{pmefpt} v_t={\rm div }\left(xv+ \frac{v^m\nabla v}{\sqrt{v^2+e^{-2t}|\nabla v|^2}}\right)\,.
\end{equation}
Therefore, formally, when $t\to\infty$ solutions of \eqref{pmefpt}
should converge to a stationary solution of $v_t={\rm
div}\left(xv+v^{m-1}\nabla v\right)$, i.e., to a Gaussian $V(x)$
for $m=1$ or to the corresponding Barenblatt solution $V_m(x)$
when $m>1$ given by
$$
 V(x)=\frac{1}{\sqrt{2\pi}}e^{-\frac{x^2}{2}}
\qquad \mbox{and} \qquad
V_m(x)=\left(C_m-\frac{m-1}{2}x^2\right)_+^\frac{1}{m-1}\,,
$$
where $C_m$ is uniquely determined by the conservation of mass. In
the original variables, then solutions should converge to the
corresponding self-similar profiles obtained from $V$ and $V_m$
via the change of variables \eqref{change} except time
translations. To be precise, the self-similar solutions are given
by
$$
U(|x|,t)=\frac{e^{-\frac{x^2}{4t}}}{\sqrt{4\pi t}} \, \mbox{ for }
m=1 \mbox{ and }\,\, U_m(|x|,t)=t^{\frac{-1}{m+1}}\left(\tilde
C_m-\frac{m-1}{2m(m+1)}|x|^2t^{\frac{-2}{m+1}}\right)_+^\frac{1}{m-1}
\mbox{ for } m> 1,
$$
where $\tilde C_m$ is determined as above.

\begin{figure}[ht]
\includegraphics[width=8.5cm]{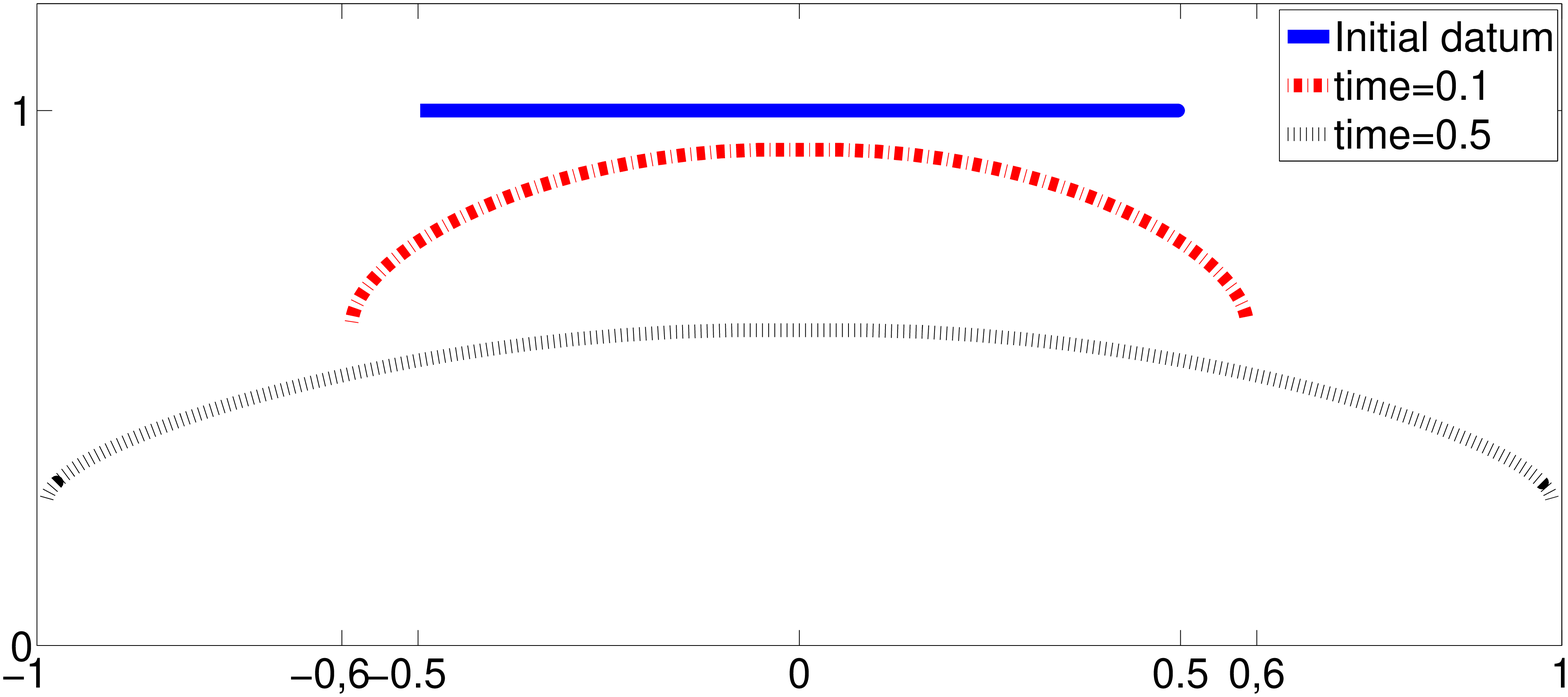}
\includegraphics[width=8.5cm]{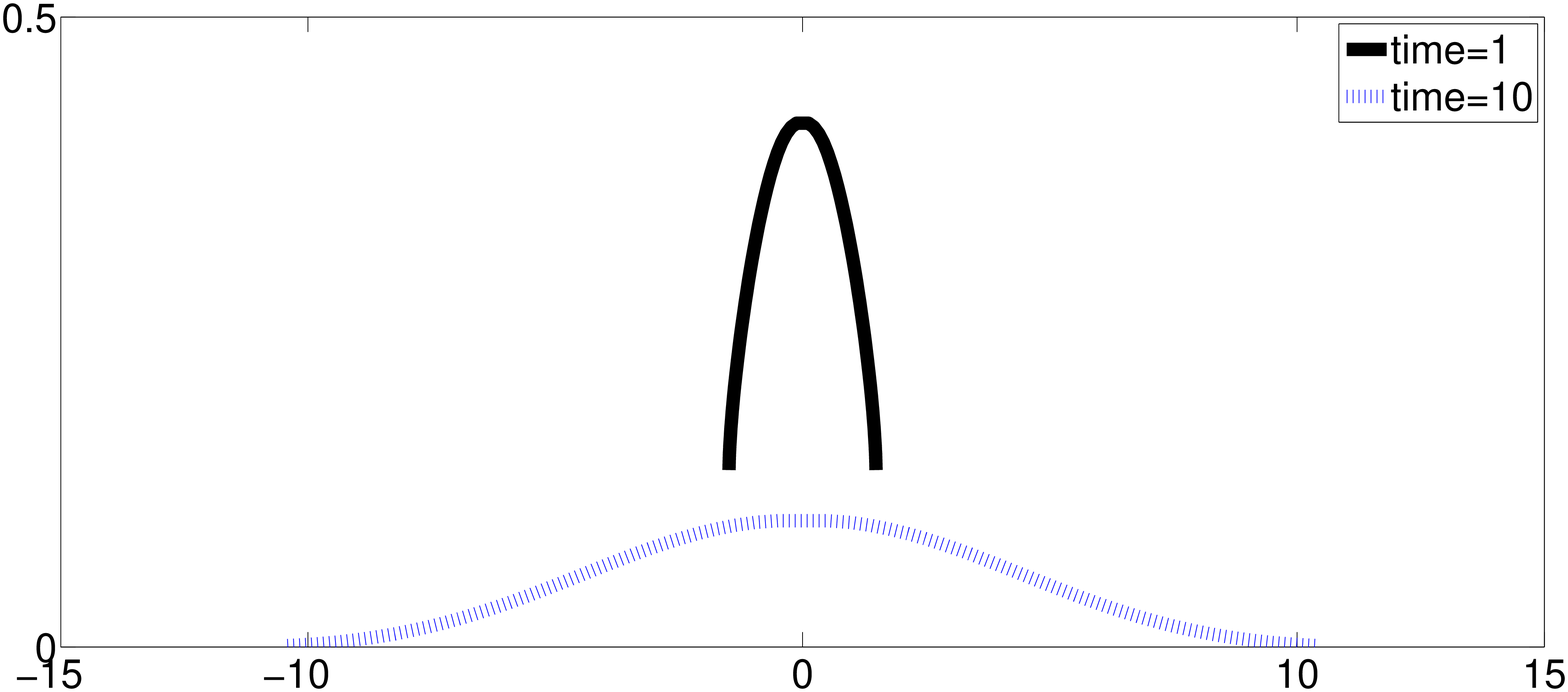}
\includegraphics[width=8.5cm]{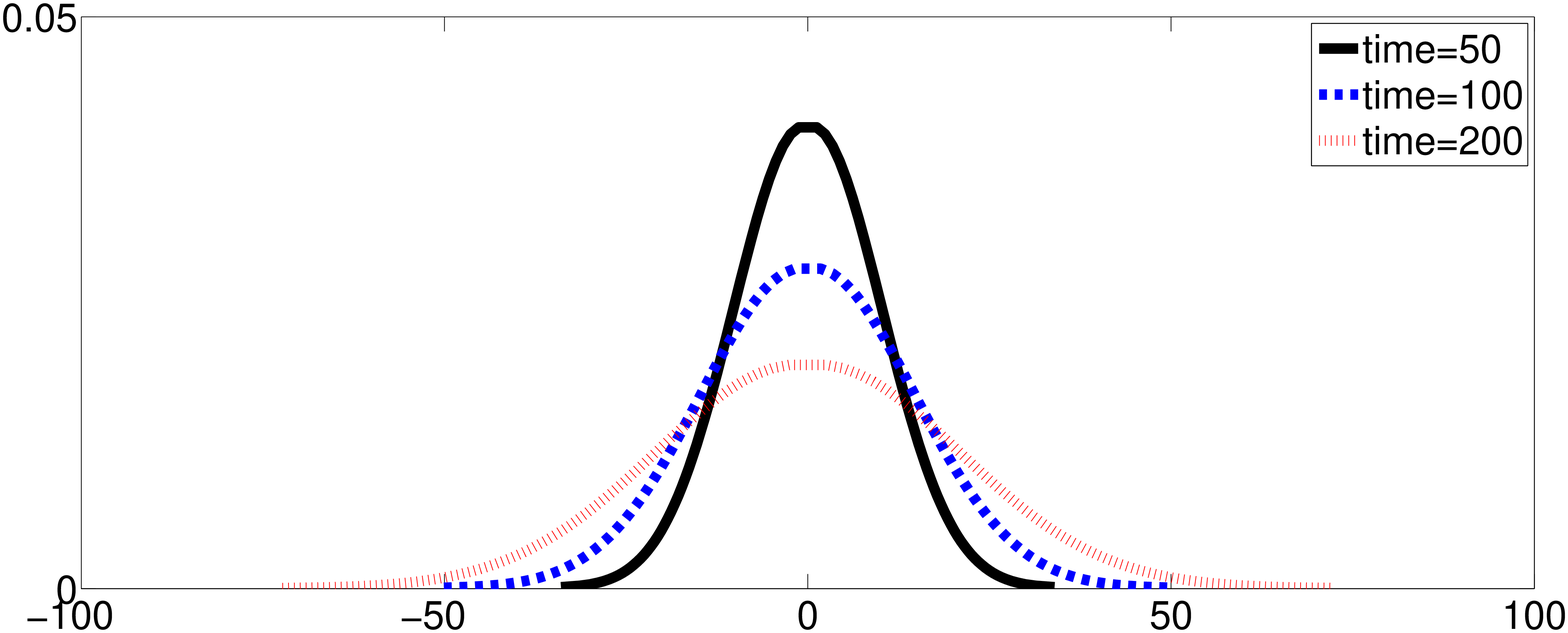}
\includegraphics[width=8.5cm]{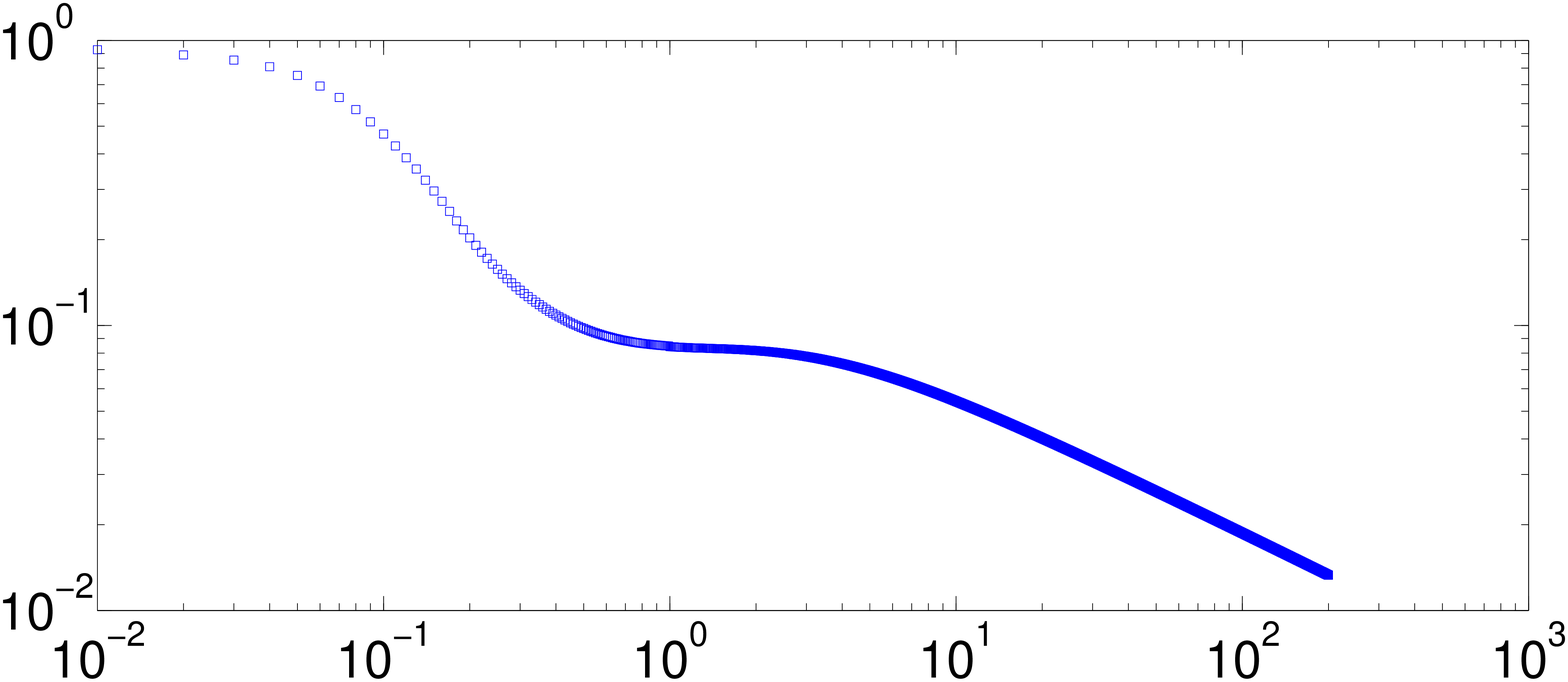}
\caption{Top and Bottom left: Evolution of $u_0$ in case $m=1$ at
different times with $N=100$. Bottom right: log-log plot of the
estimate $\|u(t)-U(t)\|_1$ with $N=1000$.} \label{convgauss}
\end{figure}

In the following computations we have taken
$u_0:=\chi_{[-\frac{1}{2},\frac{1}{2}]}$, $N=100$. We plot the
evolution of the initial datum for different values of $m$ and an
estimate of the difference of $u-U_m$ in the $L^1$-norm. More
precisely, we took $\|u(t)-U_m(t)\|_1:=\frac{1}{N}\sum_{i=1}^N
|u(x_i,t)-U_m(x_i,t)|$.

Some comments are in order. First of all, in Figure
\ref{convgauss}, we note that for $m=1$, while time is small, the
numerical solution satisfies both the linear propagation of the
support property, as well as the vertical contact angle property.
However, for larger times, these two conditions are lost during
the computation. This is due to the fact that we took a fixed
number of nodes ($N=100$), and as time increases, this number of
nodes is clearly insufficient. We have observed that by increasing
the number of nodes (for instance to $N=1000$) the time in which
the numerical solution is more accurate increases. We can also see
in Figure \ref{convgauss}, that, in spite of this, the numerical
solution tends to a Gaussian with an algebraic rate of convergence
that seems to be $\tfrac12$ the one of the heat equation. However,
it is exactly by the same reason as before that when time
increases, the rate of convergence degenerates. For this reason,
we have included in Figure \ref{convgauss} the $L^1$-convergence
rate with $N=1000$.

\begin{figure}[ht]
\includegraphics[width=8.5cm]{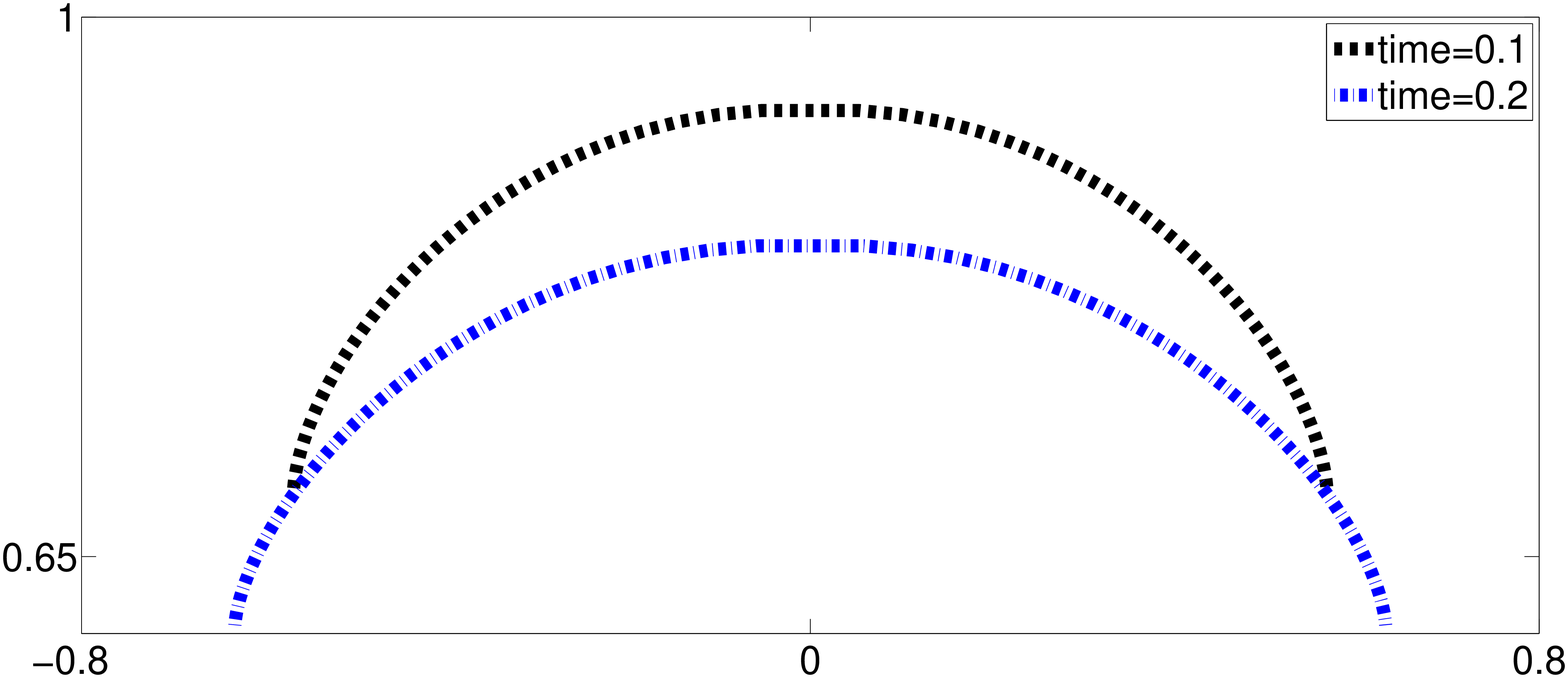}
\includegraphics[width=8.5cm]{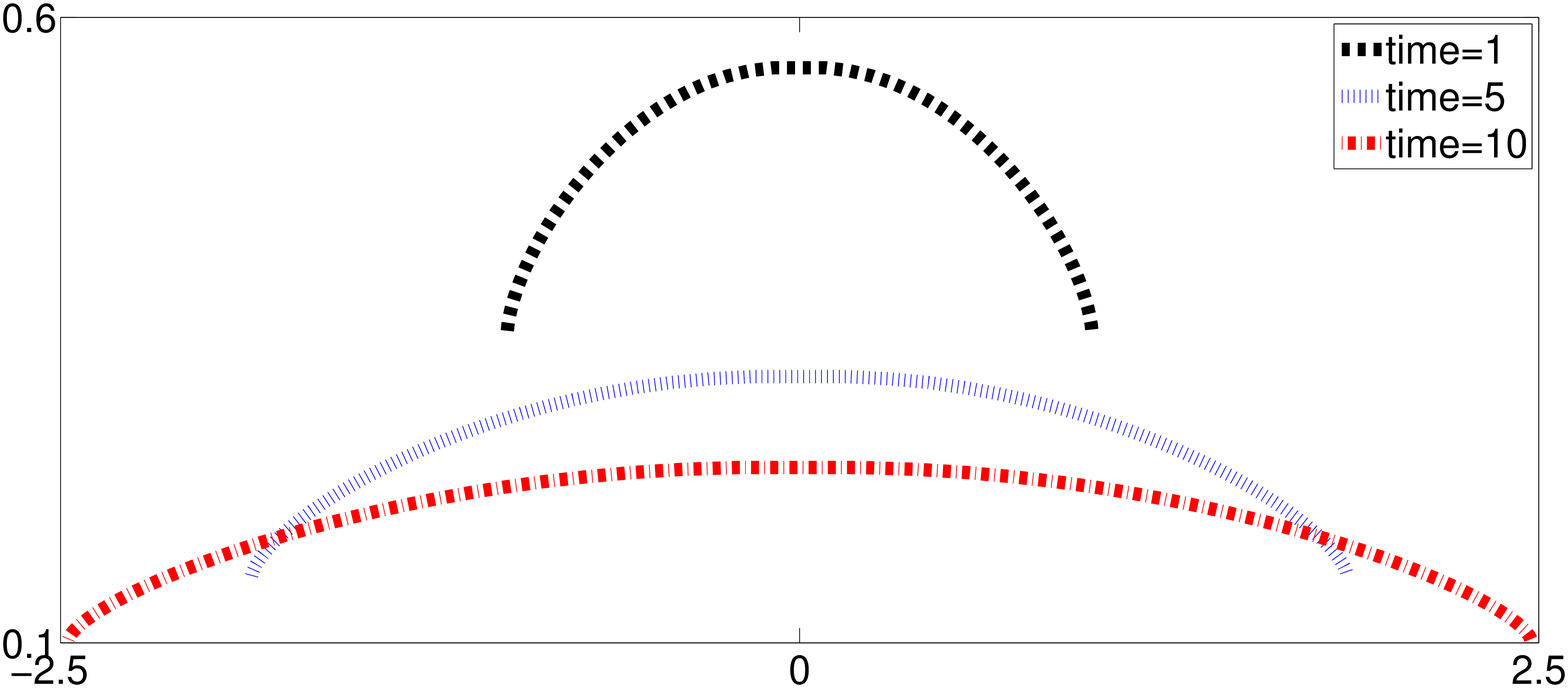}
\includegraphics[width=8.5cm]{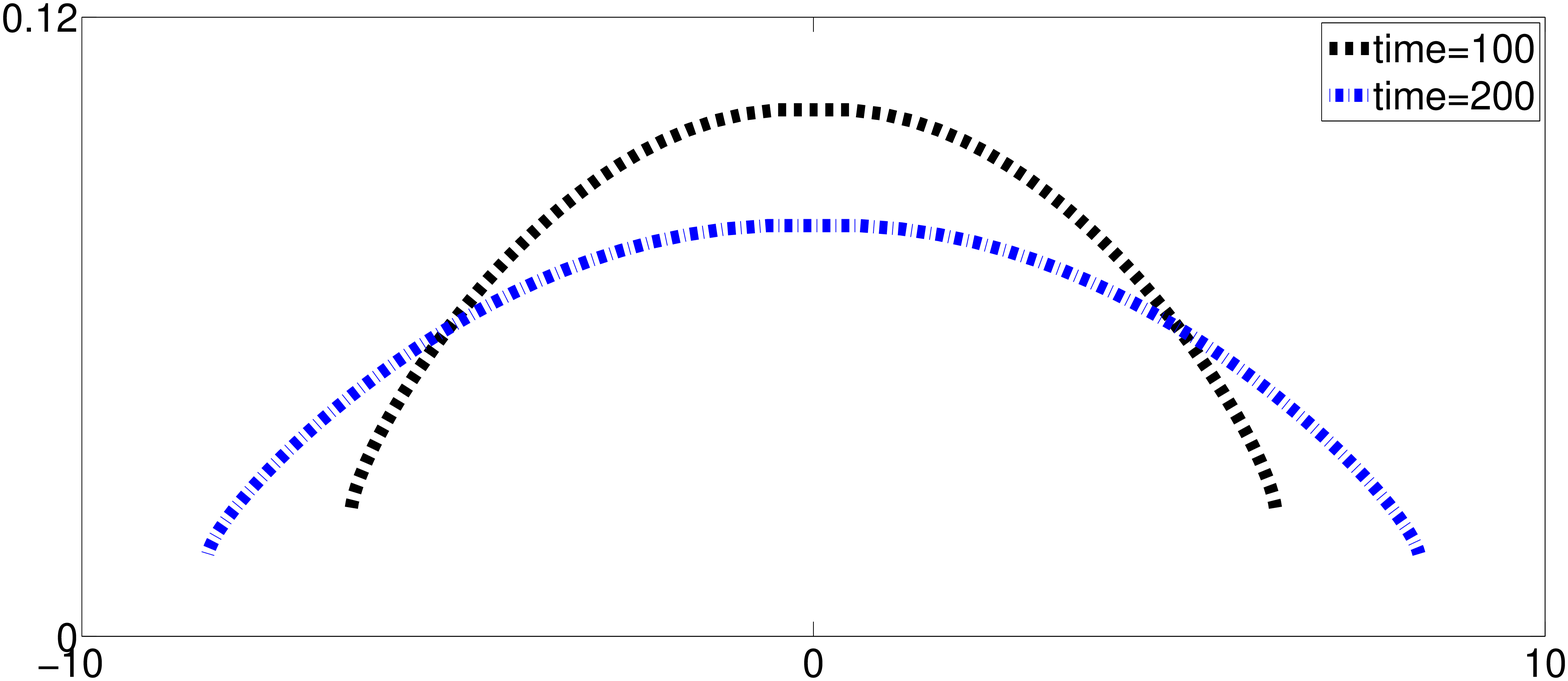}
\includegraphics[width=8.5cm]{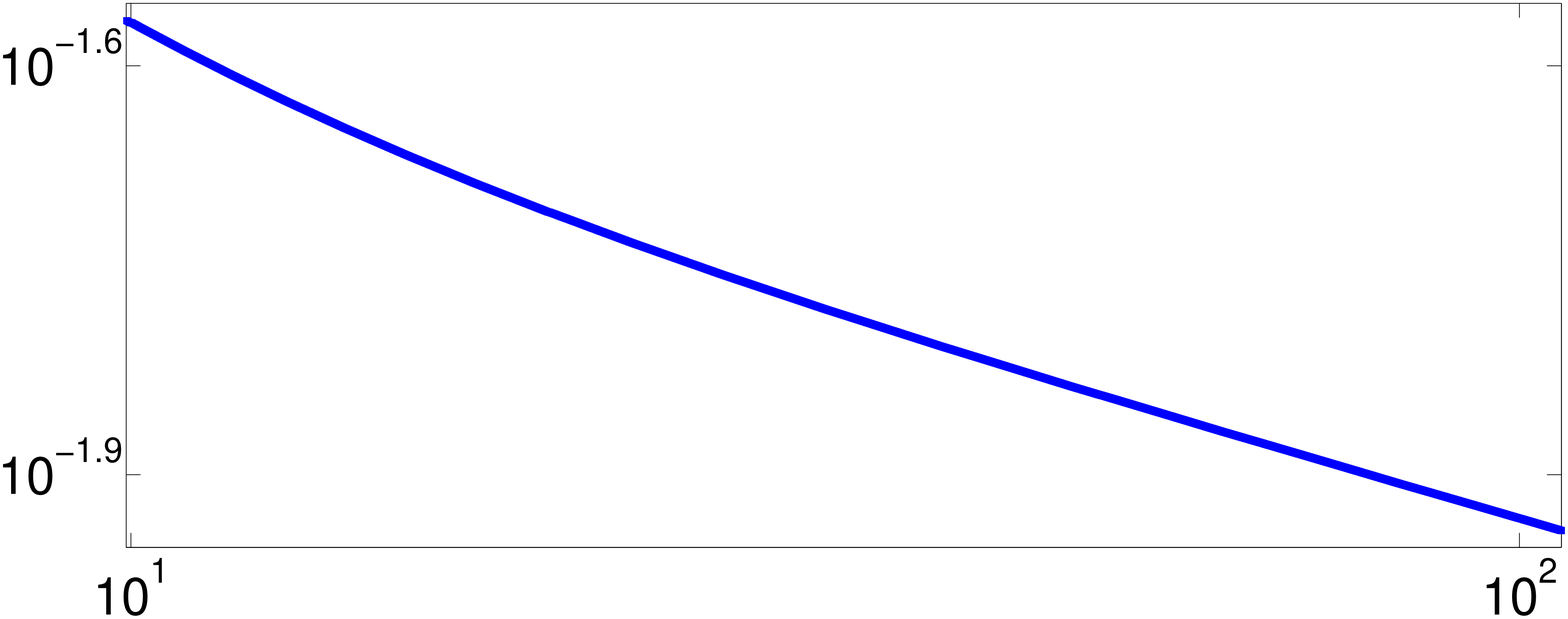}
\caption{Top and Bottom left: Evolution of $u_0$ in case $m=2$ at
different times. Bottom right: log-log plot of the estimate
$\|u(t)-U_2(t)\|_1$.} \label{convm2}
\end{figure}

Instead, when $m>1$, the support of the solution does not
propagate so fast and we can observe in Figures \ref{convm2} and
\ref{convm10} how the vertical contact angle property is preserved
even for large times. Moreover, in Figures \ref{convm2} and
\ref{convm10} we can see how the numerical solution tends to $U_m$
for $m=2$ and $m=10$. In both cases, the rate of convergence is
algebraic and, numerically, it is surprisingly seen that it might
correspond to  $\frac{1}{3}$ in the first case and to
$\frac{1}{11}$ in the second one; i.e.: the same convergence rate
as for the porous medium equation, see
\cite{carrtosc:pm:00,Otto01}.

\begin{figure}[ht]
\includegraphics[width=8.5cm]{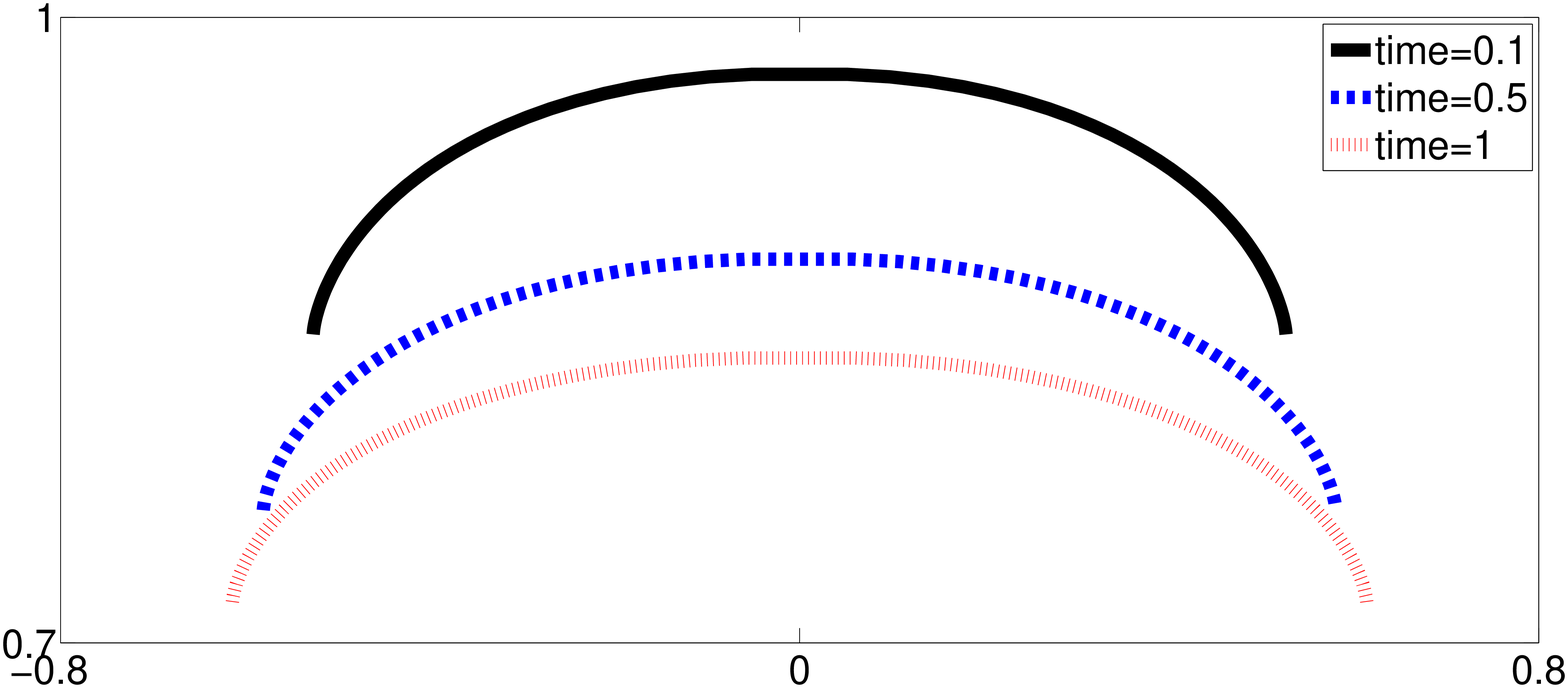}
\includegraphics[width=8.5cm]{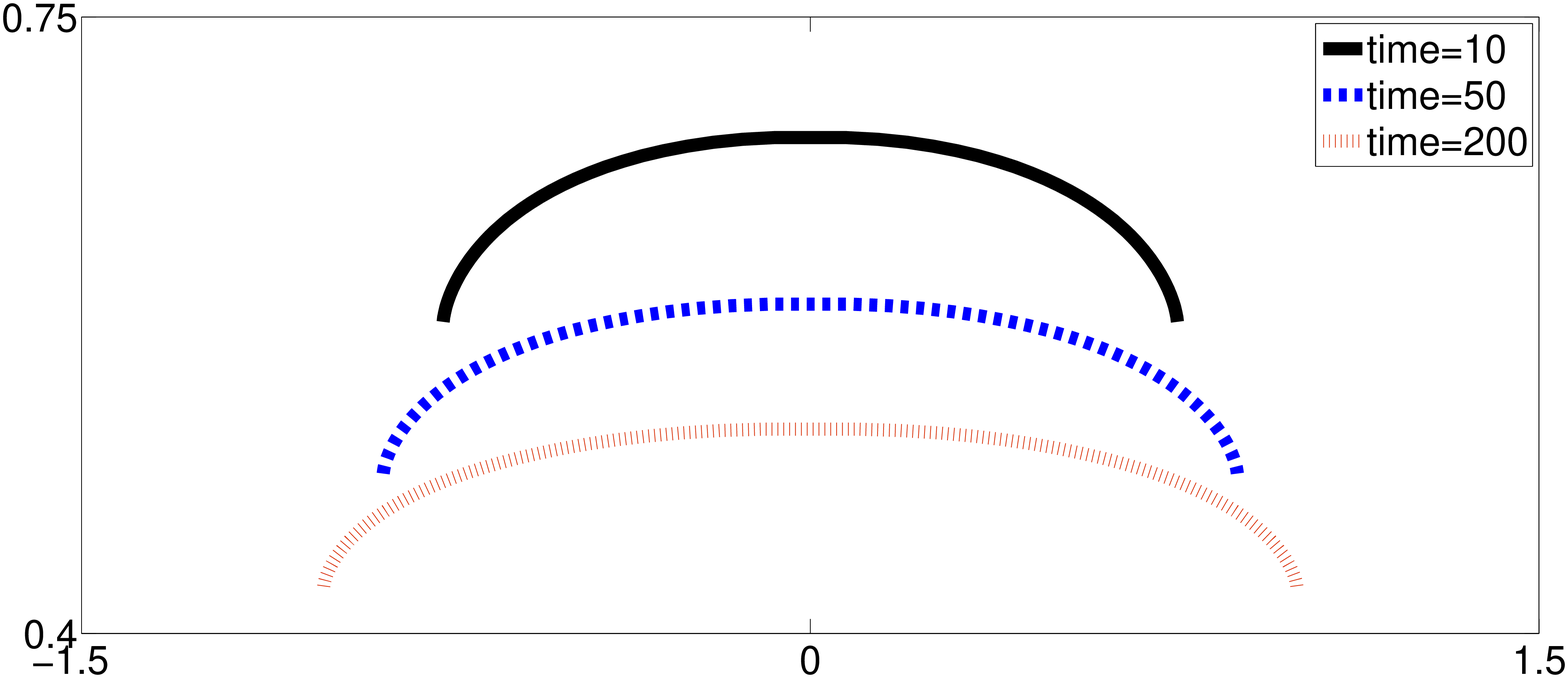}
\includegraphics[width=8.5cm]{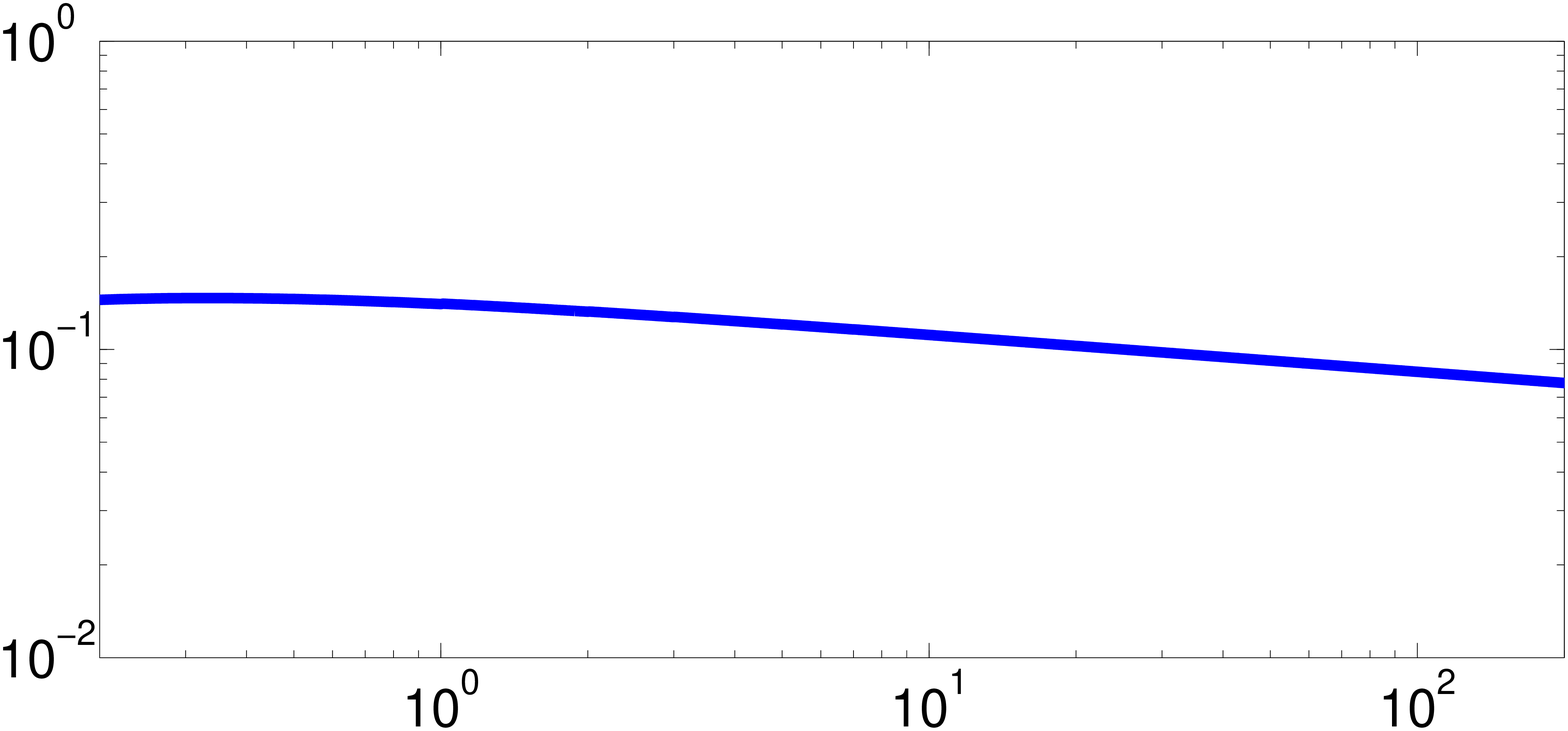}
\includegraphics[width=8.5cm]{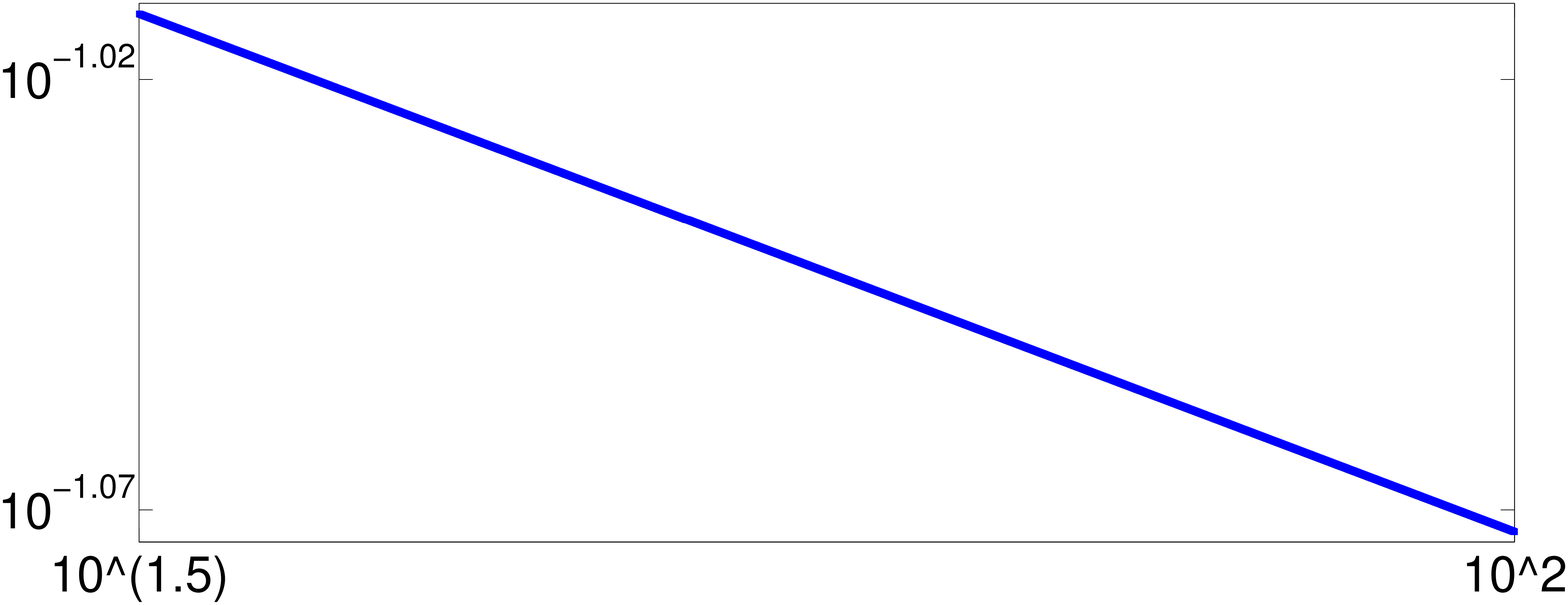}
\caption{Top: Evolution of $u_0$ in case $m=10$ at different
times. Bottom left: log-log plot of the estimate
$\|u(t)-U_{10}(t)\|_1$. Bottom right: zoom of the final time
interval} \label{convm10}
\end{figure}

\subsection{Convergence toward the homogeneous relativistic heat
equation}

We finally show numerically how solutions to \eqref{RHEg},
converge to solutions of the homogeneous relativistic heat
equation
$$
\left\{\begin{array}
  {cc}\displaystyle u_t=\left(u\frac{ u_x}{|u_x|}\right)_x & {\rm in \ } \R^N\times [0,T] \\ u_0(x)=u_0 & {\rm in \ } \R^N
\end{array}\right.
$$
when the kinematic viscosity $\nu \to +\infty$ as already proved
in \cite{ACMRelatTM}. In Fig. \ref{fig.convnuinfty} we estimate
the evolution in time of the difference in the $L^1$-norm for
solutions corresponding to the initial data
$u_0=\chi_{[-\frac{1}{2},\frac{1}{2}]}$ for different values of
$\nu$ with respect to the explicit solution $u_{hom}$, given by
$$
u_{hom}(x,t)=\frac{1}{1+2t}\chi_{[-\frac{1}{2}-t,\frac{1}{2}+t]}
$$
when $\nu\to\infty$.

\begin{figure}[ht]
\includegraphics[width=8.5cm]{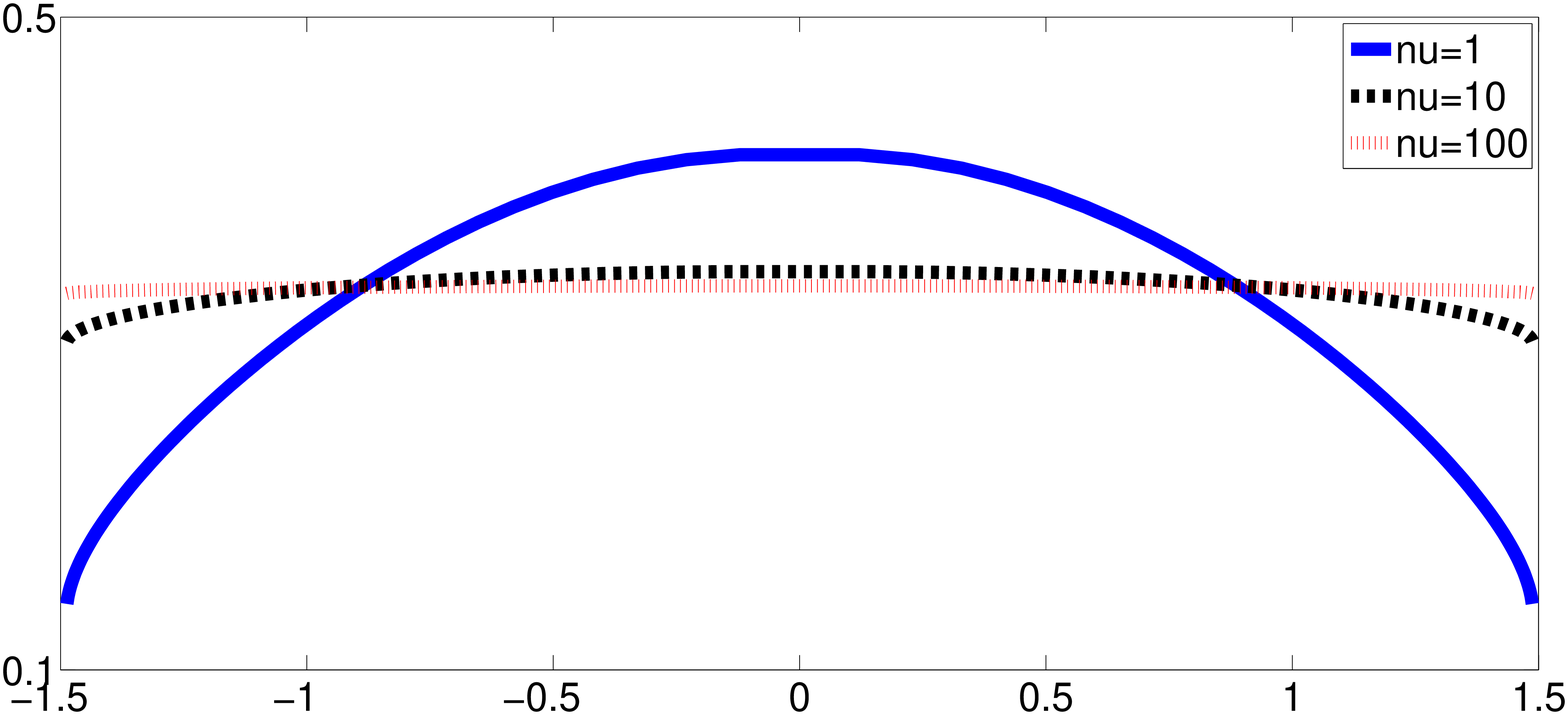}\includegraphics[width=8.5cm]{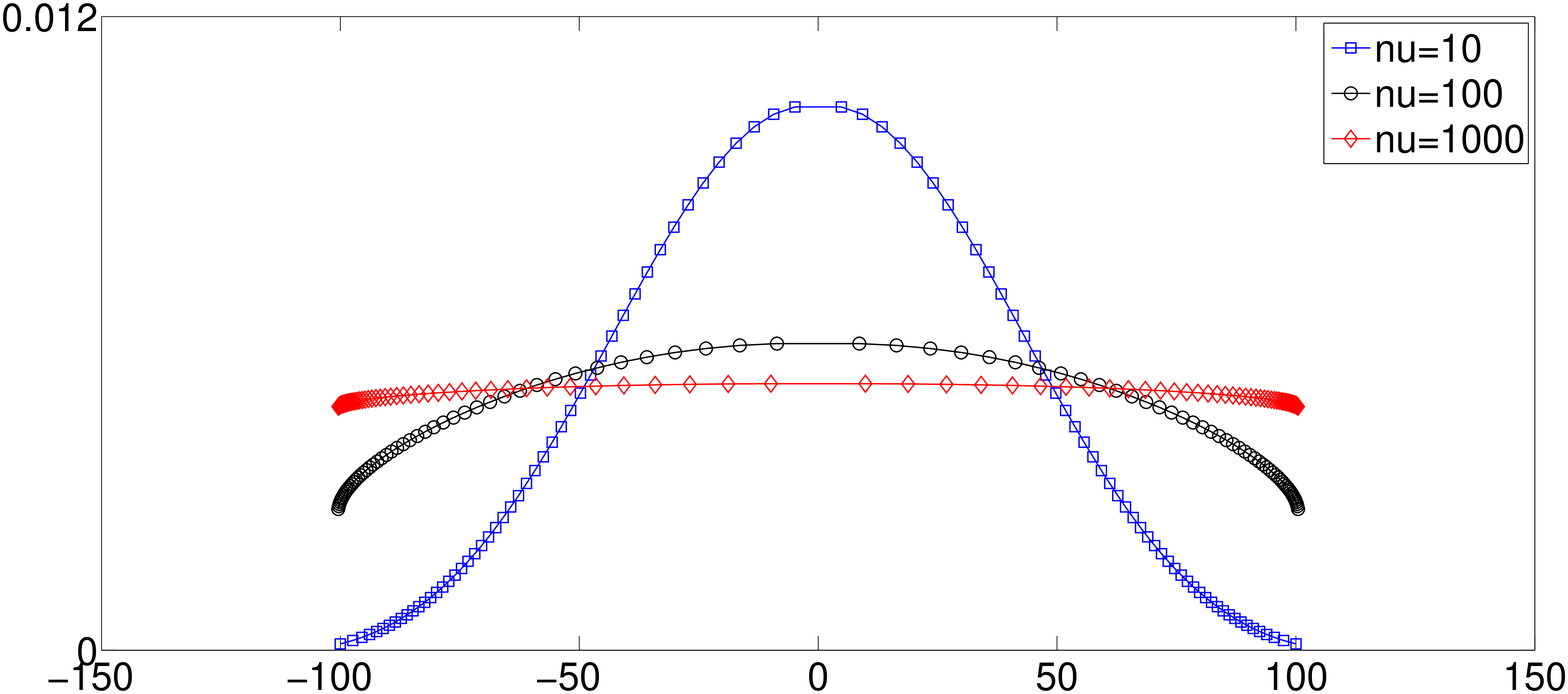}
\centering\includegraphics[width=8.5cm]{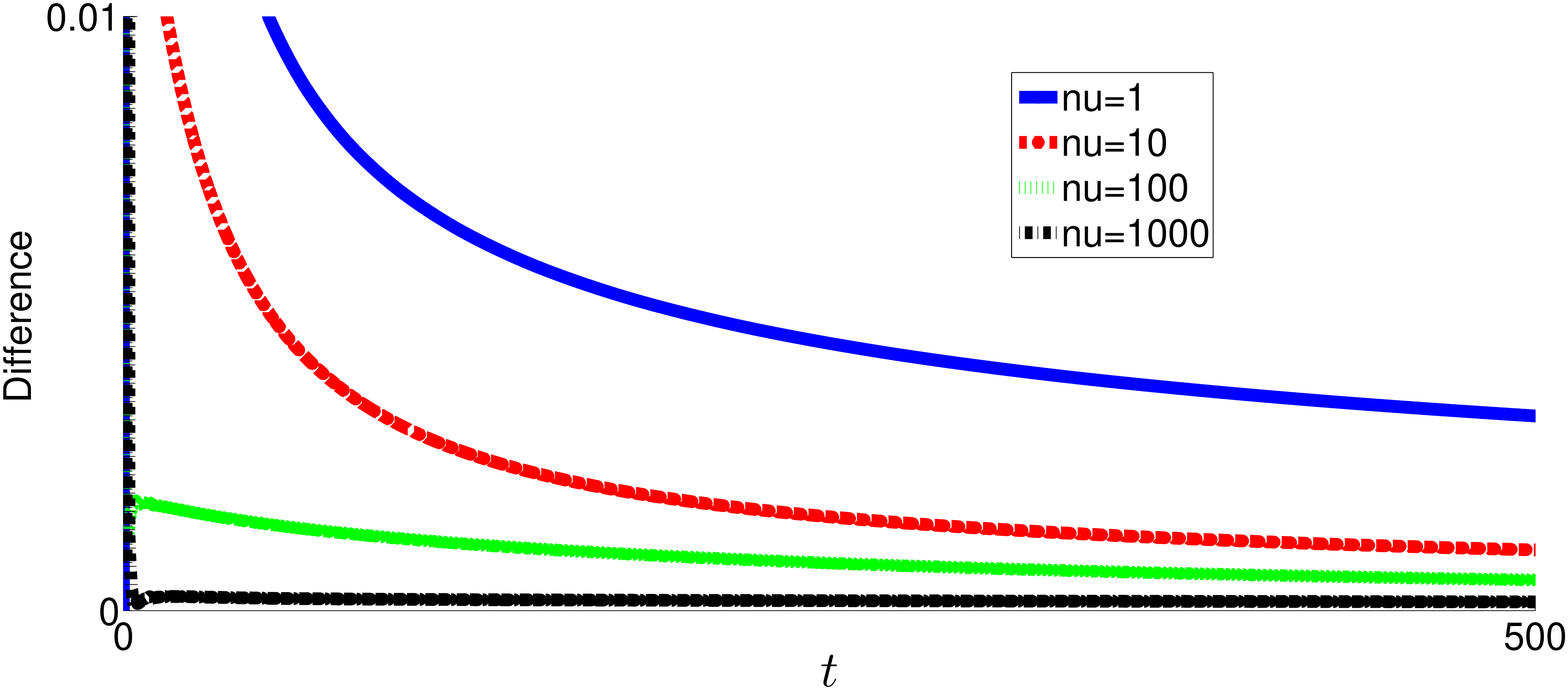}
\caption{Top left: Numerical solution at $t=1$ for different
values of $\nu$. Top Right: Numerical solution at $t=100$ for
different values of $\nu$. Bottom: Evolution of the
$L^1$-difference with respect to $u_{hom}$.}
\label{fig.convnuinfty}
\end{figure}


\setcounter{equation}{0} \setcounter{subsection}{0}
\renewcommand{\theequation}{A.\arabic{equation}}
\renewcommand{\thesubsection}{A.\arabic{subsection}}

\section*{Appendix: A primer on Entropy Solutions}
\label{sect:appendix}

We collect in this Appendix some definitions that are needed to
work with entropy solutions of flux limited diffusion equations.

Note that the equation (\ref{RHE}) can be written as
\begin{equation} \label{DirichletproblemP}
u_t = \b(u,  u_x)_x, \qquad \hbox{in \hspace{0.2cm} $Q_T=(0,T)\times \R$}
\end{equation}
where $\b(z,\xi) = \nabla_\xi f(z,\xi)$ and
\begin{equation}\label{funct:frhe}
f(z,\xi) = z  \sqrt{z^2 + \vert \xi\vert^2}.
\end{equation}
As usual, we define
\begin{equation}\label{funct:arhe}
h(z,\xi) = \b(z,\xi)\cdot \xi = \frac{z \vert
\xi\vert^2}{\sqrt{z^2 + \vert \xi\vert^2}}.
\end{equation}
Note that $f$ is convex in $\xi$ and both $f,h$ have linear growth as $|\xi|\to\infty$.

\subsection{Functions of bounded variation and some generalizations}\label{sect:bv}

Denote by ${\mathcal L}^N$ and ${\mathcal H}^{N-1}$ the
$N$-dimensional Lebesgue measure and the $(N-1)$-dimensional
Hausdorff measure in $\R^N$, respectively. Given an open set
$\Omega$ in $\R^N$  we denote by ${\mathcal D}(\Omega)$  the space
of infinitely differentiable functions with compact support in
$\Omega$. The space of continuous functions with compact support
in $\R^N$ will be denoted by $C_c(\R^N)$.

Recall that if $\Omega$ is an open subset of $\R^N$, a function $u
\in L^1(\Omega)$ whose gradient $Du$ in the sense of distributions
is a vector valued Radon measure with finite total variation in
$\Omega$ is called a {\it function of bounded variation}. The
class of such functions will be denoted by $BV(\Omega)$.  For $u
\in BV(\Omega)$, the vector measure $Du$ decomposes into its
absolutely continuous and singular parts $Du = D^{ac} u + D^s u$.
Then $D^{ac} u = \nabla u \ \L^N$, where $\nabla u$ is the
Radon--Nikodym derivative of the measure $Du$ with respect to the
Lebesgue measure $\L^N$. We also split $D^su$ in two parts: the
{\it jump} part $D^j u$ and the {\it Cantor} part $D^c u$. It is
well known (see for instance \cite{Ambrosio}) that $$D^j u = (u^+
- u^-) \nu_u \H^{N-1} \res J_u,$$ where $u^+(x),u^-(x)$ denote the
upper and lower approximate limits of $u$ at $x$, $J_u$ denotes
the set of approximate jump points of $u$ (i.e. points $x\in
\Omega$ for which $u^+(x)> u^-(x)$), and $\nu_u(x) =
\frac{Du}{\vert D u \vert}(x)$, being $\frac{Du}{\vert D u \vert}$
the Radon--Nikodym derivative of $Du$ with respect to its total
variation $\vert D u \vert$. For further information concerning
functions of bounded variation we refer to \cite{Ambrosio}.

We need to consider the following truncation functions. For $a <
b$, let $T_{a,b}(r) := \max(\min(b,r),a)$, $ T_{a,b}^l =  T_{a,b}-l$.
We denote
$$
\mathcal T_r:= \{ T_{a,b} \ : \ 0 < a < b \},  \ \ \
$$
$$
\mathcal{T}^+:= \{ T_{a,b}^l \ : \ 0 < a < b ,\, l\in \R, \, T_{a,b}^l\geq 0 \}.  \ \ \
$$

Given any function $w$ and $a,b\in\R$ we shall use the notation
$\{w\geq a\} = \{x\in \R^N: w(x)\geq a\}$, $\{a \leq w\leq b\} =
\{x\in \R^N: a \leq w(x)\leq b\}$, and similarly for the sets $\{w
> a\}$, $\{w \leq a\}$, $\{w < a\}$, etc.

We need to consider the following function space
$$
TBV_{\rm r}^+(\R^N):= \left\{ w \in L^1(\R^N)^+  \ :  \ \ T_{a,b}(w) - a \in BV(\R^N), \
\ \forall \ T_{a,b} \in \mathcal T_r \right\}.
$$
Notice that $TBV_{\rm r}^+(\R^N)$ is closely related to the space
$GBV(\R^N)$  of generalized functions of bounded variation
introduced by E. Di Giorgi and L. Ambrosio in \cite{Ambrosio}.
Using the chain rule for BV-functions (see for instance
\cite{Ambrosio}), one can give a sense to $\nabla u$ for a
function $u \in TBV^+(\R^N)$ as the unique function $v$ which
satisfies
\begin{equation*}\label{E1WRN}
\nabla T_{a,b}(u) = v \1_{\{a < u  < b\}} \ \ \ \ \ {\mathcal
L}^N-{\rm a.e.}, \ \ \forall \ T_{a,b} \in \mathcal{T}_r.
\end{equation*}
We refer to \cite{Ambrosio} for details.

\subsection{Functionals defined on BV}\label{sect:functionalcalculus}

In order to define the notion of entropy solutions of
(\ref{DirichletproblemP}) and give a characterization of them, we
need a functional calculus defined on functions whose truncations
are in $BV$.

Let $\Omega$ be an open subset of $\R^N$. Let $g: \Omega \times \R
\times \R^N \rightarrow [0, \infty[$ be a Borel function such that
\begin{equation*}\label{LGRWTH}
C(x) \vert \zeta \vert - D(x) \leq g(x, z, \zeta)  \leq M'(x) + M
\vert \zeta \vert
\end{equation*}
for any $(x, z, \zeta) \in \Omega \times \R \times \R^N$, $\vert
z\vert \leq R$, and any $R>0$, where $M$ is a positive constant
and  $C,D,M' \geq 0$ are bounded Borel functions which may depend
on $R$. Assume that $C,D,M' \in L^1(\Omega)$.

Following Dal Maso \cite{Dalmaso} we consider the
functional:
\begin{eqnarray*}\label{RelEnerg}
{\mathcal R}_g(u)&:=& \displaystyle\int_{\Omega} g(x,u(x), \nabla
u(x)) \, dx + \int_{\Omega} g^0 \left(x,
\tilde{u}(x),\frac{Du}{\vert D u \vert}(x) \right) \,  d\vert D^c
u \vert
\nonumber \\
&&+ \displaystyle\int_{J_u} \left(\int_{u_-(x)}^{u_+(x)}
g^0(x, s, \nu_u(x)) \, ds \right)\, d \H^{N-1}(x),
\end{eqnarray*}
for $u \in BV(\Omega) \cap L^\infty(\Omega)$, being $\tilde{u}$ is
the approximated limit of $u$ \cite{Ambrosio}. The recession
function $g^0$ of $g$ is defined by
\begin{equation*}\label{Asimptfunct}
 g^0(x, z, \zeta) = \lim_{t \to 0^+} tg \left(x, z, \frac{\zeta}{t}
 \right).
\end{equation*}
It is convex and homogeneous of degree $1$ in $\zeta$.

In case that $\Omega$ is a bounded set, and under standard
continuity and coercivity assumptions,  Dal Maso proved in
\cite{Dalmaso} that ${\mathcal R}_g(u)$ is $L^1$-lower semi-continuous
for $u \in BV(\Omega)$. More recently, De Cicco, Fusco, and Verde
\cite{DCFV} have obtained a very general result about the
$L^1$-lower semi-continuity of ${\mathcal R}_g$ in $BV(\R^N)$.

Assume that $g:\R\times \R^N \to [0, \infty[$ is a Borel function
such that
\begin{equation}\label{LGRWTHnox}
C \vert \zeta \vert - D \leq g(z, \zeta)  \leq M(1+ \vert \zeta \vert)
\qquad \forall (z,\zeta)\in \R^N, \, \vert z \vert \leq R,
\end{equation}
for any $R > 0$ and for some constants  $C,D,M \geq 0$ which may
depend on $R$. Observe that both functions $f,h$ defined in
(\ref{funct:frhe}), (\ref{funct:arhe}) satisfy (\ref{LGRWTHnox}).

Assume that
$$
\1_{\{u\leq a\}} \left(g(u(x), 0) - g(a, 0)\right), \1_{\{u \geq b\}} \left(g(u(x),
0) - g(b, 0) \right) \in L^1(\R^N),
$$
for any $u\in L^1(\R^N)^+$. Let $u \in TBV_{\rm r}^+(\R^N)  \cap
L^\infty(\R^N)$  and $T = T_{a,b}-l\in {\mathcal T}^+ $. For each
$\phi\in C_c(\R^N)$, $\phi \geq 0$, we define the Radon measure
$g(u, DT(u))$ by
\begin{eqnarray}\label{FUTab}
\langle g(u, DT(u)), \phi \rangle &: =& {\mathcal R}_{\phi g}(T_{a,b}(u))+
\displaystyle\int_{\{u \leq a\}} \phi(x)
\left( g(u(x), 0) - g(a, 0)\right) \, dx  \nonumber\\
&& \displaystyle  + \int_{\{u \geq b\}} \phi(x)
\left(g(u(x), 0) - g(b, 0) \right) \, dx.
\end{eqnarray}
If $\phi\in C_c(\R^N)$, we write $\phi = \phi^+ -
\phi^-$ with $\phi^+= \max(\phi,0)$, $\phi^- = - \min(\phi,0)$,
and we define $\langle g(u, DT(u)), \phi \rangle : =
\langle g(u, DT(u)), \phi^+ \rangle- \langle g(u, DT(u)), \phi^- \rangle$.

Recall that, if $g(z,\zeta)$ is continuous in $(z,\zeta)$, convex
in $\zeta$ for any $z\in \R$, and $\phi \in C^1(\R^N)^+$ has compact
support, then  $\langle g(u, DT(u)), \phi \rangle$ is lower
semi-continuous in $TBV^+(\R^N)$ with respect to
$L^1(\R^N)$-convergence \cite{DCFV}. This property is used to prove existence of
solutions of (\ref{DirichletproblemP}).

We can now define the required functional calculus (see
\cite{ACMRelatE,ACMRelat,CER2}).

Let us denote by ${\mathcal P}$ the set of Lipschitz continuous
functions $p : [0, +\infty[ \rightarrow \R$ satisfying
$p^{\prime}(s) = 0$ for $s$ large enough. We write ${\mathcal
P}^+:= \{ p \in {\mathcal P} \ : \ p \geq 0 \}$.

Let $S \in  \mathcal{P}^+$, $T \in \mathcal{T}^+$.
We assume that
$u \in TBV_{\rm r}^+(\R^N) \cap L^\infty(\R^N)$ and note that
$$
\1_{\{u\leq a\}} S(u)\left(f(u(x), 0) - f(a, 0)\right), \1_{\{u \geq b\}} S(u)\left(f(u(x),
0) - f(b,0) \right) \in L^1(\R^N).             
$$
Since $h(z, 0) = 0$, the last assumption clearly holds also for
$h$. We define by $f_S(u,DT(u))$, $h_S(u,DT(u))$ as the Radon
measures given by (\ref{FUTab}) with $f_{S}(z,\zeta) = S(z)
f(z,\zeta)$. and  $h_{S}(z,\zeta) = S(z) h(z,\zeta)$,
respectively.

\subsection{The  notion of of entropy solution}\label{sect:defESpp}

Let $L^1_{w}(0,T,BV(\R^N))$  be the space of weakly$^*$
measurable functions $w:[0,T] \to BV(\R^N)$ (i.e., $t \in [0,T]
\to \langle w(t),\phi \rangle$ is measurable for every $\phi$ in the predual
of $BV(\R^N)$) such that $\int_0^T \Vert w(t)\Vert_{BV} \, dt< \infty$.
Observe that, since $BV(\R^N)$ has a separable predual (see
\cite{Ambrosio}), it follows easily that the map $t \in [0,T]\to
\Vert w(t) \Vert_{BV}$ is measurable. By  $L^1_{loc, w}(0, T,
BV(\R^N))$ we denote the space of weakly$^*$ measurable functions
$w:[0,T] \to BV(\R^N)$ such that the map $t \in [0,T]\to \Vert
w(t) \Vert_{BV}$ is in $L^1_{loc}(]0, T[)$.

\begin{definition} \label{def:espb}
Assume that $u_0 \in (L^1(\R^N)\cap L^\infty(\R^N))^+$. A
measurable function $u: ]0,T[\times \R^N \rightarrow \R$ is an
{\it entropy  solution}   of (\ref{DirichletproblemP}) in $Q_T =
]0,T[\times \R^N$ if $u \in C([0, T]; L^1(\R^N))$,
$T_{a,b}(u(\cdot)) - a \in L^1_{loc, w}(0, T, BV(\R^N))$ for all
$0 < a < b$, and
\begin{itemize}
\item[(i)]  $u(0) = u_0$, and \item[(ii)] \ the following
inequality is satisfied
\begin{eqnarray}\label{pei}
&& \hspace{-0.6cm}\displaystyle \int_0^T\int_{\R^N} \phi
h_{S}(u,DT(u)) \, dt + \int_0^T\int_{\R^N} \phi h_{T}(u,DS(u)) \, dt \nonumber
 \\ &&\hspace{-0.2cm}\leq  \displaystyle\int_0^T\int_{\R^N} \Big\{ J_{TS}(u(t)) \phi^{\prime}(t) - \b(u(t), \nabla u(t)) \cdot \nabla \phi \
T(u(t)) S(u(t))\Big\} dxdt,  \nonumber
\end{eqnarray}
 for truncation functions $S,  T \in \mathcal{T}^+$, and any  smooth function $\phi$ of
 compact support, in particular  those  of the form $\phi(t,x) =
 \phi_1(t)\rho(x)$, $\phi_1\in {\mathcal D}(]0,T[)$, $\rho \in
 {\mathcal D}(\R^N)$, where $J_q(r)$ denotes the primitive of $q$ for any function $q$; i.e. $\displaystyle J_q(r):=\int_0^r q(s)\,ds$
\end{itemize}
\end{definition}


\vskip 0.4 true cm \noindent {\bf Acknowledgements.} JAC
acknowledges partial support by MICINN project, reference MICINN
MTM2011-27739-C04-02, by GRC 2009 SGR 345 by the Generalitat de
Catalunya, and by the Engineering and Physical Sciences Research
Council grant number EP/K008404/1. JAC also acknowledges support
from the Royal Society through a Wolfson Research Merit Award. VC
acknowledges partial support by MICINN project, reference
MTM2009-08171, by GRC reference 2009 SGR 773 and by ''ICREA
Acad\`emia'' prize for excellence in research funded both by the
Generalitat de Catalunya. S. Moll acknowledges partial support by
MICINN project, reference MTM2012-31103.


\end{document}